\documentclass[11pt]{amsart}

\usepackage{hyperref}
\usepackage{amssymb,amsfonts}
\usepackage{amssymb,textcomp}
\usepackage{graphicx}
\usepackage{enumerate}

\usepackage[utf8]{inputenc}
\usepackage[T1]{fontenc}

\usepackage[mathscr]{eucal}

\usepackage[svgnames,x11names]{xcolor}

 \hypersetup{
   colorlinks, linkcolor={DarkSlateBlue},
   citecolor={Chocolate4}, urlcolor={DarkOliveGreen4}
  }

\usepackage[a4paper, twoside=false, vmargin={2cm,3cm}, includehead]{geometry}

\newtheorem{lemma}{Lemma}[section]
\newtheorem{theorem}[lemma]{Theorem}
\newtheorem*{theorem*}{Theorem}
\newtheorem{corollary}[lemma]{Corollary}

\newtheorem*{question*}{Question}
\newtheorem{proposition}[lemma]{Proposition}
\newtheorem*{proposition*}{Proposition}

\newtheorem{conjecture}{Conjecture}

\newtheorem*{problem*}{Problem}

\theoremstyle{definition}
\newtheorem{definition}{Definition}[section]
\newtheorem*{claim*}{Claim}

\newtheorem*{remark}{Remark}
\newtheorem*{remarks}{Remarks}

\newcommand{\lE}{\mathbb{E}^{\log}}

\newcommand{\C}{{\mathbb C}}
\newcommand{\E}{{\mathbb E}}
\newcommand{\D}{{\mathbb D}}

\newcommand{\N}{{\mathbb N}}
\renewcommand{\P}{{\mathbb P}}
\newcommand{\Q}{{\mathbb Q}}
\newcommand{\R}{{\mathbb R}}
\renewcommand{\S}{\mathbb{S}}
\newcommand{\T}{{\mathbb T}}
\newcommand{\Z}{{\mathbb Z}}
\newcommand{\U}{{\mathbb U}}

\newcommand{\CX}{{\mathcal X}}
\newcommand{\CY}{{\mathcal Y}}

\newcommand{\bN}{{\mathbf{N}}}

 \newcommand{\Var}{\operatorname{Var}}
 \newcommand{\Cov}{\operatorname{Cov}}

\newcommand{\inv}{^{-1}}

\DeclareMathOperator{\spec}{Spec}
\DeclareMathOperator{\id}{id}

\begin{document}
	
		\title[Furstenberg systems of pretentious and MRT multiplicative functions]{Furstenberg systems of pretentious and MRT multiplicative functions}

	\author{Nikos Frantzikinakis}
	\address[Nikos Frantzikinakis]{University of Crete, Department of Mathematics and Applied Mathematics, Voutes University Campus, Heraklion 71003, Greece} \email{frantzikinakis@gmail.com}

	\author{Mariusz Lema\'nczyk}
	\address[Mariusz Lema\'nczyk]{Faculty of Mathematics and Computer Science,
		Nicolaus Copernicus University, Toruń, Poland}		 \email{mlem@mat.umk.pl}
	
		\author{Thierry de la Rue}
	\address[Thierry de la Rue]{Universit\'e de Rouen Normandie, CNRS,
		Laboratoire de Math\'ematiques Rapha\"el Salem,
		Avenue de L’Universit\'e - 76801 Saint \'Etienne du Rouvray, France}		
	\email{Thierry.de-la-Rue@univ-rouen.fr}
	
\begin{abstract}
We prove structural results for measure preserving systems, called Furstenberg systems, naturally associated with bounded multiplicative functions.
We show that  for all pretentious multiplicative functions  these systems always have rational discrete spectrum and, as a consequence, zero entropy.
We obtain several other refined structural and spectral  results,
 one consequence  of which is that
 the Archimedean characters are the only pretentious multiplicative functions  that have Furstenberg systems with trivial rational spectrum,  another is that a pretentious multiplicative function has ergodic Furstenberg systems if and only if it pretends to be a Dirichlet character, and a last one is that for any fixed pretentious multiplicative function all its Furstenberg systems are isomorphic.
We also study structural properties  of Furstenberg systems of a class of multiplicative functions, introduced by Matom\"aki,  Radziwi{\l}{\l},  and  Tao, which  lie in the intermediate zone  between pretentiousness and   strong aperiodicity.  In a work of the last two authors and Gomilko, several examples of  this class with exotic ergodic behavior were identified, and here we complement this study and discover some new unexpected phenomena. Lastly, we prove that Furstenberg systems of general bounded multiplicative functions have divisible spectrum.  When these systems are obtained using logarithmic averages,  we show that  trivial rational spectrum implies a strong dilation invariance property, called
strong stationarity, but,  quite surprisingly, this property fails when the systems are obtained using Ces\`aro averages.
\end{abstract}

\thanks{The first author was supported  by the Research Grant ELIDEK HFRI-NextGenerationEU-15689.
	The second  author was  supported by Narodowe Centrum Nauki grant UMO-2019/33/B/ST1/00364.}



\date{\today}

\maketitle

\setcounter{tocdepth}{1}
\tableofcontents

	\section{Introduction}

A function $f\colon \N\to \U$, where $\U$ is the complex unit disc,  is called {\em multiplicative} if
$$
f(mn)=f(m)\cdot f(n)  \quad \text {  whenever  }  (m,n)=1.
$$
It is called {\em completely multiplicative} if the previous equation holds for all $m,n\in\N$.\footnote{A completely multiplicative function is bounded if and only if it takes values in $\U$, in our discussion, whenever we state  that a general multiplicative function is bounded we mean that it takes values in $\U$. }  In recent years,  extensive effort has been put into the study of correlations and other statistical properties of bounded multiplicative functions, mostly motivated by problems surrounding
 the  conjectures of Chowla~\cite{Ch65}, Elliott~\cite{El90,El94}, and Sarnak~\cite{S, Sa12}.
 An approach that at times  offers   advantages,  is to   associate to  the class of bounded multiplicative functions
certain   measure preserving systems, called Furstenberg systems,  that encode their   statistical behavior. One  then uses this new framework, in conjunction with  the machinery of ergodic theory, to extract interesting and often highly non-trivial conclusions.
 This has led to some important  successes, we refer the reader to the survey \cite{FKL17} for more details.


A structural result for Furstenberg systems for logarithmic averages  of general bounded multiplicative functions  was given   in \cite{FH18, FH21}. It asserts,  roughly speaking, that their ergodic components are direct products of systems with algebraic structure
(inverse limits of nilsystems) and Bernoulli systems.
Obtaining more refined structural results for special multiplicative functions, like the Liouville  or the M\"obius function, has turned out to be extremely challenging. Although the conjectures of  Chowla and Elliott  predict that  Furstenberg systems of these multiplicative functions, and more general bounded  multiplicative functions that satisfy strong aperiodicity assumptions (such as those in \cite[Equation~$(1.9)$]{MRT15}), enjoy strong randomness properties, we are still far from being able
to verify this (though see \cite{KMT23, MR15, MRT15, MRT20,MRT23,T16, Tao15,  TT18, TT19} for progress in this direction). To add to this mystery,  recent  work in \cite{GLR21}	
exhibited examples of bounded multiplicative functions that have
rather erratic and unexpected statistical behavior, ranging from very structured but non-periodic, to completely random, according to the scale of the intervals used to  define their statistics.

In this article, we plan to focus on two classes of multiplicative functions  that complement the notoriously difficult  class of strongly aperiodic ones, and  our goal is to obtain a rather complete understanding of their statistical properties by studying their Furstenberg systems. Let us briefly summarize some of our main results;  the reader will find their exact statements in Section~\ref{S:Results} and further background and explanations regarding our notation in Section~\ref{S:Background}.

 \smallskip

{\em  Pretentious multiplicative functions.}
We first focus on the class of pretentious multiplicative functions (see Definition~\ref{D:Pretentious}), which already exhibit interesting structural properties.
The simplest examples  are the Dirichlet characters, these  are the periodic completely multiplicative functions, their non-zero values are always   roots of unity and  their Furstenberg systems are rotations on finite cyclic groups. A particular example is given by the sequence
$$
\chi_{3,1}(n):={\bf 1}_{3\Z+1}(n)- {\bf 1}_{3\Z+2}(n), \quad n\in\N,
$$
with (a unique) Furstenberg system isomorphic to  a rotation on  $\Z/(3\Z)$.
Other simple examples of pretentious
 completely multiplicative functions can be obtained by  assigning the value $-1$ to finitely many  primes and $1$ to the remaining primes. This gives rise to pretentious multiplicative functions that can be approximated in density by periodic sequences, and their Furstenberg systems are ergodic procyclic systems, that is, rotations on inverse limits of cyclic groups.
 Consider for example the completely multiplicative function defined by
 $$
 f(p^k(pn+j)):=(-1)^k, \quad j=1,\ldots, p-1, \, k,n\in \Z_+,
 $$
 where $p$ is some fixed prime.
Its  (unique) Furstenberg system is  isomorphic to an ergodic  rotation on the inverse limit of the cyclic groups $\Z/(p^s\Z)$, $s\in\N$.
 Things become more complicated when we assign values different than one on infinitely many primes.  If  we define the completely multiplicative function by
 $$
 f(p):=-1,\quad \text{ for } p\in \P',\text{ such that } \sum_{p\in \P'}\frac{1}{p}<+\infty,
 $$
and   $f(p):=1$ on $\P\setminus \P'$, then $f$  turns out to   always be pretentious, but the structural properties of  its Furstenberg systems   are less clear.
And things become even more interesting when $f$ is allowed to take values on the complex unit disc $\U$. For non-zero $t\in \R$, consider for example the pretentious completely multiplicative function (often called an {\em Archimedean character}),  defined by
  $$
  f(n):=n^{it},
  \quad n\in \N.
 $$
Then $f$ cannot be approximated in density by periodic sequences, it has uncountably many Furstenberg systems for Ces\`aro averages but only one for logarithmic averages, and they are all isomorphic to the identity transformation  on the circle with the Lebesgue measure.
Another interesting example is given by defining $f$ on the primes by (for $t\in \R$, we let $e(t):=e^{2\pi i t}$)
$$
f(p):=e(1/\log\log{p}), \quad p\in \P,
$$
which leads to a completely multiplicative function that pretends to be $1$ but does not have convergent means.
We can also  get mixed behavior, with all  previous aspects present, while still maintaining our pretentiousness assumption.
For instance, this is the case when we define the completely multiplicative function on the primes by
$$
f(p):=p^{it} \cdot \chi_{3,1}(p)\cdot e(1/ \log\log{p}), \quad p\in \P.
$$
For more examples, see Section~\ref{SSS:Examples}.
For general pretentious multiplicative functions,  it is not at all clear if their Furstenberg systems  always have rational discrete spectrum or even zero entropy, and in fact, examples of multiplicative functions in the MRT class (see Definition~\ref{D:MRT})  seem to indicate otherwise. It is these and related, more refined questions, that we  answer in this article.

  In Theorems~\ref{T:StructurePretentious}, \ref{T:StructurePretentiousRefined}  we show that
  all Furstenberg systems of  pretentious multiplicative functions $f\colon \N\to \U$,   for  Ces\`aro or  logarithmic averages, have  rational discrete spectrum, and any two Furstenberg  systems of a fixed $f$ are isomorphic.\footnote{This is not a property shared by general bounded multiplicative functions. For example, every element of  the MRT class of multiplicative functions that is studied below,   has a wide variety of pairwise non-isomorphic Furstenberg systems, the structure of which differs sharply depending on whether we use Ces\`aro or     logarithmic averages. This is well illustrated in  Theorems~\ref{T:StructureMRTCesaro2}, \ref{T:StructureMRTCesaro1}, \ref{T:StructureMRTLogarithmic}.} Furthermore,  we show that these systems
   are  ergodic exactly when
the multiplicative function pretends to be a Dirichlet character, in which   case the multiplicative function has good subsequential approximations in the Besicovitch norm by periodic sequences.
Our results complement those in
\cite{BPLR19,DD82}, where rational discrete spectrum was established for a restricted class of pretentious multiplicative functions, see Theorem~\ref{T:RAP} below.    Theorem~\ref{T:spectrum}, together with Theorem~\ref{T:SpectrumDivisible}, enable us in several cases to identify the spectrum of all Furstenberg systems of  pretentious multiplicative functions, and we show that $n^{it}$ is the only one  that has trivial rational spectrum. Lastly, in Theorem~\ref{T:ExactSpectrum},  for completely multiplicative functions that pretend to be Dirichlet characters,
we describe explicitly the spectrum of their Furstenberg systems and characterize them up to isomorphism as rotations on procyclic groups. A rather immediate consequence of  the previous results is the \textit{a priori} non-obvious fact, that if a sequence satisfies the Sarnak or the Chowla-Elliott conjecture, then so does any of its multiples by a pretentious multiplicative function (see Theorems~\ref{T:Sarnak} and \ref{T:Chowla}). Some other consequences, regarding existence and vanishing of correlations of pretentious multiplicative functions, are given in Theorems~\ref{T:ceslogzero} and ~\ref{T:ceslogconverge}.

\smallskip

 {\em MRT multiplicative functions.}
 We study  a class of multiplicative functions that lie in the intermediate zone between pretentiousness and strong aperiodicity.  They were introduced in \cite[Appendix~B]{MRT15}, in order to give examples of non-pretentious multiplicative functions with non-vanishing $2$-point correlations. They are constructed by imitating the function $n^{it}$ on long intervals of consecutive primes, but using different values of $t$ 
 as the size of the interval grows (the explicit defining properties are given in Definition~\ref{D:MRT}).
 This class was studied in \cite{GLR21},
 where  Furstenberg systems with rather exotic and unexpected behavior were identified.
 In Theorem~\ref{T:aperiodic}, we show that all  MRT multiplicative functions are aperiodic  (this was known only for a certain range of parameters). In Theorems~\ref{T:StructureMRTCesaro2}-\ref{T:StructureMRTLogarithmic},   we give structural results for their Furstenberg systems that complement those in  \cite{GLR21}, covering a wider range of subsequential limits.  This enables us to show  in Theorem~\ref{T:StructureMRTCesaro1} that in the case of Ces\`aro averages,
 trivial rational spectrum   does not always imply a dilation invariance property known as  strong stationarity, contrasting a result  for logarithmic averages that holds for all bounded multiplicative functions (see Theorem~\ref{T:sst}).
 Moreover, although unipotent systems feature in our structural results  both in the  case of Ces\`aro and logarithmic averages, the exact structure  in each case is  sharply different - in the first case we  show in Theorems~\ref{T:StructureMRTCesaro2}  and \ref{T:StructureMRTCesaro1} that we get unipotent systems of fixed level, while  in the second  case  we  get mixtures of infinitely many unipotent systems with an  unbounded  number of levels, see Theorem~\ref{T:StructureMRTLogarithmic}.

\smallskip

In the course of proving the previous results, we also establish some structural properties for general bounded multiplicative functions that are of independent interest.

\smallskip

 {\em General bounded multiplicative functions.}  We study spectral properties of Furstenberg systems of general multiplicative functions with values on the complex unit disc for Ces\`aro and  logarithmic averages. In \cite{FH18}, it was shown that for logarithmic averages, the spectrum of these systems is a subset of the rationals, and here we give some additional information.  In  Theorem~\ref{T:SpectrumDivisible} we show that for  completely multiplicative functions  the  spectrum is a divisible subset of $\T$.  These properties are used in order to identify the spectrum of pretentious completely multiplicative functions. Finally, in Theorem~\ref{T:sst},  we show that if a Furstenberg system for logarithmic averages of a bounded multiplicative function has
  trivial  rational  spectrum, then the system  is necessarily  strongly stationary (see Definition~\ref{D:sst}), a property that has very strong  structural consequences, some of which are recorded in Corollary~\ref{C:sst}.  As stated before, the MRT class provides  examples where this property fails for  Furstenberg systems of  multiplicative functions  defined using   Ces\`aro averages.

\subsection{Notation} \label{SS:notation}
We let  $\N:=\{1,2,\ldots\}$,  $\Z_+:=\{0,1,2,\ldots \}$, $\R_+:=[0,+\infty)$, $\S^1$ be the unit circle, and $\U$ be  the closed complex unit disc.
With $\P$ we
denote the set of prime numbers.

With $\T$ we denote the one dimensional torus  $\R/\Z$, and we often identify it   with $[0,1)$.
We also often denote elements of $\T$ with  real numbers and we are implicitly  assuming that these real numbers are taken  modulo $1$.

For $t\in \R$, we let $e(t):=e^{2\pi i t}$, except in Sections~\ref{S:MRTCesaro} and \ref{S:MRTlogarithmic} (and the related Appendix~\ref{A:estimates}), where it is more convenient for us to let $e(t):=e^{it}$.

For $z\in \C$, with $\Re(z)$ we denote the real part of $z$.

For  $N\in\N$, we let $[N]:=\{1,\dots,N\}$, and if $M\in [1,+\infty)$, we let $[M]=\{1,\ldots, \lfloor M\rfloor\}$.

We usually denote sequences on $\N$ or on $\Z$ by  $(a(n))$, instead of $(a(n))_{n\in\N}$ or $(a(n))_{n\in \Z}$; the domain of the sequence is going to be clear from the context.
Whenever we write  $(N_k)$, we assume that $N_k$ is a strictly increasing sequence of positive integers.

 If $A$ is a finite non-empty subset of the integers and $a\colon A\to \C$, we let
 $$
 \E_{n\in A}\, a(n):=\frac{1}{|A|}\sum_{n\in A}\, a(n), \quad \lE_{n\in A}\, a(n):=\frac{1}{\sum_{n\in A}\frac{1}{n}}\sum_{n\in A}\frac{a(n)}{n}.
 $$

Given  $a,b\colon \N\to \C$, we write $a(n)\prec b(n)$  if $\lim_{n\to \infty} a(n)/b(n)=0$.

 Throughout the article, the letter $f$ is typically  used for multiplicative functions  and the letter $\chi$ for Dirichlet characters.

\subsection*{Acknowledgement}
This research was conducted while the authors were visiting the
 Institute for Advanced Studies at Princeton during parts of the academic year 2022/23. We thank the institute for its hospitality and support.

\section{Main results}\label{S:Results}
In this section, we give precise statements of our main results. To ease the exposition, we refer the  reader to  Section~\ref{S:Background} for the definitions of various  notions used in the statements.
 \subsection{Spectral results for general  completely multiplicative functions}
 We start with some results about Furstenberg systems
of general completely multiplicative functions.
 These results  are of independent interest, but they will also  be used subsequently to deduce results  for pretentious multiplicative functions and to  contrast results obtained for certain MRT multiplicative functions.

 \subsubsection{Divisibility properties of the spectrum} \label{SS:divisibility}
 In this subsection, we discuss divisibility properties of the spectrum (see Definition~\ref{D:Spectrum}) of Furstenberg systems of bounded completely multiplicative functions.  Note that in all cases, we get stronger results  when the Furstenberg systems are defined using logarithmic  averages versus Ces\`aro averages, and it is not clear if equally strong results can be obtained for Ces\`aro averages (see a related question in Section~\ref{SS:Problems}).
  \begin{theorem}\label{T:SpectrumDivisible}
 	Let  $(X,\mu, T)$ be a Furstenberg system of  a completely multiplicative function $f\colon\N \to \U$. Let also  $\alpha\in \spec(X,\mu,T)$  and $r\in \N$ such that  $f(r)\neq 0$.
 	\begin{enumerate}
 	\item\label{I:Sdiv1}(Logarithmic averages)  If the Furstenberg system is defined using logarithmic averages, then
 	$(\alpha+k)/r\in \spec(X,\mu,T)$  for some $k\in \{0,\ldots, r-1\}$.
 	
 	\smallskip
 	
 	\item\label{I:SpecCes}(Ces\`aro averages) If the Furstenberg system is defined using Ces\`aro averages, then
 	$(\alpha+k)/r\in \spec(X,\mu_r,T)$ for some $k\in \{0,\ldots, r-1\}$, where $(X,\mu_r,T)$ is another Furstenberg system of $f$ for Ces\`aro averages.
 	If $(X,\mu,T)$ is ergodic, then we can take $\mu_r=\mu$.
\end{enumerate}
 \end{theorem}
 \begin{remarks}
 $\bullet$ 	The result fails if we assume multiplicativity and not complete multiplicativity.
 Take for example $f(n):=(-1)^{n+1}$, $n\in \N$, and $\alpha=1/2$, $r=2$, or  $f:={\bf 1}_{\Z\setminus 3\Z}- {\bf 1}_{3\Z}$, and $\alpha=1/3$, $r=3$.
 Also, in the completely multiplicative case, the result fails if we allow $f(r)$ to take the value $0$, consider for example the case of Dirichlet characters.

 $\bullet$ We caution the reader that if $\alpha=1/p$, $p\in \P$,  and  $f(r)\neq 0$, then the previous result is not going to give us additional values in the spectrum unless $p$ divides $r$, and it will give us additional values if  $r=p$ (see Corollary~\ref{C:SpectrumDivisible}).

$\bullet$ If $\mu=\lim_{k\to\infty}\E_{n\in[N_k]}\, \delta_{T^nf}$, then in part \eqref{I:SpecCes}
we can take   $\mu_r:=\lim_{k\to\infty}\E_{n\in[rN_k']}\, \delta_{T^nf}$ for any subsequence  $N_k'$ of $N_k$ for which the previous weak-star limit exists, where we think of $f$ as a point in $\U^\Z$.
\end{remarks}
Theorem~\ref{T:SpectrumDivisible} is proved in Section~\ref{SS:ProofThm2.1}. The key idea is to study the action of the maps $\tau_r$, defined in \eqref{E:taur}, on eigenfunctions of the system, and show that in the case of Furstenberg systems of completely multiplicative functions, non-zero functions are mapped to non-zero functions by $\tau_r$. This is a consequence of Lemma~\ref{L:abscont}, which we combine with  Lemma~\ref{L:taurchi} in order to prove  Theorem~\ref{T:SpectrumDivisible}.

\begin{definition}
A subset $A$ of $\T$ is called {\em divisible}  if for every $\alpha\in A$ and  $r\in\N$ there exists $\alpha'\in  A$ such that $\alpha=r\alpha'$.
\end{definition}
For example,  the sets  $\{m/2^k\colon m=0,\ldots, 2^k-1, k\in \N \}$ and $\Q\cap [0,1)$ are divisible, and any non-trivial divisible subset of $\T$ has to be infinite.

The following is an immediate consequence of the previous result.
 \begin{corollary}\label{C:Divisible}
	Let   $f\colon \N\to \U\setminus \{0\}$  be  a completely multiplicative function.
	\begin{enumerate}
\item (Logarithmic averages)	The spectrum of any  Furstenberg system of $f$ for logarithmic averages  is a divisible subset of $\T$.

\smallskip

\item (Ces\`aro averages)	The combined spectrum of all Furstenberg systems of $f$ for Ces\`aro averages is a divisible subset of $\T$. The same property holds
for the spectrum of any fixed  Furstenberg system of $f$ as long as  the system is ergodic.
\end{enumerate}
\end{corollary}
\begin{remark}
If $f$ is completely multiplicative but is allowed to take the value $0$, then  the result fails, consider for example $f$ to be a non-trivial Dirichlet character. The result also fails if $f$ is not allowed to take the value $0$ but is only assumed to be multiplicative, see
the examples given in   the first remark after Theorem~\ref{T:SpectrumDivisible}.
\end{remark}
  Note that if $\alpha=0$ or $\alpha=m/n$ and $(r,n)=1$, then the content of Theorem~\ref{T:SpectrumDivisible}  is empty. On the other hand, if $\alpha=1/p$ and $r=p^s$ for some $s\in\N$,  then we do get non-trivial consequences as the next result shows.

  \begin{corollary}\label{C:SpectrumDivisible}
  	Let  $(X,\mu, T)$ be a Furstenberg system   of a completely multiplicative function $f\colon\N \to \U$ and suppose that $1/p\in \spec(X,\mu,T)$   for some $p\in \P$ with $f(p)\neq 0$.
  		\begin{enumerate}
  		\item\label{I:SpectrumDivisible1} (Logarithmic averages)  If the Furstenberg system is defined using logarithmic averages,
  	 then  $q/p^s\in \spec(X,\mu,T)$   for every $s\in \N, q\in \{0,\ldots, p^s-1\}$.
  	
  	 \smallskip
  		
  		\item\label{I:SpectrumDivisible2} (Ces\`aro averages) If the Furstenberg system is defined using Ces\`aro averages, then
    	 $q/p^s\in \spec(X,\mu_{p,s},T)$ for  every $s\in \N, q\in \{0,\ldots, p^s-1\},$
  	where $(X,\mu_{p,s},T)$ is another Furstenberg system of $f$ for Ces\`aro averages.
  		If $(X,\mu,T)$ is ergodic, then we can take $\mu_{p,s}=\mu$ for every $p\in \P$ and $s\in \N$ .
  	\end{enumerate}
  \end{corollary}
 \begin{remark}
 	As a consequence, if the rational spectrum of a Furstenberg system of a
 	completely multiplicative  	function $f\colon\N \to \{-1,1\}$ is non-trivial, then it is infinitely generated. This is not the case for general multiplicative functions  $f\colon\N \to \{-1, 1\}$,
 	see the examples in the first remark after Theorem~\ref{T:SpectrumDivisible} -  they are given by non-constant periodic sequences, hence their Furstenberg systems have non-trivial finite  spectrum (hence non-divisible).
 \end{remark}
 \begin{proof}
 	We prove the first part.
 	Let $s\in \N$. Since $f(p)\neq 0$, applying part~(\ref{I:Sdiv1}) of Theorem~\ref{T:SpectrumDivisible} for $r:=p^{s}$, we deduce that $(1+kp)/p^{s+1}\in \spec(X,\mu,T)$ for some $k\in \{0,\ldots, p^{s}-1\}$. Since $(1+kp,p^{s+1})=1$ and the spectrum is closed under multiplication by an integer, it follows that
 	$q/p^{s}$ is in the spectrum for all $q\in \{0,\ldots, p^{s}-1\}$. Since $s\in \N$ is arbitrary, we get the asserted statement.
 	
 	The second part follows by arguing as in the first part and applying part~\eqref{I:SpecCes} of Theorem~\ref{T:SpectrumDivisible}.
 \end{proof}

 Further consequences of these results for the class of pretentious multiplicative functions will be given in Theorems~\ref{T:spectrum}  and \ref{T:ExactSpectrum}   below.

 \subsubsection{Trivial rational spectrum implies strong stationarity}
It turns out that  for Furstenberg systems of bounded multiplicative functions defined using  logarithmic averages, trivial rational spectrum has very strong structural consequences. A notable one is strong stationarity,  the dilation invariance property described in  Definition~\ref{D:sst}, which  also played a crucial role in the description of Furstenberg systems of general bounded multiplicative functions in \cite{FH18}.
 \begin{theorem}\label{T:sst}
 	If  $(X,\mu,T)$ is a  Furstenberg system  for logarithmic averages	of  a multiplicative function  $f\colon \N\to \U$,  then the following two conditions are equivalent:
 	\begin{enumerate}
 		\item \label{I:sst2} The system has trivial rational spectrum.
 		
 		\smallskip
 		
 		\item \label{I:sst1} The system is strongly stationary. 	
 	\end{enumerate}
 \end{theorem}
\begin{remarks}
$\bullet$ 	The implication $\eqref{I:sst1}\implies \eqref{I:sst2}$ follows from \cite{Jen97} and holds for  general strongly stationary systems. So the more interesting implication is  $\eqref{I:sst2}\implies \eqref{I:sst1}$, and this makes use of the fact that $f$ is multiplicative.

$\bullet$ Quite surprisingly, as the examples of Theorem~\ref{T:StructureMRTCesaro1} show, the result is no longer true when the Furstenberg systems  are defined using Ces\`aro averages.
\end{remarks}
Theorem~\ref{T:sst} is proved in Section~\ref{SS:Proofsst} and uses
the  ergodic limit formulas stated in Theorem~\ref{T:F} and a recent result about correlations of multiplicative functions  of Tao-Ter\"av\"ainen~\cite{TT18} stated in Theorem~\ref{T:TT}.

 Using the previous result and known results about the structure of strongly stationary systems (see the Main Theorem in \cite{Fr04} for part (i) and \cite[part~(ii) of Proposition~3.12]{FH18} for part (ii)), we get the following structural result for Furstenberg systems of bounded multiplicative functions that have trivial rational spectrum.
 \begin{corollary}\label{C:sst}
 	Let  $f\colon \N\to \U$  be a multiplicative function and   $(X,\mu,T)$ be a Furstenberg system     for logarithmic averages
 	that has trivial rational spectrum. Then:
 	\begin{enumerate}
 		\item The system $(X,\mu,T)$  has trivial spectrum,  is strongly stationary,  and almost every ergodic component is isomorphic to a direct product of a Bernoulli system and an inverse limit of nilsystems.
 		
 		\smallskip
 		
 		\item The system $(X,\mu,T)$ is disjoint from all  ergodic systems with zero entropy.
 	\end{enumerate}
 \end{corollary}
The bulk of this  result can also be deduced from  \cite[part~(ii) of Proposition~3.12]{FH18}  and  \cite[Theorem~1.5]{FH21}.

 The Liouville function is probably the most noteworthy example of a multiplicative function for which
 it is not known whether all its Furstenberg systems have trivial rational spectrum (the Chowla conjecture predicts that this is indeed the case), i.e.  it is  not known whether any  rational in $(0,1)$ is on the  spectrum
of any of its  Furstenberg systems for Ces\`aro or logarithmic averages.

 	\subsection{Structural results for pretentious multiplicative functions }
 	Our goal in this section is to give  a detailed description of the Furstenberg systems of  pretentious multiplicative functions. See Definition~\ref{D:Pretentious} for the definition of pretentiousness and related background, and  Section~\ref{SSS:Examples} for various motivating examples.
 	\subsubsection{A known structural result}
We  start with a known result  that gives substantial  information for a rich class of
pretentious multiplicative functions.  	
 It  follows by combining \cite[Theorem~6]{DD82} and \cite[Theorem~1.7]{BPLR19}.
 	\begin{theorem}[\cite{BPLR19,DD82}]\label{T:RAP}
 		Let $f\colon \N\to \U$ be a multiplicative function and suppose
 		that there exists a Dirichlet character $\chi$ such that the series
 		\begin{equation}\label{E:fchi}
 			\sum_{p\in\P}\frac{1}{p}(1-f(p)\cdot \overline{\chi(p)}) \quad \text{converges}.
 		\end{equation}
 		Then $f$ is Besicovitch rationally almost periodic for Ces\`aro averages,
 		 it has a unique Furstenberg system with respect to Ces\`aro averages (hence also for logarithmic), and this unique system is isomorphic to  an ergodic procyclic system.
 	\end{theorem}
 \begin{remarks}
$\bullet$ 	Note that condition  \eqref{E:fchi}  is stronger than saying that $f$ pretends to be $\chi$ (see Definition~\ref{D:Pretentious}),  for $\chi=1$ see Example~\eqref{vii} in Section~\ref{SSS:Examples}.


$\bullet$ The reader will find results that extend various aspects  of Theorem~\ref{T:RAP} to all  multiplicative functions that pretend to be Dirichlet characters in   Theorem~\ref{T:StructurePretentiousRefined} and  Corollary~\ref{C:chi}.
\end{remarks} 	
  	 When $f$ is pretends to be $1$ and \eqref{E:fchi} is not satisfied (this is the case in Example~\eqref{vii} of Section~\ref{SSS:Examples}), then it follows from \cite[Corollary~2]{DD82} that the mean value of $f$ on some arithmetic progression does not exist
 	and as a consequence $f$ is not  Besicovitch rationally almost periodic.  Prior to our work, for such multiplicative functions,  it was not clear what Furstenberg systems may arise, and  it seemed plausible that they do not all have rational discrete spectrum.  We show   in Theorem~\ref{T:StructurePretentious} that this is not the case and in Theorem~\ref{T:StructurePretentiousRefined}  we give more refined information about  their structure.
 	
 \subsubsection{New structural results}\label{SSS:StructurePretentious}
 Our first main result
 applies to all pretentious multiplicative functions and shows that their Furstenberg systems have rational discrete spectrum. It extends Theorem~\ref{T:RAP}, which covers
 multiplicative functions that  satisfy property \eqref{E:fchi}.
  \begin{theorem}\label{T:StructurePretentious}
 All Furstenberg systems of pretentious  multiplicative functions for  Ces\`aro  or logarithmic averages have rational
 	discrete spectrum. As a consequence, they have  zero entropy and they do not have irrational spectrum.
 \end{theorem}
 \begin{remarks}
 	$\bullet$    The collection of Furstenberg systems of a complex valued pretentious multiplicative function 	may depend on whether we use   Ces\`aro or logarithmic averages. This is the case, for  example, when    $f(n)=n^{it}$, $t\neq 0$,  see the discussion in the Example~\eqref{v} of Section~\ref{SSS:Examples}.  It is a non-trivial fact though, that      when $f$ pretends to be a Dirichlet character, its Furstenberg systems for Ces\`aro and logarithmic averages coincide, see part~\eqref{I:StructurePretentious2.5'} of  Theorem~\ref{T:StructurePretentiousRefined} below.
 	
 	$\bullet$ Establishing that all Furstenberg systems of pretentious multiplicative functions have zero entropy is a non-trivial task on its own.
 	In fact, prior to our work it seemed plausible  that some MRT functions  (see Definition~\ref{D:MRT}) were pretentious, in which case  we would get using \cite[Corollary~2.10]{GLR21},
 	that  some pretentious multiplicative function has a Bernoulli Furstenberg system, hence of positive entropy.
 	
 	$\bullet$ It seems likely that this result can also be obtained by studying the formulas for the $2$-point correlations given by Klurman in \cite{Kl17}. We opted to take a different approach
 	in order to also get the more refined properties stated in Theorem~\ref{T:StructurePretentiousRefined}.
 	
 	$\bullet$
 	It is a rather straightforward consequence of results of Klurman in \cite{Kl17}, that if  $f\colon \N\to \U$  is a non-trivial aperiodic multiplicative function, then  some Furstenberg system of $f$ for logarithmic averages has a Lebesgue component (see Corollary~\ref{C:LebesgueComponent} in the appendix), and as a consequence, does not have rational discrete spectrum.  Hence, a multiplicative function that takes values in $\U$ is pretentious  if and only if
 	all its Furstenberg systems for logarithmic averages have rational discrete  spectrum. It follows from \cite[Theorem~1.2]{KMT23} that a similar equivalence also holds for Ces\`aro averages.
 \end{remarks}
Our main result for pretentious multiplicative functions is stated next and 
  gives more refined structural information about their Furstenberg systems. In what follows, when we write $f\sim g$ we mean $\D(f,g)<+\infty$ (see Definition~\ref{D:Pretentious}).
\begin{theorem}\label{T:StructurePretentiousRefined}
	Let $f\colon \N\to \U$ be a  pretentious multiplicative function and $(X,\mu,T)$ be a
	Furstenberg system of $f$ for Ces\`aro or logarithmic averages.
	\begin{enumerate}
		\item \label{I:StructurePretentious1}  If  $f\sim \chi$  for some primitive Dirichlet character $\chi$, then
		$(X,\mu,T)$ is an ergodic procyclic system and it is non-trivial if $f\neq 1$.
		
		\smallskip
		
		\item \label{I:StructurePretentious2}   If $f\sim  n^{it}\cdot \chi$ for some $t\neq 0$, then
		$(X,\mu,T)$
		is a  non-ergodic system with   rational discrete spectrum  and its spectrum is non-trivial unless $f(n)=n^{it}$, $n\in \N$,  for some $t\in \R$. Furthermore,
		$(X,\mu,T)$ is isomorphic to the direct product of the system $(\T,m_\T,\id)$
		 and some Furstenberg system of $\tilde{f}:=f\cdot n^{-it}\sim \chi$.
		
		\smallskip
		
		\item  \label{I:StructurePretentious2.5}  In cases \eqref{I:StructurePretentious1} and \eqref{I:StructurePretentious2}, any two Furstenberg systems of $f$ for Ces\`aro or logarithmic averages are isomorphic.
		
		\smallskip
		
		\item  \label{I:StructurePretentious2.5'}	 In case~\eqref{I:StructurePretentious1}, a Furstenberg system of $f$ for Ces\`aro averages along $(N_k)$ is well defined if and only if it is well defined for logarithmic averages, and the two Furstenberg systems are equal (i.e. the corresponding $T$-invariant measures coincide).
					
		\smallskip

		\item  \label{I:StructurePretentious3}  In case~\eqref{I:StructurePretentious1},   every sequence $N_k\to \infty$ has a subsequence $(N_k')$ along which  $f$ is Besicovitch rationally almost periodic  for Ces\`aro and logarithmic averages. In case~\eqref{I:StructurePretentious2}, there is no sequence  $N_k\to\infty$ along which  $f$ is Besicovitch rationally almost periodic  for Ces\`aro or logarithmic averages.
	\end{enumerate}
\end{theorem}
 	\begin{remarks}
 		$\bullet$ Note that a pretentious multiplicative function belongs to exactly one of the classes treated in cases \eqref{I:StructurePretentious1} and \eqref{I:StructurePretentious2}.
 		
 		
 	$\bullet$ See  Proposition~\ref{P:RAPchi'} for additional information  regarding  a variant of part~\eqref{I:StructurePretentious3} that does not require to pass to a subsequence.
 	\end{remarks}
 	Section~\ref{S:StructurePretentious} is devoted to the proof of Theorem~\ref{T:StructurePretentiousRefined} (Theorem~\ref{T:StructurePretentious} is an immediate  consequence) and  the proof  is completed in Section~\ref{SS:ProofPretentious}. The argument uses several ingredients: We start in   Section~\ref{S:PreliminariesPretentious}  with some preliminary work that leads to  the decomposition result stated in Lemma~\ref{L:12e}, which in turn  implies the subsequential Besicovitch rational almost periodicity property stated in Proposition~\ref{P:RAPchi}. This basic tool is then exploited in Section~\ref{S:StructurePretentious} and together with several arguments of ergodic flavor leads to the proof of Theorem~\ref{T:StructurePretentiousRefined}.

	
	
We deduce from Theorem~\ref{T:StructurePretentiousRefined} some equivalent characterizations of   various classes  of pretentious multiplicative functions. The first one concerns  multiplicative functions that pretend to be Dirichlet characters.
\begin{corollary}\label{C:chi}
	Let $f\colon \N\to \U$ be a pretentious multiplicative function. Then the following properties are equivalent:
	\begin{enumerate}
		\item\label{I:chi1} $f\sim \chi$ for some primitive Dirichlet character $\chi$.
		
		\smallskip
		
		\item\label{I:chi2}  Some Furstenberg system of  	$f$ for Ces\`aro  or logarithmic averages is ergodic.
		
		\smallskip
		
		\item\label{I:chi3}  All Furstenberg systems of  	$f$ for Ces\`aro  and logarithmic averages are ergodic procyclic systems.
		
		\smallskip
		
			\item\label{I:chi4}  There exists a sequence $N_k\to\infty$ along which $f$ is 	Besicovitch rationally almost periodic for Ces\`aro or logarithmic averages.
			
			\smallskip
			
				\item\label{I:chi5}  	Every sequence $N_k\to\infty$ has a subsequence along which $f$ is Besicovitch rationally almost periodic  for Ces\`aro and logarithmic averages.			
	\end{enumerate}
\end{corollary}
Corollary~\ref{C:chi} is proved in Section~\ref{SS:chi}.

 Our second corollary concerns multiplicative functions that are equal to Archimedean characters.
\begin{corollary}\label{C:nit}
	Let $f\colon \N\to \U$ be a pretentious multiplicative function. Then the following properties are equivalent:
	\begin{enumerate}
		\item\label{I:nit1}  $f(n)=n^{it}, n\in \N,$ for some $t\in \R$.
		
		\smallskip
		
		\item\label{I:nit2}  At least one Furstenberg system of  	$f$ for Ces\`aro  or logarithmic averages has trivial rational spectrum.
		
		\smallskip
		
		\item\label{I:nit3}  All Furstenberg systems of  	$f$ for Ces\`aro  or logarithmic averages have trivial rational spectrum.
		
		\smallskip
		
		\item\label{I:nit4}  All Furstenberg systems of  	$f$ for Ces\`aro  or logarithmic averages are identity systems.
		
		\smallskip
		
		\item\label{I:nit5}  	$\lim_{N\to\infty}\E_{n\in[N]}|f(n+1)-f(n)|=0$.
	\end{enumerate}
\end{corollary}
\begin{remark}
	The equivalence   of \eqref{I:nit1} and \eqref{I:nit5} was established in \cite[Theorem~1.8]{Kl17}  in a stronger form  that does not assume pretentiousness.  We give a different argument for this equivalence, with   an ergodic flavor, but it only works in the pretentious case. The equivalence  of \eqref{I:nit4} and \eqref{I:nit5} was established in  \cite[Proposition~5.1]{GLR21}. So our original contribution to this corollary is the insertion of \eqref{I:nit2} and \eqref{I:nit3} on this set of  equivalences.
\end{remark}
Corollary~\ref{C:nit} is proved in Section~\ref{SS:nit}.

\subsubsection{Spectral results}\label{SS}
We give a  result  that in conjunction with Theorem~\ref{T:SpectrumDivisible} helps us identify the spectrum of Furstenberg systems  of pretentious  multiplicative functions.
 \begin{theorem}\label{T:spectrum}
	Let  $f\colon \N\to \U$  be a multiplicative function that
	satisfies $f\sim n^{it} \cdot \chi$ for some  $t\in \R$ and  primitive Dirichlet character $\chi$ with conductor $q$, and  $(X,\mu,T)$ be a Furstenberg system of $f$ for Ces\`aro or logarithmic averages. Then for every $p\in \P$ the following properties are equivalent:
	\begin{enumerate}
		\item \label{I:spectrum1} $1/p \in \spec(X,\mu,T)$.
		
		\smallskip
		
		\item \label{I:spectrum2}  Either $p\mid q$ or $f(p^s)\neq p^{ist}\cdot \chi(p^s)$ for some $s\in\N$.
\end{enumerate}
		
\end{theorem}
 	
 \begin{remarks}
$\bullet$ If $f\sim 1$ and $s\geq 2$, then $f(p^s)\neq 1$ does not always imply  $1/p^s \in \spec(X,\mu,T)$. Take for example  $f(n):={\bf 1}_{\Z\setminus 3\Z}- {\bf 1}_{3\Z}$, $n\in\N$. Then $f(3^2)\neq 1$ and  $1/3^2\notin \spec(X,\mu,T)$.

$\bullet$ 	If   $f$  is completely multiplicative, $f\sim 1$, and  $f(p)\notin\{0,1\}$, then   combining this result with  Theorem~\ref{T:StructurePretentiousRefined} and 	Theorem~\ref{T:SpectrumDivisible},  we get   that $q/p^s\in \spec(X,\mu,T)$   for every $s\in \N$ and $q\in \{0,\ldots, p^s-1\}$.  On the other hand, if $f(p)=1$, then  $1/p\not\in\spec(X,\mu,T)$.
\end{remarks}
The implication $\eqref{I:spectrum1}\implies \eqref{I:spectrum2}$ of Theorem~\ref{T:spectrum} is proved in Section~\ref{SS:spectrumproof1}
and uses the  periodic approximation property of Proposition~\ref{P:RAPchi}, in order to deduce the result for $f$  from its periodic approximants. The implication
 $\eqref{I:spectrum2}\implies \eqref{I:spectrum1}$ is  proved in Section~\ref{SS:spectrumproof2}. It is somewhat more involved and
uses  part~\eqref{I:StructurePretentious1} of Theorem~\ref{T:StructurePretentiousRefined}   and a combination of elementary and ergodic considerations.

Our next result provides    more refined spectral information
for  multiplicative functions that pretend   to be Dirichlet characters and a complete characterization of  the spectrum of their Furstenberg systems. We will use it to justify various
claims we make in the examples given in Section~\ref{SSS:Examples}. We assume complete multiplicativity in order to have Theorem~\ref{T:SpectrumDivisible} available for us, and explain in the remarks which consequences carry over to   general multiplicative functions.
\begin{theorem}	\label{T:ExactSpectrum}
	Let 	$f\colon \N\to \U\setminus\{0\}$ be a completely multiplicative function with $f\sim \chi$ for some primitive Dirichlet character $\chi$ with conductor $q$ 
	and $(X,\mu,T)$ be a Furstenberg system for Ces\`aro or logarithmic   averages of  $f$. Then
			  $\spec(X,\mu,T)$ is equal to the  subgroup $\Lambda$  of $\T$ generated   by $\{1/p^s\colon p\in A, s\in\N \}$, where $A:=\{p\in \P\colon \text{either } p\mid q \text{ or }f(p)\neq \chi(p)\}$.
\end{theorem}
\begin{remarks}	
	$\bullet$	 Since ergodic discrete spectrum systems are isomorphic if and only if they have the same spectrum, we deduce from part~\eqref{I:StructurePretentious1}  of  Theorem~\ref{T:StructurePretentiousRefined} and the previous result, that if $f$ is as in the statement, then all Furstenberg systems  of $f$ are isomorphic to an ergodic rotation on the procyclic group that is the dual group of the subgroup $\Lambda$ defined in the statement above. 

		
	$\bullet$ 	The conclusion fails for non-completely multiplicative functions, even if they take values on $\pm 1$;  see the example in the first remark following Theorem~\ref{T:spectrum}.
	Another example is  given by the square of the M\"obius function (see Example~\eqref{iv} in Section~\ref{SSS:Examples}). The conclusion  also fails if $f$ is completely multiplicative but we allow it to take the value $0$, consider for example $f:={\bf 1}_{2\Z+1}$, which has $1/2$ but not $1/4$  on the spectrum of its Furstenberg system.
	
	$\bullet$ 	See the claim   in Section~\ref{SS:ExactSpectrum} for a variant of the  inclusion $\spec(X,\mu,T)\subset \Lambda$, which holds for all multiplicative functions that pretend to be Dirichlet characters. Also, our argument gives that even without  complete multiplicativity,
	$\spec(X,\mu,T)$ contains the subgroup $\Lambda$  of $\T$  generated by $\{1/p\colon p\in A\}$,  where $A:=\{p\in \P\colon \text{ either } p\mid q \text{ or } f(p^s)\neq \chi(p^s) \text{ for some } s\in \N\}$.
	
	 $\bullet$ Using part~\eqref{I:StructurePretentious2} of Theorem~\ref{T:StructurePretentiousRefined} and the previous result, we can explicitly identify the structure of the Furstenberg systems of any pretentious completely multiplicative function that avoids the value $0$.
	
\end{remarks}
We prove Theorem~\ref{T:ExactSpectrum} in Section~\ref{SS:ExactSpectrum}.
We essentially use the divisibility properties of the spectrum given in Section~\ref{SS:divisibility}, Proposition~\ref{P:ceslog}, which asserts that in this case Ces\`aro and logarithmic averages can be used interchangeably, and
part~\eqref{I:StructurePretentious1} of Theorem~\ref{T:StructurePretentiousRefined}, which asserts the ergodicity of the corresponding Furstenberg systems.


\subsection{Applications of the structural results}
We give some number theoretic consequences of our main results.
\subsubsection{Correlations of pretentious multiplicative functions}
We start with a vanishing property for  weighted correlations of pretentious multiplicative functions.
\begin{theorem}\label{T:ceslogzero}
	Let $f\colon \N\to \U$ be a multiplicative function with  $f\sim  \chi$ for some primitive  Dirichlet character $\chi$ with conductor $q$.
	Suppose that either
	\begin{enumerate}
		\item\label{I:alpha1}  $\alpha$ is irrational, or
		
		\smallskip
		
		\item \label{I:alpha2}  $\alpha=k/p$, $p\in \P$, $(k,p)=1$, $p\nmid q$,  and $f(p^s)=\chi(p^s)$ for every $s\in \N$.
	\end{enumerate}
	Then
	\begin{equation}\label{E:corrave0}
		\lim_{N\to\infty} \E_{n\in [N]}\, e(n\alpha)\, \prod_{j=1}^\ell f_j(n+n_j)=0
	\end{equation}
	for all  $n_1,\ldots, n_\ell \in \Z$ and $f_1,\ldots, f_\ell\in \{f,\overline{f}\}$.
\end{theorem}
\begin{remarks}
$\bullet$
 If  $f\sim n^{it}\cdot \chi$ for some $t\in \R$ and primitive Dirichlet character $\chi$, then a similar result holds as long as we change the  assumption on $f$ in  property~\eqref{I:alpha2}
 to  $f(p^s)=p^{ist}\cdot \chi(p^s)$ for every $s\in \N$.

$\bullet$ If $\alpha$ is irrational, and we replace Ces\`aro with logarithmic averages, then  \eqref{E:corrave0}  holds for all bounded multiplicative functions
$f$ (see \cite[Corollary~1.4]{FH21}). It is an open problem whether in this more general setting,  a similar property holds for Ces\`aro averages -  this  is only known for $\ell=1$ \cite{D75,DD74}.
\end{remarks}
Theorem~\ref{T:ceslogzero} is proved in  Section~\ref{SS:ceslogzero}.

Our next result establishes  existence of certain correlations of pretentious multiplicative functions.
\begin{theorem}\label{T:ceslogconverge}
	Let $f\colon \N\to \U$ be a pretentious multiplicative function.  If  $k_1,\ldots, k_\ell$ are integers such that $\sum_{j=1}^\ell k_j=0$, then the limit
	\begin{equation}\label{E:corrsum0}
		\lim_{N\to\infty} \E_{n\in [N]}\, \prod_{j=1}^\ell f^{k_j}(n+n_j)
	\end{equation}
	exists 	 	for all  $n_1,\ldots, n_\ell \in \Z$, where we use the notation
	$f^k:=\overline{f}^{|k|}$ for $k<0$.
\end{theorem}
\begin{remarks}
	$\bullet$ 	The result fails if  we do not assume that $\sum_{j=1}^\ell k_j=0$, take $\ell=1$ and $f$ any pretentious multiplicative function that does not have a mean value.  It also fails if we do not assume that $f$ is pretentious: as was shown in 	 \cite[Theorem~B.1]{MRT15} there exist aperiodic multiplicative functions $f$ (in fact all multiplicative functions  in the MRT class should work),  for which   $\limsup_{N\to\infty}|\E_{n\in[N]} \, \overline{f(n)}\cdot f(n+1)|>0$ and in  \cite{GLR21} it was shown that  $\liminf_{N\to\infty}|\E_{n\in[N]} \, \overline{f(n)}\cdot f(n+1)|=0$.

	$\bullet$ The result was known for $\ell=2$ and in this case an explicit formula for the correlations is given in \cite[Theorem~1.5]{Kl17}.
\end{remarks}
Theorem~\ref{T:ceslogconverge} is proved in  Section~\ref{SS:ceslogconverge}.

\subsubsection{Chowla and Sarnak conjecture}
To facilitate discussion, we introduce the following notions and refer  the reader to Definition~\ref{D:ergodicnotions} for the notion of  completely deterministic sequences.
\begin{definition}
	 We say that a sequence   $a\colon \N\to \U$:
	\begin{enumerate}
		\item {\em Satisfies the Sarnak conjecture for Ces\`aro averages along $(N_k)$}, if
		$$
		\lim_{k\to\infty}\E_{n\in[N_k]}\, a(n)\, w(n)=0
		$$
		for all  sequences $w\colon \N\to \U$ that are   completely deterministic  along $(N_k)$.\footnote{See \cite{AKLR14} for the relation  of this statement with the original conjecture of Sarnak~\cite{S}.}

		\smallskip
		
		\item  {\em  Satisfies the Chowla-Elliott  conjecture for Ces\`aro averages along $(N_k)$}, if
		\begin{equation}\label{E:CE}
		\lim_{k\to \infty} \E_{n\in[N_k]}\,  a^{\epsilon_1}(n+n_1)\cdots a^{\epsilon_\ell}(n+n_\ell)=0
		\end{equation}
		for all $\ell\in \N$, distinct $n_1,\ldots, n_\ell\in \Z_+$, and all $\epsilon_1,\ldots, \epsilon_\ell\in \{-1, 1\}$,
		where we set  $a^{-1}:=\overline{a}$.
		\end{enumerate}
	Similar definitions apply for logarithmic averages.
\end{definition}
We give two results, which  roughly speaking assert that if a bounded sequence satisfies the Sarnak or the Chowla-Elliot conjecture, then so does any multiple of the sequence
  by a pretentious multiplicative function.
We caution the reader that although these claims are rather easy to prove when the multiplicative function satisfies Theorem~\ref{T:RAP},
 a priori,
it is not   even clear  that such  claims are expected to hold for general pretentious multiplicative functions.  In fact, since all MRT functions satisfy the
Sarnak and the Chowla-Elliott conjecture along some subsequence $(N_k)$ (this is a consequence of the main result in \cite{GLR21}), one has to first nullify the possibility that the MRT class contains a pretentious multiplicative function, which is a non-trivial task on its own.



\begin{theorem}\label{T:Sarnak}
Let  $a\colon \N\to \U$ be a   sequence that satisfies the Sarnak conjecture for Ces\`aro  averages along $(N_k)$, and     $b:=a\cdot f$, where $f\colon \N\to\U$ is a pretentious multiplicative function. Then  $b$ also satisfies the Sarnak conjecture  for  Ces\`aro  along $(N_k)$. A similar fact also holds for logarithmic averages.
\end{theorem}
  We prove Theorem~\ref{T:Sarnak} in Section~\ref{SS:Sarnak} by making essential use of Theorem~\ref{T:StructurePretentious}.

\begin{theorem}\label{T:Chowla}
Let $a\colon \N\to \{-1, 1\}$ be a sequence  that satisfies the  Chowla-Elliott conjecture  for Ces\`aro    averages along $(N_k)$,   and $b:=a\cdot f$, where $f\colon \N\to \U$ is a pretentious multiplicative function. Then  $b$ also satisfies the Chowla-Elliott conjecture  for  Ces\`aro   averages along $(N_k)$. A similar fact also holds for logarithmic averages.
\end{theorem}
We prove Theorem~\ref{T:Chowla} in Section~\ref{SS:Chowla} by making essential use of Theorem~\ref{T:StructurePretentious}.


\subsection{Structural results about MRT multiplicative functions}\label{SS:MRTresults}
We focus here  on  a class of multiplicative functions  defined in \cite[Appendix~B]{MRT15}, and which were used to show that  the 2-point Elliott conjecture fails for some aperiodic  multiplicative functions. It turns out that this class provides a rich playground on which one can construct examples of multiplicative functions with rather exotic statistical behavior, resulting in Furstenberg systems with unexpected structural properties, not to be found within the class of pretentious or strongly aperiodic multiplicative functions (such as those satisfying \cite[Equation~(1.9)]{MRT15}).
Our goal here  is to give an explicit description of their Furstenberg systems for Ces\`aro and logarithmic averages
for a wider range of parameters than those dealt in \cite{GLR21}.  As a consequence, we get further unanticipated  properties for such Furstenberg systems, only to be found on this class of multiplicative functions  with ``intermediate randomness'' properties.

\subsubsection{The setting}\label{SSS:MRT}
We reproduce here the definition of the  MRT class of completely  multiplicative functions from
\cite[Definition~3.1]{GLR21}.
\begin{definition}\label{D:MRT}
We say that the completely  multiplicative function $f\colon \N\to \S^1$ belongs to the {\em MRT class},
or is an {\em MRT (multiplicative) function}  if the following property is satisfied: There exist strictly  increasing sequences of integers $(t_m)$, $t_1:=1$,  and $(s_m)$ such that, for each $m\in \N$, the following holds:  	

\begin{enumerate}
	\item \label{I:i} $t_m<s_{m+1}<s_{m+1}^2\leq t_{m+1}$,
	
	\medskip
	
	\item \label{I:ii} $f(p)=p^{is_{m+1}}$ for each prime $p\in(t_m,t_{m+1}]$, 	
	
	\medskip
	
	\item \label{I:iii} $|f(p)-p^{is_{m+1}}|\leq 1/t_m^2$ for each prime $1<p\leq t_m$.
\end{enumerate}
\end{definition}
It is shown in \cite[Appendix~B]{MRT15} that the MRT class is non-empty, and if $s_{m+1}>e^{t_m}$, then  $f$ is an aperiodic multiplicative function. But this is not a growth assumption that we impose in our defining axioms, nor is it known whether it follows from these axioms.
On the other hand, it is  shown in \cite{GLR21} that the axioms \eqref{I:i}-\eqref{I:iii} imply that,  for every $K\in \N$, we have $s_{m+1}/t^K_m\to\infty$.

It is a consequence of \cite{GLR21} that every element  of the MRT class has uncountably many different Furstenberg systems, and depending on the choice of our averaging intervals $[N_m]$, we may get substantially different structural properties, ranging from the identity system to  the Bernoulli system. Our goal here is to give
a complete classification of these Furstenberg systems for Ces\`aro and logarithmic averages, when roughly speaking, $N_m$ grows as a fractional power of $s_{m+1}$ (i.e. $\lim_{m\to\infty} N_m/s_{m+1}^\beta= c$ for some $\beta,c>0$).

{\em We stress that in the structural results that we give below, we fix an arbitrary MRT function $f$ and describe the structure of different  Furstenberg systems of this fixed $f$ for Ces\`aro or logarithmic averages,  depending  on the choice of the sequence $N_m\to \infty$ that defines the Furstenberg system.}

The following is a first non-trivial    consequence of our results.
\begin{theorem}\label{T:aperiodic}
	All MRT multiplicative functions are aperiodic (non-pretentious).
\end{theorem}
 This  follows at once from   Theorem~\ref{T:StructureMRTCesaro2} (or the structural result obtained in \cite{GLR21}) and Theorem~\ref{T:StructurePretentious},  since for every MRT function not all of its Furstenberg systems for Ces\`aro averages have rational discrete spectrum.

	





\subsubsection{Furstenberg systems of MRT functions for Ces\`aro averages}
We first state structural results for  Furstenberg systems of MRT functions defined using Ces\`aro averages. In the next subsection, we cover  the case of logarithmic averages, which lead to systems with substantially different structural properties.

 \begin{definition}\label{D:Sk}
	For  $d\in \Z_+$, we consider the measure preserving system
	$
	(Y_d,\nu_d,S_d),
	$
	where  $Y_d:=\T^{d+1}$, $\nu_d:=m_{\T^{d+1}}$,  and  $S_d\colon \T^{d+1}\to \T^{d+1}$ is the transformation defined by
	$$
	S_d(x_0,\ldots, x_d):=(x_0,x_1+x_0,\ldots, x_d+x_{d-1}), \quad x_0,\ldots, x_d\in \T.
	$$
	We call it the {\em level $d$ unipotent system}.
	\end{definition}
For $d=0$, we have $S_0(x_0)=x_0$ on $\T$ with $m_\T$. For $d=1$, we have
$S_1(x_0,x_1)=(x_0,x_1+x_0)$ on $\T^2$ with $m_{\T^2}$, and so on.
We introduce these systems in order to reproduce in an ergodic setting the correlations of MRT functions given in Proposition~\ref{P:corrlog2}.

\begin{definition}
Given  $a,b\colon \N\to \C$, we write $a(n)\prec b(n)$  if  $\lim_{n\to \infty} a(n)/b(n)=0$.
\end{definition}

We start with a result previously obtained in \cite{GLR21} when $N_m=\lfloor s_{m+1}^{1/c}\rfloor$ and $c>0$ is not an integer (the case $c\in \N$ is covered by Theorem~\ref{T:StructureMRTCesaro1}). We use somewhat different techniques to cover the more general case below. 
 \begin{theorem}\label{T:StructureMRTCesaro2}
	For $d\in \Z_+$, let  $(X,\mu_d,T)$ be the  Furstenberg system  for 
	Ces\`aro averages of the MRT function  $f\colon \N\to \S^1$, taken  along the sequence of intervals $[N_m]$ that satisfies $s_{m+1}^{1/(d+1)}\prec N_m\prec s_{m+1}^{1/d}$
	(for $d=0$ we simply assume that $s_{m+1}\prec N_m\leq t_{m+1}$). Then $(X,\mu_d,T)$  has trivial spectrum, is  strongly stationary,  and  is  isomorphic to the level $d$ unipotent system $(\T^{d+1},m_{\T^{d+1}},S_{d})$ given in   Definition~\ref{D:Sk}.
\end{theorem}
Theorem~\ref{T:StructureMRTCesaro2} is proved in Section~\ref{SS:MRTCesaroProofs}.

Our method of proof is somewhat different than the one used in  \cite{GLR21}, the additional flexibility of our method enables us to  give more refined results and also handle subsequent problems concerning  logarithmic averages. We continue with a case that exhibits different structural behavior and gives examples of non strongly stationary systems, a feature not present in the case of logarithmic averages.
\begin{definition}\label{D:Sad}
	For  $\alpha>0$ and $d\in \Z_+$,  let $S_{\alpha,d}\colon \T^{d+1}\to \T^{d+1}$ be the transformation defined by
$$
S_{\alpha,d}(x_0,\ldots, x_d):=(x_0,x_1+g_{\alpha,d}(x_0),x_2+x_1,\ldots, x_d+x_{d-1}),  \quad x_0,\ldots, x_d\in \T,
$$
where $g_{\alpha,d}\colon \T\to \R_+$ is defined by $g_{\alpha,d}(x):=1 /(\alpha^d\{x\}^d)$ for
$x\neq 0$.
\end{definition}
This system helps us reproduce the correlations of MRT functions given in Proposition~\ref{P:corrlog}
and leads to the following result.
\begin{theorem}\label{T:StructureMRTCesaro1}
For $\alpha>0$ and $d\in \Z_+$, 	let $(X,\mu_{\alpha,d},T)$ be the  Furstenberg system for Ces\`aro averages of the MRT function  $f\colon \N\to \S^1$, taken  along the sequence of intervals $[N_m]$ where  $N_m:=\lfloor\alpha s_{m+1}^{1/d}\rfloor$. Then $(X,\mu_{\alpha,d},T)$  has trivial spectrum, it  is not always strongly stationary (in the sense described in Definition~\ref{D:sst}),  and it is  isomorphic to the system  $(\T^{d+1},m_{\T^{d+1}},S_{\alpha,d})$, where $S_{\alpha,d}$ is taken as in Definition~\ref{D:Sad}.
\end{theorem}
\begin{remarks}
	$\bullet$ 		Note that the previous result contrasts Theorem~\ref{T:sst}, which shows that in the case of logarithmic averages, all Furstenberg systems of bounded multiplicative functions with trivial rational spectrum are strongly stationary.
	
	$\bullet$ 		If  we consider logarithmic averages, we will see in Theorem~\ref{T:StructureMRTLogarithmic} that the corresponding system, when
	$N_m:=\lfloor s_{m+1}^{1/d}\rfloor$  is not isomorphic to the level $d$ unipotent system given in Definition~\ref{D:Sk} but to   an infinite ``mixture'' of such systems.
	
	$\bullet$  For $d=1$, taking the limit as $\alpha \to \infty$, we get that the measure
	$\mu_{\alpha,1}$ on the sequence space converges  weak-star to a measure that induces a system isomorphic to the identity transformation on $\T$ with $m_\T$. Taking the limit as $\alpha \to 0^+$, we get a system isomorphic to the  one defined by $T(x,y)=(x,y+x)$  on $\T^2$ with $m_{\T^2}$.   A similar thing happens for general $d\in \N$,
	as $\alpha\to 0+$ or $\alpha\to \infty$, we get respectively  the level $d$ and level $d-1$  unipotent systems given in Definition~\ref{D:Sk}.
In particular,  contrasting \cite{GLR21}, where a countable family of Furstenberg systems was described, here, we give an uncountable family of explicit Furstenberg systems of $f$ that illustrates the ``continuous'' transition from the level $d-1$ to the level $d$ unipotent system according to the choice of $N_m$.	
\end{remarks}
Theorem~\ref{T:StructureMRTCesaro1} is proved in Section~\ref{SS:MRTCesaroProofs}.

\subsubsection{Furstenberg systems of MRT functions for logarithmic averages}
We describe structural properties of  Furstenberg systems of MRT functions defined using logarithmic averages. The reader is advised to compare   Theorems~\ref{T:StructureMRTCesaro2} and \ref{T:StructureMRTCesaro1} with  Theorem~\ref{T:StructureMRTLogarithmic}, and note that,  for the same choice of sequence of averaging intervals $([N_m])$, the  structure of the corresponding Furstenberg systems differs sharply when we use   Ces\`aro versus logarithmic averages.
  \begin{definition}\label{D:muc}
	For $d\in \Z_+$, let $(Y_d,\nu_d,S_d)$
be the level $d$ unipotent system given in  Definition~\ref{D:Sk}.	
	For $c>0$, we  define the system
	$(Z_c,\nu'_c, R_c)$, where $Z_c$ consists of disjoint copies of $Y_d$, $d\geq \lfloor c\rfloor$,
	with the corresponding $\sigma$-algebra, $R_c$ is defined to be $S_d$
	on each piece $Y_d$,
	and
	$$
	\nu'_c:=\Big(1-\frac{c}{\lceil c\rceil}\Big) \nu_{\lfloor c\rfloor }  + c\sum_{d=\lceil c\rceil}^\infty
	\Big(\frac{1}{d}-\frac{1}{d+1}\Big)\, \nu_d.
	$$
\end{definition}
The motivation for introducing this rather awkward system, is that it helps us reproduce the correlations of MRT functions for logarithmic averages  given in Proposition~\ref{P:corrlogc} and leads to the following result.
	 \begin{theorem}\label{T:StructureMRTLogarithmic}
	For $c>0$, let $(X,\mu_c,T)$ be the  Furstenberg system for logarithmic averages of the MRT function  $f\colon \N\to \S^1$, taken  along the sequence of intervals $[s_{m+1}^{1/c}]$. Then $(X,\mu_c,T)$  has trivial spectrum, is strongly stationary,  and is  isomorphic to    the system  $(Z_c,\nu'_c ,R_c)$ given in  Definition~\ref{D:muc}.
\end{theorem}
\begin{remark}
	A similar statement holds if instead of  $[s_{m+1}^{1/c}]$ we use  any other sequence $[L_{s_{m+1}}]$, where $(L_N)$ is a sequence with fractional degree  $1/c$
	(see Definition~\ref{D:fracdeg}). We actually prove the result in this more general setting.
\end{remark}
Theorem~\ref{T:StructureMRTLogarithmic} is proved in Section~\ref{S:MRTlogarithmic}.

\subsection{Open problems} \label{SS:Problems}
We refer the reader to Definition~\ref{D:ergodicnotions} for the various ergodic notions associated with bounded sequences throughout this subsection.  We also remark that although the conjectures  below are stated for Ces\`aro averages, similar conjectures can be made for logarithmic averages and we find them  equally interesting.

In Theorem~\ref{T:StructurePretentious} we established that all pretentious multiplicative functions
are completely deterministic for Ces\`aro and logarithmic averages.
We believe that this  zero-entropy property  characterizes pretentiousness within the class of all bounded multiplicative functions.
\begin{conjecture}\label{Con:1}
A non-trivial multiplicative function $f\colon \N\to \U$ is pretentious if and only if it is completely deterministic for Ces\`aro averages.\footnote{This is not to be confused with the notion of zero topological entropy, which  is stronger. See Example~\eqref{iv} in Section~\ref{SSS:Examples}.}
\end{conjecture}
\begin{remark}
	It is an immediate consequence of  \cite[Main Theorem]{GLR21}, that all multiplicative functions in the MRT class are not
	completely deterministic for Ces\`aro averages or logarithmic averages.
\end{remark}
The structural results  in Section~\ref{SS:MRTresults} suggest  that all Furstenberg systems of  MRT
functions for Ces\`aro and logarithmic averages  are disjoint from all ergodic zero entropy systems.
A consequence of this would be that all MRT functions satisfy the  Sarnak conjecture  with ergodic deterministic weights.   Since the MRT class seems to be  the most likely place  to look for examples, where the previous property fails, we expect the following to hold.
\begin{conjecture}\label{Con:2}
Let $f\colon \N\to \U$  be an aperiodic multiplicative function.
Then
$$
\lim_{N\to\infty} \E_{n\in [N]} \, f(n)\, w(n)=0
$$
for every  $w\colon \N\to \U$  that is completely deterministic and ergodic for Ces\`aro averages.

\end{conjecture}
\begin{remark}
One could replace $w(n)$ with $(F(T^nx))$, for all uniquely ergodic systems  $(X,T)$ with zero entropy,  continuous functions  $F\in C(X)$, and points $x\in X$.
\end{remark}
Note that when we use logarithmic averages, the previous property is known to hold under a  strong aperiodicity assumption on the multiplicative function $f$  \cite[Theorem~1.5]{FH21}. But even in the case of logarithmic averages, the previous conjecture   enlarges the scope of this result to all  aperiodic  multiplicative functions. We also remark that
without the ergodicity assumption on the weight, the conjecture is false as is shown in  \cite[Section~5.2]{GLR21}. Lastly, for all totally ergodic weights and logarithmic averages, the conjecture is known  for every zero mean multiplicative function, this follows from the disjointness of Furstenberg systems of bounded multiplicative functions and totally ergodic systems, which is
a consequence of \cite[Proposition~3.12]{FH18} and \cite[Theorem~1.5]{FH21}.

We also state a conjecture regarding a variant of
Sarnak's conjecture on which  we do not impose any ergodicity assumptions on the weight $w$.\footnote{A similar conjecture was made independently in \cite[Conjecture~2.11]{KMT23}.}
\begin{conjecture}\label{Con:3}
Let $f\colon \N\to \U$ be an aperiodic multiplicative function. Then there exists a subsequence $N_k\to\infty$ such that
	$$
	\lim_{k\to\infty} \E_{n\in [N_k]} \, f(n)\, w(n)=0
	$$
	for every
	 $w\colon \N\to \U$ that is completely deterministic for Ces\`aro averages along  $(N_k)$.
\end{conjecture}
As remarked above, if we do not take subsequential limits,  then  the asserted convergence to $0$ fails  for some aperiodic multiplicative function $f$ \cite[Section~5.2]{GLR21}.
 If   Conjecture~\ref{Con:3} is verified
  for logarithmic averages along $(N_k)$ for some  specific
 aperiodic multiplicative function $f$,
then a modification of an  argument in~\cite{Tao17} would give that $f$ satisfies the Chowla-Elliott  conjecture for logarithmic averages  along $(N_k)$.\footnote{One has to verify that the implications Conjecture 1.5 $\implies$  Conjecture 1.6 $\implies$ Conjecture 1.4 in \cite{Tao17} hold with an arbitrary bounded multiplicative function $f$ in place of $\lambda$. The first implication works for arbitrary bounded sequences, so in particular for $f$. For the second implication, one  needs to do some minor adjustments as in the proof of \cite[Theorem~1.7]{F17}. Also, both implications work without any change when the averages $\lE_{n\in[N]}$ are replaced by the subsequential averages $\lE_{n\in[N_k]}$.}

We remark that Theorem~\ref{T:SarnakErgodic}  in the appendix offers some support for the last two conjectures,  at least for logarithmic averages.


Although in Corollary~\ref{C:Divisible} we have established divisibility of the spectrum when a Furstenberg system of a non-zero completely multiplicative function is constructed using logarithmic averages, the corresponding question for Ces\`aro averages remains open, and it is not clear to us which way the answer should go.
\begin{question*}\label{Q:SpecCesaro}
Let $(X,\mu,T)$ be a Furstenberg system for Ces\`aro averages of a completely multiplicative function $f\colon \N\to \U\setminus\{0\}$. Is it true that the spectrum of the system is always a divisible subset of $\T$?
\end{question*}
We caution the reader  that Corollary~\ref{C:Divisible} covers the case of the combined spectrum of all Furstenberg systems of $f$ for Ces\`aro averages, and not the spectrum of any fixed Furstenberg system of $f$ (unless it is ergodic). For pretentious multiplicative functions the answer is yes, since as we show in part~\eqref{I:StructurePretentious2.5} of Theorem~\ref{T:StructurePretentiousRefined} that  any two Furstenberg systems of $f$ for Ces\`aro or logarithmic averages coincide, so we can use Corollary~\ref{C:Divisible}.

\section{Background}\label{S:Background}
In this section we gather some notions and basic facts from ergodic theory and number theory used throughout this article.

\subsection{Dirichlet characters}
 A {\em Dirichlet character} $\chi$  is a periodic completely multiplicative function, and is often thought of as a multiplicative function in $\Z_m$ for some $m\in \N$. In this case, 
$\chi$  takes the value $0$ on integers that are not coprime to $m$, and
takes values on $\phi(m)$-roots of unity on all other integers, where $\phi$ is the Euler totient function.  The indicator function of the integers that are relatively prime to $m$ is called the {\em principal character modulo $m$} and is denoted by $\chi_{0,m}$.  A Dirichlet character $\chi$ mod $m$ may be determined by a character $\chi'$ of strictly smaller modulus $m'\mid m$ by the formula $\chi=\chi'\cdot \chi_{0,m}$. If this is not the case, we say that $\chi$ is {\em primitive} and its period $m$ is then denoted by $q$ and is called the {\em conductor of $\chi$}.  Note that $\chi=1$ is the only primitive Dirichlet character with conductor $q=1$ and also the only Dirichlet character that is primitive and principal. We remark also that for every Dirichlet character $\chi$ mod $m$, there exists a primitive Dirichlet character $\chi'$ such that $\chi=\chi'\cdot \chi_{0,m}$.

\subsection{Pretentious multiplicative functions}\label{SS:pretentious}
Following Granville and Soundararajan~\cite{GS07,GS08,GS23}, we define the notion of pretentiousness
and a related distance between multiplicative functions.
\begin{definition}\label{D:Pretentious}
If  $f,g\colon \N\to \U$ are multiplicative functions, we define the distance between them as
\begin{equation}\label{E:D}
	\D^2(f,g):=\sum_{p\in \P} \frac{1}{p} \big(1-\Re(f(p)\cdot \overline{g(p)})\big).
\end{equation}
We say that:
\begin{enumerate}
\item  {\em $f$ pretends to be $g$} and write $f\sim g$  if  $\D(f,g)<+\infty$.

\smallskip

\item \label{I:pretentious} $f$ is {\em pretentious} if  $f\sim n^{it}\cdot \chi$ for some $t\in \R$ and primitive  Dirichlet character $\chi$.

\smallskip

\item $f$ is  {\em aperiodic} (or {\em non-pretentious})  if  it is not pretentious.

\smallskip

\item $f$ is {\em trivial} if  $\lim_{N\to\infty}\E_{n\in [N]}\, |f(n)|^2=0$ (in which case $f$ is aperiodic).
\end{enumerate}
\end{definition}
In case \eqref{I:pretentious}, the real number $t$  and the primitive Dirichlet character $\chi$  (and hence its conductor $q$) are uniquely determined, meaning, if  $n^{it}\cdot \chi\sim n^{it'}\cdot \chi'$ for some $t,t'\in \R$ and primitive Dirichlet characters  $\chi,\chi'$, then $t=t'$ and $\chi=\chi'$ (see for example \cite[Proposition~6.2.3]{M18}).
 It is also known that, for every non-zero $t\in \R$, we have  $n^{it}\not\sim \chi$ for every Dirichlet character $\chi$ (see for example \cite[Corollary~11.4]{GS23}). 
 Note that for every Dirichlet character $\chi$ there exists a primitive Dirichlet character $\chi'$ such that $\chi\sim \chi'$. So without loss of generality, in our statements we use primitive Dirichlet characters, which in 
 some cases offers notational advantages. 
 
It follows from \cite[Corollary~1]{DD82}  that a  multiplicative function $f\colon \N\to \U$ is aperiodic if and only if
its Ces\`aro averages on every infinite arithmetic progression are zero, that is,
$\lim_{N\to\infty}\E_{n\in[N]}\, f(an+b)=0$ for every $a\in \N, b\in \Z_+$.

It can be shown (see \cite{GS08} or \cite[Section~4.1]{GS23}) that $\D$ satisfies the triangle inequality
$$
\D(f, g) \leq \D(f, h) + \D(h, g)
$$
for all  $f,g,h\colon \N\to \U$.
If  $f$ takes values on the unit circle, then we always have $f\sim f$. In general, it may be that $f\not\sim f$, which   happens if  and only if  $\lim_{N\to\infty}\E_{n\in[N]}\, |f(n)|^2=0$, i.e. if  $f$ is trivial (for a proof, see for example \cite[Lemma~2.9]{BPLR20}).

Also, for all  $f_1, f_2, g_1, g_2\colon \N\to \U$,  we have (see
\cite[Lemma~3.1]{GS07})
$$
 \D(f_1f_2, g_1g_2) \leq \D(f_1, g_1) + \D(f_2, g_2).
$$
It follows that if $f_1\sim g_1$ and $f_2\sim g_2$, then $f_1f_2\sim g_1g_2$.


	\subsubsection{Some examples} \label{SSS:Examples}
In order to get a sense of the variety of  measure preserving systems that arise as Furstenberg systems of pretentious multiplicative functions and their spectral properties, we provide a rather extensive assortment of  examples.
The terminology we use is explained in Sections~\ref{SS:mps} and \ref{SS:Defs}.  In all cases, the Furstenberg systems can be taken either  for  Ces\`aro or logarithmic averages.
The properties
recorded in Examples  \eqref{i}, \eqref{ii}, \eqref{iii'}, \eqref{vi}, \eqref{vii} follow from Theorem~\ref{T:ExactSpectrum}, the properties of Example~\eqref{iii} are straightforward,   the properties of Example~\eqref{iv} were established  in  \cite{CS13} (see also \cite[Theorem~9]{S}), and
those  of Example~\eqref{v} can be found in \cite[Corollary~5.5]{GLR21} and  \cite[Section~1.3]{FH21}.
Finally, we remark that Theorem~\ref{T:RAP} does not apply to  Example~\eqref{vii}, so in this case even showing that all Furstenberg systems have rational discrete spectrum is a non-trivial task. For the sake of brevity, in the following discussion when we refer to the spectrum of a multiplicative function we mean the spectrum of any of its Furstenberg systems.
\begin{enumerate}
	\item \label{i}  $f(2):=-1$ and $f(p):=1$ for $p\neq 2$,  and $f$ is completely multiplicative. The spectrum of $f$ is $m/2^s$, $m\in\Z_+,s\in\N$,  and it has a unique Furstenberg system, which is an ergodic procyclic system.
	
	\smallskip
	
	\item \label{ii} More generally, let $f(p_1)=\cdots =f(p_\ell):=-1$ and $f(p)=1$ for $p\notin \{p_1,\ldots, p_\ell\}$,  and $f$ is completely multiplicative. Then the  spectrum of $f$ consists of all integer combinations of the numbers $1/p_i^k$, where $k\in \N$, $ i\in \{1,\ldots, \ell\}$,  and it has a unique Furstenberg system, which  is an ergodic procyclic system.
	
	\smallskip
	
	\item\label{iii} $f:=\chi$ is a  Dirichlet character, or $f(n)=(-1)^{n+1}$,   or $f:={\bf 1}_{\Z\setminus 3\Z}- {\bf 1}_{3\Z}$. Then $f$ has finite rational spectrum and it has a unique  Furstenberg system, which  is an ergodic cyclic system.

	\smallskip
	
	\item\label{iii'} $f:=\tilde{\chi}$ is a modified  Dirichlet character, defined as in Lemma~\ref{L:chimodified}, where $\chi$ is a primitive Dirichlet character with conductor $q$.
	Then $f$ has a unique Furstenberg system, which is ergodic and procyclic, and its spectrum consists of all integer combinations of $1/p^k$, $k\in\N$, where $p\in \P$  divides  $q$.

	\smallskip
	
	\item \label{iv} $f:=\mu^2$, is the indicator of the square free numbers. Then $f$ has a unique Furstenberg system, which is ergodic and procyclic, and its spectrum consists of all integer combinations of $1/p^2$, where $p\in \P$. In particular, for any prime $p$, we have that $1/p^2$ is on the spectrum of this system but $1/p^3$ is not. We
	also remark that this is an example of a multiplicative function that has positive topological entropy but its unique Furstenberg system has zero entropy.
	
	\smallskip
	
	\item \label{v} $f(n):=n^{it}$ for some $t\neq 0$. Then $f$ has uncountably many   Furstenberg systems, 	 all isomorphic to the  identity transformation  on $\T$ with the Lebesgue measure. In the case of logarithmic averages, it has a unique Furstenberg system, isomorphic to the identity transformation on $\T$ with the Lebesgue measure. In both cases, all Furstenberg systems have trivial rational spectrum.
	
    \smallskip
	
	\item \label{vi} $f(p):=e(1/p)$, or $f(p):=1-1/p$, $p\in \P$, and $f$ is completely multiplicative. Then $f\sim 1$ and $f$  has a unique ergodic procyclic Furstenberg system and its spectrum consists of all rational numbers in $[0,1)$. If  we change the value of $f$ at a single prime to $1$, say set $f(3):=1$, then $1/3$ will no longer be in the spectrum, and the non-zero values in the spectrum  will now consist of all rationals  $p/q\in [0,1)$ with $(p,q)=1$ and $(q,3)=1$.
	
	\smallskip
	
	\item \label{vii}  $f(p):=e(1/ \log\log{p})$, $p\in \P$, and $f$ is completely multiplicative. Then $f\sim 1$ but $f$ does not satisfy \eqref{E:fchi}, so  Theorem~\ref{T:RAP} does not apply.\footnote{The reason for this is that if $\theta_p:=1/\log\log{p}$, $p\in \P$, then $\sum_{p\in\P}\frac{\theta_p^2}p<+\infty$, while $\sum_{p\in\P}\frac{\theta_p}p=+\infty$.}
	  We have  that $f$ does not have a mean value \cite[Corollary~2]{DD82}  and hence has several Furstenberg systems,   all of which are ergodic procyclic systems, isomorphic to each other, and their spectrum consists of all the rational numbers in $[0,1)$.
\end{enumerate}

It follows from  \cite[Theorem~1.5]{FH21}  that an irrational number cannot be on the spectrum of a  Furstenberg system for logarithmic averages of any multiplicative function (pretentious or not) that takes values in $\U$. For  pretentious multiplicative functions, an alternative proof follows from
Theorem~\ref{T:ceslogzero}, and it  also applies to Furstenberg systems for Ces\`aro averages (a case not covered in \cite{FH21}).

\subsubsection{Mean values of multiplicative functions}
We will need the following notion.
 \begin{definition}\label{D:slowgrow}
	We say that a sequence $A\colon \N\to \R_+$ is {\em slowly varying} if for every $c\in (0,1)$, we have
	$$
	\lim_{N\to\infty}  \sup_{n\in [N^c,N]}|A(n)-A(N)|=0.
	$$
\end{definition}
For example, the sequence $A(n)=\log\log\log{n}$, defined for $n\geq 3$, is slowly varying,
but the sequence $A(n)=\log\log{n}$, defined for $n\geq 2$, is not slowly varying.

We will use the following result that can be found in the form stated below in \cite[Theorems~6.2 and 6.3]{E79}. We will only apply this in the case of pretentious multiplicative functions.
\begin{theorem}[Delange-Wirsing-Hal\'asz]\label{T:Halasz}
	Let $f\colon \N\to \U$ be a multiplicative function.
\begin{enumerate}
	\item\label{I:Halasz1}  If $f\not\sim n^{it}$ for every $t\in \R$, or $f\sim n^{it}$ for some $t\in \R$  and $f(2^s)=-2^{ist}$ for every $s\in \N$, then
	$$
	\lim_{N\to\infty} \E_{n\in[N]}\, f(n)=0.
	$$
	
	\item \label{I:Halasz2} If $f\sim n^{it}$ for some $t\in \R$ and $f(2^s)\neq -2^{ist}$ for some $s\in \N$, then there exist a non-zero  $L\in \C$  and a slowly varying sequence $A\colon \N\to \R_+$ such that
	$$
	\E_{n\in[N]}\, f(n)=L\, N^{it} e(A(N)) +o_N(1).
	$$
\end{enumerate}
\end{theorem}
\begin{remarks}
	$\bullet$ If $f$ is completely multiplicative, then $L$ is always non-zero since we cannot have  $f(2^s)=-2^{ist}$ for every $s\in \N$.
	
	$\bullet$ Part~\eqref{I:Halasz1} also holds for logarithmic averages, this follows using partial summation. Part~\eqref{I:Halasz2}   fails for logarithmic averages, take for example  $f(n)=n^{it}$ for some $t\neq 0$, then the logarithmic averages vanish but not the Ces\`aro averages (more generally, the logarithmic averages vanish if $f\sim n^{it}$ for some  $t\neq 0$).
	On the other hand, if $t=0$ (equivalently, if  $f\sim 1$), then  part~\eqref{I:Halasz2}
	 continues to  hold
	if  we replace $\E_{n\in [N]}\, f(n)$ with $\lE_{n\in [N]}\, f(n)$.
To see this, note that since 	
		$\lim_{N\to\infty}  \sup_{n\in [N^c,N]}|A(n)-A(N)|=0$ for every $c>0$, we can use partial summation to deduce the asserted asymptotic for the averages $\lE_{n\in[N^c,N]}\, f(n)$ for all $c>0$. Letting $c\to 0^+$,
 gives the  result for  the averages $\lE_{n\in [N]}\, f(n)$.
\end{remarks}

\subsection{Measure preserving systems}\label{SS:mps}
Throughout the article, we make the typical assumption that all probability  spaces $(X,\CX,\mu)$ considered are {\em standard Borel}, meaning, $X$ can be endowed with the structure of a complete and separable metric space and $\CX$ is its Borel $\sigma$-algebra.
A {\em measure preserving system}, or simply {\em a system}, is a quadruple $(X,\CX,\mu,T)$
where $(X,\CX,\mu)$ is a probability space and $T\colon X\to X$ is an invertible, measurable,  measure preserving transformation.
In general, we omit the $\sigma$-algebra $\CX$  and write $(X,\mu,T)$.
The system is {\em ergodic} if the only $T$-invariant sets in $\CX$ have measure $0$ or $1$. If $f\in L^\infty(\mu)$ and $n\in\Z$, with $T^nf$ we denote the composition $f\circ T^n$, where for $n\in\N$, we let $T^n:=T\circ\cdots\circ T$ ($n$ times),  $T^{-n}=(T\inv)^n$,  and $T^0=\id$.

\subsubsection*{Factors and isomorphisms}
A {\it homomorphism},  also called a \emph{factor map}, from a system $(X,\CX,\mu, T)$ onto a
system $(Y, \CY, \nu, S)$ is a measurable map $\Phi\colon X\to Y$,
such that
$\mu\circ\Phi^{-1} = \nu$ and with $S\circ\Phi = \Phi\circ T$ valid  for $\mu$ almost every $x\in X$.
When we have such a homomorphism, we say that the system $(Y, \CY,
\nu, S)$ is a {\it factor} of the system $(X,\CX,\mu, T)$.  If
the factor map $\Phi\colon X\to Y$ is injective on a $T$-invariant set of full $\mu$-measure,
we say that $\Phi$ is an \emph{isomorphism} and that
the systems $(X,\CX, \mu, T)$ and $(Y, \CY, \nu, S)$
are {\it isomorphic}.

  The {\em pushforward of $\mu$ by $\Phi$}, is denoted by $\Phi_*\mu$ or $\mu\circ\Phi^{-1}$, and is defined by
 $$
 \int F\, d(\Phi_*\mu)=\int F\circ \Phi\, d\mu
 $$
 for every $F\in L^\infty(\nu)$.

\subsubsection*{Spectrum, procyclic,  and rational discrete spectrum systems}
The notion of the spectrum and the class of systems with rational discrete spectrum play a crucial role in this article and we define these concepts next.
\begin{definition}\label{D:Spectrum}
Given a system $(X,\mu,T)$, we define its {\em spectrum}, and denote it by   $\spec(X,\mu,T)$, to be the set of
$\alpha\in [0,1)$ for which there exists a non-zero $g\in L^2(\mu)$ such that  $Tg=e(\alpha) \cdot g$.
We call any such $g$ an  {\em eigenfunction of the system}.

\end{definition}
Hence, $\alpha\in \spec(X,\mu,T)$ if it is the phase (taken in $[0,1)$) of an eigenvalue of the operator $f\mapsto Tf$, acting on $L^2(\mu)$.
We will often identify $[0,1)$ with $\T$, and using  this identification we have that the spectrum of a system is always a subset of $\T$ that is closed under multiplication by integers.
When the system is ergodic, then the spectrum is a subgroup of $\T$.
We say that $e(\alpha)$ is a {\em rational eigenvalue}, if $\alpha\in \Q$.

\begin{definition}
We say that the system $(X,\mu,T)$
\begin{enumerate}
	\item is {\em trivial} if there exists $x_0\in X$ such that $Tx=x_0$
	for $\mu$-almost every $x\in X$.
	
	 \smallskip
	
\item is {\em an identity} if $Tx=x$ for $\mu$ almost every $x\in X$.

\smallskip

\item is {\em cyclic}  if it is isomorphic to a rotation on a finite cyclic group.

\smallskip

\item has {\em trivial rational spectrum} if $\spec(X,\mu,T)\cap \Q=\{0\}$.

\smallskip

	\item  has {\em  rational discrete spectrum}  if $L^2(\mu)$ is spanned by eigenfunctions with rational eigenvalues.

\smallskip

\item is {\em procyclic}  if it is an ergodic system with rational discrete spectrum.\footnote{ Procyclic systems are often called ergodic odometers in the literature.}
\end{enumerate}
\end{definition}
It can be shown that  a system has  rational discrete spectrum  if and only if its ergodic components are procyclic systems. Also, a procyclic
system is an inverse limit of cyclic systems and this happens if and only if it is isomorphic to an ergodic rotation on a procyclic group.
Lastly, we remark that  procyclic systems are isomorphic if and only if they have   the same spectrum, but this is not the case for
(non-ergodic) systems with rational discrete spectrum.

\subsubsection*{Joinings and disjoint systems}
Following \cite{Fu67}, if $(X,\CX, \mu,T)$ and $(Y,\CY, \nu,S)$ are two systems, we call a measure $\rho$ on $(X\times Y, \CX\otimes \CY)$ a {\em joining} of the two systems
if $\rho$  is $T\times S$-invariant and its projection onto the  $X$ and $Y$ coordinates are the measures $\mu$ and $\nu$, respectively. We say that the systems on $X$ and on $Y$ are {\em disjoint} if the only joining of the systems
is the product measure $\mu\times \nu$.

We will use the following well known facts (the Furstenberg systems of the sequences $a,b\colon \N\to \U$ used below can be take either for Ces\`aro or logarithmic averages):
\begin{enumerate}
	\item All identity systems are disjoint from all ergodic systems.
	
	\smallskip
	
	\item Bernoulli systems are disjoint from all zero entropy systems.
	
	\smallskip
	
	\item\label{I:ab}  If $a,b\colon \N\to \U$ are sequences, then any Furstenberg system (see Definition~\ref{D:Furstenberg}) of the product sequence $a\cdot b$ is a factor of a joining of some Furstenberg system of $a$ and another Furstenberg system of $b$.
	
	\smallskip
	
	\item\label{I:ratdisc} The spectrum of an ergodic joining of two discrete spectrum systems is contained in the subgroup of $\T$ generated by the spectrum of the individual
	systems.\footnote{Similarly, we can define joinings of countably many systems and property  \eqref{I:ratdisc} extends to this
		more general setting.}
	
	\smallskip
	
	\item If all  Furstenberg systems of a sequence $a\colon \N\to \U$ are disjoint from all Furstenberg systems of a sequence $b\colon \N\to \U$,
	then the spectrum of any Furstenberg system of the  product sequence $a\cdot b$
is contained in the subgroup generated by the spectrum of  a Furstenberg system of the sequence $a$ and another Furstenberg system of the sequence $b$.
\end{enumerate}

\subsection{Furstenberg systems of sequences}\label{SS:Defs}
 We reproduce the notion of a Furstenberg system of a bounded sequence
  and record some basic related facts that will be used later.
\subsubsection{Sequence space systems and correlations}
We start with two preparatory notions.
\begin{definition}\label{D:sequencespace}
Let  $X:=\U^\Z$ and denote  the elements of $X$ by $x:=(x(k))_{k\in\Z}$.
	A {\em sequence space system}, is a system $(X,\mu,T)$, where  $T\colon X\to X$  is the {\em shift transformation,}    defined by $(Tx)(k):=x(k+1)$, $k\in \Z$, and $\mu$ is a $T$-invariant measure. We implicitly assume that $X$ is equipped with the product topology.
	We let $F_0\in \C(X)$  be defined by $F_0(x):=x(0)$, $x \in X$, and call it the {\em $0^{\text{th}}$-coordinate projection}.
\end{definition}
\begin{remark} Any probability measure on $X$ is uniquely determined by its values on the set of {\em cylinder functions}
	$\{\prod_{j=1}^\ell T^{n_j} F_{j}, \ell\in\N, n_1,\ldots, n_\ell\in \Z,  F_1,\ldots, F_\ell \in\{F_0,\overline{F_0}\}\}$, since this set is
	linearly dense in $C(X)$. Another  representation for this set  is
	$\{\prod_{j=1}^\ell T^{n_j}F^{k_j}_0, \ell\in\N, n_1,\ldots, n_\ell\in \Z,  k_1,\ldots, k_\ell\in \Z\}$, where we use the notation $F_0^{k}:=\overline{F_0}^{|k|}$ for $k<0$.
\end{remark}

\begin{definition}
	Let   $ ([N_k])_{k\in\N}$ be a sequence of intervals with $N_k\to \infty$.
	We say that a sequence $a\colon \Z\to \U$
	{\em admits   correlations  along $ ([N_k])$} (or simply {\em along} $(N_k)$)  {\em for Ces\`aro averages}   if the   limits
	\begin{equation}\label{E:CorDef}
		\lim_{k\to\infty}\E_{n\in[N_k]} \prod_{j=1}^\ell a_j(n+n_j)
	\end{equation}
	exist for all $\ell \in \N$, $n_1,\ldots, n_\ell\in \Z$,  and all   $a_1,\ldots, a_\ell\in \{a,\overline{a}\}$.
\end{definition}
\begin{remarks}
	$\bullet$	 Given $a\colon \Z \to \U$,   using a diagonal argument, we get that every sequence of intervals $([N_k])$
	has a subsequence $([N_k'])$ such that the sequence  $a$  admits  correlations for Ces\`aro averages  along $([N_k'])$.
	
	$\bullet$
	If we are given a one-sided sequence $(a(n))_{n\in\N}$, we extend it to $\Z$ in an arbitrary way; then the existence and values of the correlations do not depend on the extension.
\end{remarks}
\subsubsection{Furstenberg systems of sequences}
If a sequence $a\colon \Z\to \U$   admits  correlations for Ces\`aro averages along $(N_k)$,  then we use a  variant of the correspondence principle of Furstenberg~\cite{Fu77, Fu}
to associate a sequence space  system that captures the statistical properties of $a$ along $(N_k)$. We briefly describe this process next.

We consider the sequence $a=(a(n))$ as an element of $X$. Note that  the conjugation closed algebra generated by functions of the form $x\mapsto x(k)$, $x\in X$, for   $k\in  \Z$,  separates points in $X$. We conclude that  if the
sequence $a\colon \Z\to \U$
admits   correlations  along $(N_k)$,  then,  for all $F\in C(X)$, the following limit exists
$$
\lim_{k\to\infty}\E_{n\in[N_k]}\,  F(T^na).
$$
Hence, the  following weak-star limit exists
\begin{equation}\label{uklFur}
\mu:=\lim_{k\to\infty}\E_{n\in[N_k]}\, \delta_{T^na}
\end{equation}
and we say that \emph{the point $a$ is generic for $\mu$ along $([N_k])$ or $(N_k)$}.

\begin{definition}
	\label{D:Furstenberg}
	Let  $a\colon \Z\to \U$ be a sequence that  admits
	correlations  for Ces\`aro averages along $(N_k)$,  and  $(X,\mu, T)$ be as above.
	\begin{enumerate}
		\item
		We call $(X,\mu,T)$, with $\mu$ given by \eqref{uklFur}, the {\em  Furstenberg system
			of $a$ along  $([N_k])$, or for simplicity  $(N_k)$, for Ces\`aro averages.}
		
		\smallskip
		
		\item
	If  $F_0\in \C(X)$  is the   {\em $0^{\text{th}}$-coordinate projection}, then    $F_0(T^nx)=x(n)$ for every $n\in \Z$, and
		\begin{equation}
			\label{E:correspondence}
			\lim_{k\to\infty}\E_{n\in[N_k]}\,  \prod_{j=1}^\ell a_j(n+n_j) =\int  \prod_{j=1}^\ell
			T^{n_j}F_j \, d\mu
		\end{equation}
		for all  $\ell\in \N$, $n_1,\ldots, n_\ell\in \Z$,  and
		 $a_1,\ldots, a_\ell\in \{a,\overline{a}\}$, where for $j=1,\ldots, \ell$, the function $F_j$ is $F_0,\overline{F_0}$ if respectively if $a_j$ is $a,\overline{a}$. Note that the identities \eqref{uklFur} and \eqref{E:correspondence} are equivalent, but it is \eqref{E:correspondence} that we will mostly use.
		
		 \smallskip
		
		\item
		We say that the sequence
		$a$ has a {\em unique} Furstenberg system  if  $a$ admits  correlations on $([N])_{N\in\N}$,  or equivalently, if $a$ is generic for a measure along $([N])_{N\in\N}$.
	\end{enumerate}
\end{definition}	
\begin{remarks}
$\bullet$ 	A sequence $a\colon \Z\to \U$ may have several  Furstenberg systems
	depending on which sequence of intervals $([N_k])$ we use in the evaluation of its   correlations. We call any such system {\em a  Furstenberg system of $a$ for Ces\`aro averages}.
	For a fixed sequence these systems may or may not be  isomorphic.
	
$\bullet$ Since the set of measures defining Furstenberg systems of $a$ is known to be connected (in the weak-star topology), either it has a single element, or it is uncountable.

	$\bullet$ We will sometimes use the following fact: If $(X,\mu,T)$ is a Furstenberg system of a sequence $a\colon \N\to \U$ along $(N_k)$, then for every $m\in\N$, the Furstenberg system of the $m$-th power of $a$  along $(N_k)$ is also well defined, and it is a factor of the system $(X,\mu,T)$ (the factor map is
	$\Phi\colon X\to X$ defined by $(\Phi x)(k):=x^m(k), k\in\Z$).
\end{remarks}	
Given a bounded sequence, our goal is  to obtain  structural properties of their Furstenberg systems. Ultimately, we would like to  completely determine them up to isomorphism using as building blocks  systems with algebraic structure, such as nilsystems, and systems that enjoy strong randomness properties, such as Bernoulli systems.

 Similar notions as in the previous subsection  can be defined for logarithmic averages:
\begin{definition} By replacing Ces\`aro averages with logarithmic averages in  Definitions~\ref{D:sequencespace}-\ref{D:Furstenberg},
	given a sequence $a\colon \N\to \U$ and a sequence of intervals $([N_k])$ along which $a$ admits correlations for logarithmic  averages, we can introduce {\em Furstenberg systems of $a$ 	 along    $(N_k)$ for logarithmic averages}. If $(X,\mu,T)$ is such a system, then an identity similar to \eqref{E:correspondence} is satisfied with $\lE_{n\in[N_k]}$ in place of $\E_{n\in [N_k]}$. Equivalently,
	$$
	\mu=\lim_{k\to\infty}\lE_{n\in [N_k]}\,  \delta_{T^na},
	$$
	where the limit is a weak-star limit.
	\end{definition}
\begin{remarks}
	$\bullet$ If $a$ has a unique Furstenberg system for Ces\`aro averages, then the same system is the unique Furstenberg system of $a$ for logarithmic averages.
	
	$\bullet$ In general, the sets of Furstenberg systems for Ces\`aro averages and logarithmic averages can be disjoint, as the example of $a(n)=n^{it}$, $n\in\N$, for $t\neq0$ shows, see \cite[Corollary~5.5]{GLR21} and  \cite[Section~1.3]{FH21}.

	$\bullet$
	Each ergodic Furstenberg system for logarithmic averages along $(N_k)$ is also a Furstenberg systems for Ces\`aro averages along a possibly different sequence $(N_k')$ (see \cite[Corollary~2.2]{Go-Kw-Le}).
\end{remarks}

\subsubsection{Ergodic and completely deterministic sequences}
Using the  Furstenberg correspondence principle, we can naturally associate ergodic notions to bounded sequences -
we record here two that are used in this article.
\begin{definition}\label{D:ergodicnotions}
	We say that a  sequence   $a\colon \N\to \U$ is
	\begin{enumerate}
			\item {\em ergodic for Ces\`aro averages}  if
		all its Furstenberg systems for Ces\`aro averages are ergodic.
		
		\smallskip
		
		\item {\em ergodic for Ces\`aro averages along $(N_k)$} if
		all its Furstenberg systems for Ces\`aro averages along subsequences of $(N_k)$ are ergodic.
		
			\item  {\em   completely deterministic for Ces\`aro averages}  if
		all its Furstenberg systems for Ces\`aro averages  have zero entropy.\footnote{This notion was originally introduced in \cite{Ka,We}.}
		
		\item  {\em  completely deterministic for Ces\`aro averages along $(N_k)$}  if
			all its Furstenberg systems for Ces\`aro averages along subsequences of $(N_k)$ have zero entropy.
	\end{enumerate}
	Similar definitions apply for logarithmic averages.
\end{definition}
We will use the fact that if $a,b\colon \N\to \U$ are completely deterministic sequences for Ces\`aro or logarithmic averages along $(N_k)$, then so is their product $a\cdot b$.
To see this, use property \eqref{I:ab} in the subsection about joinings in Section~\ref{SS:mps}, and the fact that zero entropy systems are closed under joinings and factors.

\subsubsection{Strong stationarity}
Lastly, we define the notion of strong stationarity that was introduced by Furstenberg and Katznelson in \cite{FK91} and turns out to be relevant for  the structural analysis of measure preserving systems associated with  non-pretentious multiplicative functions. 		
\begin{definition}\label{D:sst}
	A sequence space system $(X,\mu,T)$ is {\em strongly stationary} if
	\begin{equation}\label{E:sst}
		\int \prod_{j=1}^\ell T^{n_j}F_j\, d\mu=	\int \prod_{j=1}^\ell T^{rn_j}F_j\, d\mu
	\end{equation}
	for all $\ell,r\in \N$, $n_1,\ldots, n_\ell\in \Z$,  and  $F_1,\ldots, F_\ell\in \{F_0,\overline{F_0}\}$,
	where $F_0\in L^\infty(\mu)$ is the $0^{\text{th}}$-coordinate projection.\footnote{Equivalently,    for every $r\in\N$, the dilation map $\tau_r\colon X\to X$ in \eqref{E:taur}  is $\mu$-preserving.}
	A sequence $a\colon \N\to \U$ is {\em strongly stationary for Ces\`aro averages}  if all its Furstenberg systems for Ces\`aro averages are strongly stationary.
	
	Similar definitions apply for logarithmic averages.
\end{definition}

 \subsection{Besicovitch rationally almost periodic sequences}\label{SS:Bes}
 We will use the following variant of the classical notion of Besicovitch rational almost periodicity.
 \begin{definition}\label{D:RAP}
 	Let $N_k\to \infty$ be a sequence of integers.	 Following \cite{BPLR19},	we say that a sequence $a\colon \N\to \U$ is
 	\begin{enumerate}
 		\item 	  {\em Besicovitch rationally almost periodic for Ces\`aro averages along $(N_k)$}  if for every $\varepsilon>0$ there exists a periodic sequence $a_\varepsilon\colon \N\to \U$   such that
 		$$
 		\limsup_{k\to\infty}\E_{n\in [N_k]}|a(n)-a_\varepsilon(n)|^2\leq \varepsilon.
 		$$

 		\item  {\em Besicovitch rationally almost periodic for logarithmic averages  along $(N_k)$}
 		if for every $\varepsilon>0$ there exists a periodic sequence $a_\varepsilon\colon \N\to \U$   such that
 		$$
 		\limsup_{k\to\infty}\lE_{n\in [N_k]}|a(n)-a_\varepsilon(n)|^2\leq \varepsilon.
 		$$
 	\end{enumerate}
 	If we make no reference to $N_k$, we mean that the statement holds for every sequence $N_k\to\infty$, in other words, we can replace 	$\limsup_{k\to\infty}\E_{n\in [N_k]}$ with
 	$\limsup_{N\to\infty}\E_{n\in [N]}$, and similarly for logarithmic averages.
 \end{definition}

 It is easy to verify that if a sequence is  Besicovitch rationally almost periodic  for Ces\`aro (or logarithmic) averages, then it  has a mean value on every infinite arithmetic progression and its correlations for Ces\`aro (or logarithmic) averages exist. As
  a consequence, it has a unique Furstenberg system for Ces\`aro (or logarithmic) averages.
 We will use the next result that gives structural information for Furstenberg systems of such sequences.
 \begin{theorem}{\cite[Theorem~3.12]{BPLR19}}\label{T:BRAP}
 	Let $a\colon \N\to \U$ be a sequence that is
 	Besicovitch rationally almost periodic for Ces\`aro averages along some subsequence  $(N_k)$.
 	Then:
 	\begin{enumerate}
\item  	The sequence $a$ has a unique Furstenberg system  $(X,\mu,T)$ for Ces\`aro averages along $(N_k)$.

\smallskip

\item The system $(X,\mu,T)$  is an ergodic procyclic system.
 \end{enumerate}
 	Furthermore, a similar statement holds for logarithmic averages. 	
 \end{theorem}
The argument in \cite{BPLR19} is given only for Ces\`aro averages, but exactly the same argument also works for logarithmic averages. 	 

	\section{General results for  multiplicative functions}

\subsection{Divisibility property of the spectrum - Proof of Theorem~\ref{T:SpectrumDivisible}}\label{SS:ProofThm2.1}
In this subsection we prove  	Theorem~\ref{T:SpectrumDivisible}. We will need some preparatory results. In particular,  it will be crucial to get some understanding of
how the dilation maps $\tau_r$, which we are about to define, act on eigenfunctions of Furstenberg systems of multiplicative functions. This  is the context of Lemmas~\ref{L:taurchi} and \ref{L:abscont}.   	These two lemmas combined enable us to prove Theorem~\ref{T:SpectrumDivisible}.
\begin{definition}
	Let  $T$ be the shift transformation on the sequence space  $X=\U^\Z$.
	\begin{enumerate}
		\item
	For $r\in \N$, we let $\tau_r\colon X\to X$ be the {\em dilation by $r$ map} defined by
	\begin{equation}\label{E:taur}
		(\tau_rx)(j):=x(rj), \quad j\in \Z.
	\end{equation}
	It satisfies the commutation relation
	\begin{equation}\label{E:Ttaur}
		T\tau_r=\tau_rT^r,
	\end{equation}
	meaning,  $T(\tau_r x)=\tau_r(T^rx)$, for every $x\in X$.
	
	\smallskip

\item For $z\in \U$, we define the map $M_z\colon X\to X$ by
\begin{equation}\label{E:Mz}
	(M_zx)(j):=z\cdot x(j), \quad j\in \Z.
	\end{equation}
	It satisfies the commutation relation
	\begin{equation}\label{E:TMz}
	TM_z=M_zT.
\end{equation}
	\end{enumerate}
\end{definition}
\begin{lemma}[Proof of Lemma~2.3 in \cite{Jen97}]\label{L:taurchi}
	Let $(X,T)$ be the sequence space with the shift transformation and $\tau_r\colon X\to X$ be the dilation map defined in \eqref{E:taur}.   Let  $\chi\colon X\to \U$ be a function, such that
 for some $\alpha\in [0,1)$, we have $\chi(Tx)= e(\alpha)\, \chi(x)$ for  every $x\in X$.
	  Then
	$$
	 \chi(\tau_rx)=\chi_0(x)+\cdots+\chi_{r-1}(x), \quad x\in X,
	$$  for some $\chi_0,\ldots, \chi_{r-1}$   that are linear combinations of the functions $\chi\circ\tau_r\circ T^j$, $j=0,\ldots, r-1$, and  satisfy
	$$
	\chi_k(Tx)= e((\alpha+k)/r)\, \chi_k(x),\quad k=0,\ldots,r-1,\,  x\in X.
	$$
\end{lemma}
\begin{remark}
	Note that if  $X$ is equipped with the product topology and  $\chi$ is Borel measurable, then also $\chi_k$ is Borel measurable for $k=0,\ldots, r-1$.
\end{remark}
\begin{proof}
	For $k=0,\ldots, r-1$, let $\chi_k:=\frac{1}{r}\sum_{j=0}^{r-1}e(-j(k+\alpha)/r) \, \chi\circ \tau_r\circ T^j$.
	Then the asserted properties follow by a direct computation  using our assumption and the  commutation relation \eqref{E:Ttaur}.
\end{proof}
We caution the reader  that when we consider a sequence space system $(X,\mu,T)$ we may have $\chi\neq 0$ (with respect to $\mu$) but  $\tau_r\chi=\tau_r \chi_0=\cdots =\tau_r\chi_{r-1}=0$  (with respect to $\mu$), in which case the content of Lemma~\ref{L:taurchi}  is practically empty.\footnote{For example, let  $\mu:=(\delta_{x_0}+\delta_{x_1})/2$, where $x_i:={\bf 1}_{2\Z+i}$, $i=0,1$,  and
	$\chi(x):={\bf 1}_{\{x_0\}}(x)-{\bf 1}_{\{x_1\}}(x)$, $x\in X$. Then  $\chi(Tx)=-\chi(x)$ for $\mu$-a.e.\ $x\in X$ and $\chi\neq 0$
 with respect to $\mu$. But one easily verifies that $\chi(\tau_2 x_0)=\chi(\tau_2x_1)=0$, hence $\chi\circ\tau_2=\chi\circ\tau_2\circ T=0$  with respect to $\mu$. Note that the system $(X,\mu,T)$ is the Furstenberg system of the multiplicative function $f(n):=(-1)^{n+1}$, $n\in \N$. }
For Furstenberg systems of completely multiplicative functions though, Lemma~\ref{L:taurchi} can be combined with the next result that alleviates this problem, in order  to get interesting consequences.
\begin{lemma}\label{L:abscont}
	Let $(X,\mu,T)$ be a Furstenberg system for logarithmic averages of a completely multiplicative function $f\colon \N\to \U$. Then, for every  $r\in \N$,   we have
	$$
(M_{f(r)})_* \mu\leq r\cdot 	(\tau_r)_*\mu.
	$$
	Furthermore, a similar statement holds if a Furstenberg system is defined using Ces\`aro averages and is ergodic.
\end{lemma}
\begin{remark}
	For Ces\`aro averages we get the following more explicit result:  If we let
	$\mu:=\lim_{k\to\infty} \E_{n\in [N_k]} \, \delta_{T^n f}$ and $\mu_r:=\lim_{k\to\infty} \E_{n\in [rN_k]} \, \delta_{T^n f}$ and  we assume that both weak-star limits exist (which can always be arranged by passing to a subsequence of $(N_k)$), then $$
	(M_{f(r)})_* \mu\leq r\cdot 	(\tau_r)_*\mu_r.
	$$
\end{remark}
\begin{proof}
 It suffices to show that  for every $F\in C(X)$ with $F\geq 0$, we have
\begin{equation}\label{E:MfrF}
\int M_{f(r)} F\, d\mu\leq r\int \tau_rF\, d\mu.
\end{equation}
Suppose that the Furstenberg system of $f$ is taken along $(N_k)$.
 Thinking  of $f$ as the element $(f(k))$ of the sequence space $X$, we get that  $\mu$ is the  weak-star limit
\begin{equation}\label{E:mulog}
\mu=\lim_{k\to\infty} \lE_{n\in [N_k]} \, \delta_{T^n f}=\lim_{k\to\infty} \lE_{n\in [N_k/r]} \, \delta_{T^n f},
\end{equation}
where the second identity holds  because of the logarithmic averaging.
Note also that for every $z\in \C$ and $r\in \N$, we have $M_zF,\tau_rF\in C(X)$, whenever $F\in C(X)$.

 It follows from these facts  and   \eqref{E:TMz} that
$$
\int M_{f(r)} F\, d\mu=\lim_{k\to\infty} \lE_{n\in [N_k/r]} \, (F\circ M_{f(r)})(T^{n}f)=
\lim_{k\to\infty} \lE_{n\in [N_k/r]} \, F( T^{n}M_{f(r)}f).
$$
Since
$f$ is completely multiplicative, we have
\begin{equation}\label{komut}
(M_{f(r)}f)(n)=f(r)f(n)=f(rn)=(\tau_rf)(n).
\end{equation}
Hence, using the commutation relation \eqref{E:Ttaur}, we get that the last limit equals
$$
\lim_{k\to\infty} \lE_{n\in [N_k/r]} \, F( \tau_rT^{rn}f)\leq r \lim_{k\to\infty} \lE_{n\in [N_k]} \, F( \tau_rT^{n}f)= r\int \tau_rF\, d\mu,
$$
where to get the upper bound we used that $F\geq 0$ and the elementary estimate (for $a(n):=F( \tau_rT^{n}f)$, $n\in\N$)
$$
\lim_{k\to\infty} \lE_{n\in [N_k/r]}\,  a(rn)\leq r\, \lim_{k\to\infty} \lE_{n\in [N_k]} \, a(n),
$$
which  is valid for every  sequence $a\colon \N\to \R_+$ such that  the previous limits exist. This  establishes \eqref{E:MfrF} and completes the proof.

Suppose now that the Furstenberg system $(X,\mu,T)$ is ergodic and $\mu$ is defined by the weak-star limit 
$$
\mu=\lim_{k\to\infty} \E_{n\in [N_k]} \, \delta_{T^n f}
$$
for some sequence $N_k\to\infty$.  We claim that for every $r\in \N$ the weak-star limit below exists and we  have
 \begin{equation}\label{E:muCes}
\mu=\lim_{k\to\infty} \E_{n\in [N_k/r]} \, \delta_{T^n f}.
\end{equation}
Assuming the claim, we use \eqref{E:muCes} as our starting point in place of \eqref{E:mulog} and repeat the previous argument verbatim to get the desired conclusion.
In order to prove the claim, after passing to a subsequence, we can assume that the second limit also  exists (if the identity fails, then it would also fail on a subsequence along which the second weak-star limits exist). Let
$\mu'$ be the  measure by this limit (the limit on the right hand side in \eqref{E:muCes}). Then $\mu'$ is $T$-invariant and satisfies $\mu'\leq  r\mu$, hence it is absolutely continuous with respect to $\mu$.
Since the system $(X,\mu,T)$ is ergodic, we deduce that $\mu=\mu'$, completing the proof of the claim.
	\end{proof}
		\begin{proof}[Proof of Theorem~\ref{T:SpectrumDivisible}]
		We prove part~\eqref{I:Sdiv1}. 	By assumption, there exists $\chi\in L^\infty(\mu)$ such that $\chi\neq 0$ with respect to $\mu$ and $T\chi=e(\alpha)\chi$.
		After redefining the function $\chi$ on a set of $\mu$-measure $0$, we can assume that $\chi$ is defined for every $x\in X$ and satisfies  the identity
		\begin{equation}\label{E:chi}
		\chi(Tx)=e(\alpha)\, \chi(x), \quad  \text{for every } x\in X.\footnote{If $Z:=\{x\in X\colon \chi(Tx)\neq e(\alpha) \, \chi(x)\}$ and $Z':=\bigcup_{n\in \Z}T^nZ$, then $\mu(Z')=0$. Define $\tilde{\chi}:=\chi\cdot {\bf 1}_{X\setminus Z'}$, then $\tilde{\chi}(Tx)=e(\alpha)\, \tilde{\chi}(x)$ for every $x\in X$ and $\tilde{\chi}\neq 0$ with respect to $\mu$.}
		\end{equation}
		
		Let $r\in \N$ be  such that $f(r)\neq 0$, let  $X_r:=M_{f(r)}X$.  Note that $X_r$ is a $T$-invariant subset of $X$ and the map $M_{f(r)}\colon X\to X_r$ is a homeomorphism. We also define $\tilde{\chi}\colon X\to X$ by
		$$
		\tilde{\chi}(x):=\chi(M^{-1}_{f(r)}x) \cdot {\bf 1}_{X_r}(x), \quad x\in X,
		$$
		where $M^{-1}_{f(r)}x$ is defined in an arbitrary way on $X\setminus X_r$ (this does not affect the definition of $\tilde{\chi}(x)$).
		
		
		We claim that
		$\tau_r\tilde{\chi}\neq 0$ with respect to $\mu$. Indeed, we have
		$$
		\mu(\tau_r\tilde{\chi}\neq 0)=((\tau_r)_*\mu)(\tilde{\chi}\neq 0).
		$$
	By Lemma~\ref{L:abscont}, the right hand side is at least
		$$
		\frac{1}{r}((M_{f(r)})_*\mu)
		(\tilde{\chi}\neq 0)=\frac{1}{r}\mu(\chi\neq 0)>0,
		$$
		where the last equality holds since
		$$
		M_{f(r)}\tilde{\chi}(x)=\chi(M^{-1}_{f(r)}\circ M_{f(r)}x)\cdot   {\bf 1}_{X_r}(M_{f(r)}x)=\chi(x), \qquad \text{for every } x\in X.
		$$
		This proves the claim.
		
		Using  \eqref{E:TMz},  \eqref{E:chi}, and the fact that ${\bf 1}_{X_r}(Tx)={\bf 1}_{X_r}(x)$ for every $x\in X$, we get
		$$
		\tilde{\chi}(Tx)=e(\alpha)\, \tilde{\chi}(x), \quad  \text{for every } x\in X.
		$$
		By Lemma~\ref{L:taurchi} we have
		$$
		 \tilde{\chi}(\tau_rx)=\chi_0(x)+\cdots+\chi_{r-1}(x),  \quad  \text{for every }  x\in X,
		$$  for some $\chi_0,\ldots, \chi_{r-1}$   that are linear combinations of the functions $\tilde{\chi}\circ\tau_r\circ T^j$, $j=0,\ldots, r-1$, and  satisfy
		$$
		\chi_k(Tx)= e((\alpha+k)/r)\, \chi_k(x),\quad   \text{ for every } x\in X,
		$$
		for $k=0,\ldots,r-1$.
		
	 Since   $\tau_r\tilde{\chi}\neq 0$ with respect to $\mu$, we have     $\chi_k\neq 0$ with respect to $\mu$ for some $k\in \{0,\ldots, r-1\}$.  It follows that  $(\alpha+k)/r\in \spec(X,\mu,T)$.
		
We prove part~\eqref{I:SpecCes}.	If $(X,\mu,T)$ is a Furstenberg system for Ces\`aro averages of $f$ and $\alpha\in \spec(X,\mu,T)$,
		then using the remark following Lemma~\ref{L:abscont} and the previous argument,  we get that 	$(\alpha+k)/r\in \spec(X,\mu_r,T)$ for some $k\in \{0,\ldots, r-1\}$.	
Since $(X,\mu_r,T)$ is also a Furstenberg system of $f$, we get the first asserted statement.

	Lastly, suppose that that the system $(X,\mu,T)$ is ergodic. Then arguing as we did in  Lemma~\ref{L:abscont} in order to arrive to  \eqref{E:muCes}, we get that if
$$
\mu':=\lim_{k\to\infty} \E_{n\in [N_k/r]} \, \delta_{T^n f},
$$
then the weak-star limit exists and $\mu=\mu'$. Then $\alpha\in \spec(X,\mu',T)$, and arguing as before, we have that 	$(\alpha+k)/r\in \spec(X,\mu'_r,T)$ for some $k\in \{0,\ldots, r-1\}$.	
Since $\mu'_r=\mu$, the second asserted statement follows.
\end{proof}

	\subsection{Strong stationarity - Proof of Theorem~\ref{T:sst}}\label{SS:Proofsst}
	
	We will also use the following result from \cite[Theorem~6.4]{Fr04}
	(it follows immediately since the ergodic components of systems with trivial rational spectrum are totally ergodic).
	\begin{theorem}\label{T:F}
		Let $(X,\mu,T)$ be a system with trivial rational spectrum. Then
		$$
		\lim_{N\to\infty}\E_{n\in [N]}  \int \prod_{j=1}^\ell T^{jn}F_j\, d\mu=	 	\lim_{N\to\infty}\E_{n\in [N]}  \int \prod_{j=1}^\ell T^{j(rn+k)}F_j\, d\mu
		$$
		for all $\ell,r\in \N$, $k\in \{0,\ldots, r-1\}$,   and $F_1,\ldots, F_\ell\in L^\infty(\mu)$. Equivalently, for every  rational $\alpha\in (0,1)\cap \Q$, we have
		$$
		\lim_{N\to\infty}\E_{n\in [N]} \, \Big( e(n\alpha)\cdot \int \prod_{j=1}^\ell T^{jn}F_j\, d\mu \Big)=0
		$$
		for all $\ell\in \N$  and $F_1,\ldots, F_\ell\in L^\infty(\mu)$.
	\end{theorem}
	We deduce from this the following.
	\begin{corollary}\label{C:F}
		Let $(X,\mu,T)$ be a system with trivial rational spectrum and  for some  fixed $\ell\in\N$ and $F_1,\ldots, F_\ell\in L^\infty(\mu)$, let  the sequence $C\colon \N\to \U$ be defined by
		$$
		C(r):= \int \prod_{j=1}^\ell T^{jr}F_j\, d\mu, \quad r\in \N.
		$$
		If $C$ is  the limit of periodic sequences in the uniform norm, then $C$ is constant.
	\end{corollary}
	\begin{proof}
		By Theorem~\ref{T:F}, we have that
		\begin{equation}\label{E:Cr}
			\lim_{R\to\infty} \E_{r\in[R]}\,  e(r\alpha)\cdot C(r)=0
		\end{equation}
		for every $\alpha\in (0,1)\cap \Q$. The same also holds for irrational $\alpha\in (0,1)$ since $C$ is a uniform limit of periodic sequences. Hence, equation \eqref{E:Cr} holds for all $\alpha\in(0,1)$
		and  the theory of Bohr almost periodic  sequences implies that $C$ is constant.\footnote{This follows from classical  results of Besicovitch~\cite{Bes55}, which are proved  for functions but apply equally well to sequences. Alternatively, one could use \cite[Theorem~2.7]{BPLR20} with the uniform norm in place of the Besicovitch norm.}
	\end{proof}
	We will also use the following result of Tao-Ter\"av\"ainen~\cite{TT18}.
	\begin{theorem}[\cite{TT18}]\label{T:TT}
		Let $f_1,\ldots,f_\ell\colon \N\to \U$ be arbitrary multiplicative functions. Consider a subsequence $N_k\to \infty$ along which the limits below exist
		$$
		C(r):=	\lim_{k\to\infty}\lE_{n\in[N_k]} \prod_{j=1}^\ell f_j(n+jr)
		$$
		for every $r\in \N$. Then the sequence $C$ is a uniform limit of periodic sequences.
	\end{theorem}
	
We are now ready to prove Theorem~\ref{T:sst}.
		\begin{proof}[Proof of Theorem~\ref{T:sst}]
			The implication $\eqref{I:sst1}\implies \eqref{I:sst2}$ is a general fact that holds for all strongly stationary systems and follows from \cite{Jen97} (the proof is given for ergodic systems but the same argument applies for general systems).
			We prove now the more interesting implication $\eqref{I:sst2}\implies \eqref{I:sst1}$, which is a very particular property of Furstenberg systems for logarithmic averages of bounded multiplicative functions.
			
			Let $N_k\to\infty$ be the subsequence along which the Furstenberg system $(X,\mu,T)$ of $f$ for logarithmic averages is defined.
			Recall that by Definition~\ref{D:sst} we have to verify that identity \eqref{E:sst} holds. 			
			Let $\ell\in \N$ and $f_1,\ldots, f_\ell\in \{ f,\overline{f}\}$, $n_1,\ldots, n_\ell\in \Z$.
			For  $F_1,\ldots, F_\ell\in \{F_0,\overline{F_0}\}$  we define the sequence
			\begin{equation}\label{E:fF}
				C(r):=	\int \prod_{j=1}^\ell T^{rn_j}F_j\, d\mu, \quad r\in \N.
			\end{equation}
		By \eqref{E:correspondence} we have
		$$
		C(r)=\lim_{k\to\infty}\lE_{n\in [N_k]}\prod_{j=1}^\ell f_j(n+rn_j), \quad r\in \N,
		$$
		where if $F_j$ is $F_0,\overline{F_0}$, then  respectively $f_j$ is $f,\overline{f}$.
		 	 By Theorem~\ref{T:TT}, the sequence $C$ is a uniform limit of periodic sequences. Since, additionally, the system $(X,\mu,T)$ has trivial rational spectrum, by Corollary~\ref{C:F}, the sequence $C$ is constant. We have thus established that identity \eqref{E:sst} holds for
			all $\ell,r\in \N$ and  $F_1,\ldots, F_\ell\in \{F_0,\overline{F_0}\}$. Hence, the system $(X,\mu,T)$ is strongly stationary.
		\end{proof}

\section{Preliminary results for  pretentious multiplicative functions}\label{S:PreliminariesPretentious}
	 	The main goal of this section is to do some preparatory work that will be used in the next two sections to prove our main structural results for Furstenberg systems of pretentious multiplicative functions. The main ingredient needed for later use is Proposition~\ref{P:RAPchi}, which establishes a key subsequential Besicovitch rational almost periodicity property for all multiplicative functions that pretend to be Dirichlet characters.  This fact is proved via the  decomposition result of Lemma~\ref{L:12e}, which in turn follows from the estimate in Lemma~\ref{L:thetap} that we prove in the next subsection.
	 	
	 	\subsection{Preliminary estimates} Our main goal in this subsection is to establish the estimate in  Lemma~\ref{L:thetap}, which we were not able to find in the form needed in the literature. See though \cite[Proposition~2.3]{Kl17} and \cite[Lemma~2.5]{KMPT21}
for  closely related concentration inequalities. 	
	 	
    \begin{lemma}\label{L:pt1}
	 		Let $f\colon \N\to \U$ be a multiplicative function such that $f\sim 1$.
	 	For every $p\in \P$ we can write $f(p)=r_p \, e(\theta_p)$ for some $r_p\in [0,1]$ and $\theta_p\in [-1/2,1/2)$.
	 Then
	 $$
	 		\sum_{p\in\P} \frac{1-r_p}{p}<+\infty \quad \text{and} \quad \sum_{p\in\P}\frac{\theta_p^2}{p}<+\infty.
	 $$
 	\end{lemma}
	 	\begin{proof}
	 		Since $f\sim 1$, we have
	 		$$
	 		\sum_{p\in\P} \frac{1-r_p\cos(\theta_p)}{p}<+\infty.
	 		$$
	 		From this we get immediately that the first series converges. If we use this fact and add and subtract $\cos(\theta_p)$ into the numerator of the last series, we deduce that the series
	 		$\sum_{p\in\P} \frac{1-\cos(\theta_p)}{p}$ converges, from which the convergence of the second series readily follows (we crucially use here that  $\theta_p$ is in $[-1/2,1/2)$ and not in $[0,1)$).
	 	\end{proof}
	\begin{lemma}\label{L:thetap'}
	Let  $f\colon \N\to \S^1$ be a multiplicative function  such that for every $p\in\P$ we have   $f(p^s):=e(\theta_p)$ for some $\theta_p\in [-1/2,1/2)$ and  all $s\in \N$.  For $N\in \N$, let
	\begin{equation}\label{E:AN}
		A(N):=\sum_{p\in \P\cap [N]}  \frac{\theta_p}{p}.
	\end{equation}
	Then, for some  universal constant $C>0$,  we have
	\begin{equation}\label{E:ANest}
	\E_{n\in [N]}|f(n)-e(A(N))|^2\leq  C \,	\Big(\sum_{p\in \P\cap [N]} \frac{\theta_p^2}{p}+ \frac{\log\log{N}}{\log{N}}\Big).
	\end{equation}
\end{lemma}
\begin{proof}
For convenience, we use probabilistic language. The reader can translate the argument to conventional number theoretic language by replacing $X_p$ with $\theta_p\cdot {\bf 1}_{p\Z}$
and $\E_N$ with $\E_{n\in[N]}$ throughout.

	For a fixed $N\in \N$, we  consider the finite probability space that consists of the  interval $[N]$ together with the probability measure that  assigns mass $1/N$ to each point in $[N]$. For a prime $p\in [N]$, we define the random variable
	$ X_p\colon [N]\to \{0,\theta_p\}$ by
	$$
	X_p(n):=
	\begin{cases} \theta_p \quad & \text {if } \, p\mid n, \\
		0 & \text{otherwise}.
	\end{cases}
	$$
	These random variables are relevant for our problem because if  we let
	$$
	S_N:=\sum_{p\in \P\cap [N]}\, X_p,
	$$
	then from the definition of $f$ on powers of primes and its multiplicativity, we have
	$$
	f(n)=e(S_N(n)), \quad n\in [N].
	$$
	Note that if
	$$
	a(N):=\E_N(S_N),
	$$
	then
	\begin{equation}\label{E:Var1}
		\E_{n\in [N]}|f(n)-e(a(N))|^2\leq 4\pi^2\,  \E_{n\in [N]}|S_N(n)-a(N)|^2=4 \pi^2\cdot \Var_N(S_N),
	\end{equation}
	where to get the first estimate we used that $|e(x)-e(y)|=2\pi \big|\int_x^y e(t)\, dt\big|\leq 2\pi|x-y|$.
	
	So, our problem reduces to getting an upper  bound  for   $ \Var_N(S_N)$.
	We start with some   easy computations that give:
\begin{equation}\label{srednia}
	\E_N(X_p)= \frac{\theta_p}{N}\, \Big\lfloor\frac{N}{p}\Big\rfloor,
\end{equation}
	$$
	\Var_N(X_p)= \E_N(X_p^2)-(\E_N(X_p))^2\leq \frac{\theta_p^2}{p},
	$$
	$$
	\Cov_N(X_p,X_q)=\E_N(X_p\cdot X_q)-\E_N(X_p)\cdot \E_N(X_q) = \frac{\theta_p\, \theta_q}{N}\left(\Big\lfloor\frac{N}{pq}\Big\rfloor - \frac{1}{N}\Big\lfloor\frac{N}{p}\Big\rfloor \cdot \Big\lfloor\frac{N}{q}\Big\rfloor\right).
	$$
	To get the last equality we assume that $p\neq q$ and uses  that $\E_N(X_p\cdot X_q)=\frac{\theta_p\, \theta_q}{N}\, \Big\lfloor\frac{N}{pq}\Big\rfloor$, which follows from
	the primality of $p,q$.
	
	Now let us assume first  that all $\theta_p$'s have the same sign. Then using the inequality
	$$
	\Big\lfloor\frac{N}{pq}\Big\rfloor - \frac{1}{N}\Big\lfloor\frac{N}{p}\Big\rfloor \cdot \Big\lfloor\frac{N}{q}\Big\rfloor \leq \frac{1}{p}+\frac{1}{q},
	$$
	we get for all $p\neq q$
	$$
	\Cov_N(X_p,X_q) \leq \frac{\theta_p\, \theta_q}{N}\Big(\frac{1}{p}+\frac{1}{q}\Big).
	$$
	
	Using these estimates and the identity
	$$
	\Var_N(S_N)=\sum_{p\in \P\cap[N]} \Var_N(X_p)+\sum_{p,q\in \P\cap [N], p\neq q} \Cov_N(X_p\cdot X_q),
	$$
	we deduce that
	$$
	\Var_N(S_N)\leq  \sum_{p\in \P\cap[N]} \frac{\theta_p^2}{p}+
	\frac{1}{N}\sum_{p,q\in \P\cap [N], p\neq q} \theta_p\, \theta_q\Big(\frac{1}{p}+\frac{1}{q}\Big).
	$$
	Since $\theta_p\in [-1/2,1/2)$, $p\in \P$, we can bound the second sum by $\sum_{p,q\in \P\cap [N], p\neq q}\frac{1}{p}$, which in turn is bounded by $C_1 \, \frac{N}{\log{N}}\,  \log\log{N}$ for some universal constant $C_1$. Hence,
	\begin{equation}\label{E:Var2}
		\Var_N(S_N)\leq  \sum_{p\in \P\cap[N]} \frac{\theta_p^2}{p}+C_1\, \frac{\log\log{N}}{\log{N}}.
	\end{equation}
	The bound \eqref{E:ANest}, with $a(N)$  in place of $A(N)$,  then follows by combining \eqref{E:Var1} and  \eqref{E:Var2}.
	To get  \eqref{E:ANest}, as stated,  note that by~\eqref{srednia}
$$
		a(N)=\sum_{p\in \P\cap [N]}  \frac{\theta_p}{N}\, \Big\lfloor\frac{N}{p}\Big\rfloor
$$
	and
	$$
	|a(N)-A(N)|\leq \frac{C_2}{\log{N}}
	$$
	for some universal constant $C_2>0$.
	
	This completes the proof in the case where all the $\theta_p$'s have the same sign.
	
	In the general case, we can decompose $f$ as a product $f_+\cdot f_-$, where $f_+$ and $f_-$ are multiplicative functions defined by
	$$ f_+(p^s) := \begin{cases}
                    e(\theta_p),\quad  & \text{if }\theta_p\geq0\\
	                1,\quad   &\text{otherwise}
	              \end{cases}
    \quad\text{ and }\quad
    f_-(p^s) := \begin{cases}
                    e(\theta_p),\quad  &\text{if }\theta_p<0\\
	                1, \quad  &\text{otherwise,}
	              \end{cases}
    $$
	for every $p\in\P$ and every $s\in\N$. Now, let
	$$
	A_+(N):=\sum_{\substack{p\in \P\cap [N],\\ \theta_p\geq0}}  \frac{\theta_p}{p}
	\quad\text{ and }\quad
	A_-(N):=\sum_{\substack{p\in \P\cap [N],\\ \theta_p<0}}  \frac{\theta_p}{p}.
	$$
	Then, we have for all $n\in[N]$
	\begin{align*}
	|f(n)-e(A(N))| &= |f_+(n)\cdot f_-(n)-e(A_+(N))\cdot e(A_-(N))| \\
	&\leq |f_+(n)-e(A_+(N))| + |f_-(n)-e(A_-(N))|.
	\end{align*}
	It follows   that
	$$
	\E_N |f(n)-e(A(N))|^2 \leq 2\cdot \left(\E_N |f_+(n)-e(A_+(N))|^2 +
	\E_N |f_-(n)-e(A_-(N))|^2\right),
	$$
	and we get the conclusion by applying the preceding analysis to $f_+$ and $f_-$.
\end{proof}

	 The next estimate will be key for us.
	 	\begin{lemma}\label{L:thetap}
	 		Let  $f\colon \N\to \S^1$ be a multiplicative function  such that $f\sim 1$. Suppose that  for every $p\in \P$ we have  $f(p^s):=e(\theta_p)$ for some $\theta_p\in [-1/2,1/2)$ and all  $s\in \N$.  Then there exist a slowly varying  sequence  $A\colon \N\to \R_+$ (see Definition~\ref{D:slowgrow}), defined as in   \eqref{E:AN},  and a universal constant $C>0$, such that
	 		$$
	 		\limsup_{N\to \infty}	\E_{n\in [N]}|f(n)-e(A(N))|^2\leq  C \,	\sum_{p \in \P} \frac{\theta_p^2}{p}.
	 		$$	
	 		Furthermore, the last estimate also holds with $\lE_{n\in[N]}$ in place of $\E_{n\in[N]}$ (without changing $A(N)$).
	 	\end{lemma}
	 	\begin{remark}
	 	Recall that, by Lemma~\ref{L:pt1}, we have 	$\sum_{p\in \P}\frac{\theta_p^2}{p}<+\infty$. The lemma will be used when $\sum_{p \in \P} \frac{\theta_p^2}{p}$ is small,  to deduce that  $\limsup_{N\to \infty}	\E_{n\in [N]}|f(n)-e(A(N))|^2$ is small.
	 	\end{remark}
	 	\begin{proof}
	 		We first work with   Ces\`aro averages.	In this case, the estimate follows immediately from Lemma~\ref{L:thetap'} and it remains to show that the sequence in \eqref{E:AN} is slowly varying.
	 		So let $c\in (0,1)$.  Using~\eqref{E:AN} and the  Cauchy-Schwarz inequality, we get
	 		$$
	 		\sup_{n\in [N^c,N]}|A(n)-A(N)|\leq \sum_{p\in \P\cap [N^c,N]}\frac{|\theta_p|}{p}\leq (B_N\cdot C_N)^{\frac{1}{2}},
	 		$$
	 		where
	 		$$
	 		B_N:=\sum_{p\in \P\cap [N^c,N]}\frac{\theta^2_p}{p}, \quad   C_N:=\sum_{p\in \P\cap [N^c,N]}\frac{1}{p}.
	 		$$
	 	We have 	$\sum_{p\in \P}\frac{\theta_p^2}{p}<+\infty$, and this implies  $\lim_{N\to\infty}B_N=0$. We also have $\lim_{N\to\infty}C_N=\log(1/c)$, hence $\lim_{N\to\infty}(B_N\cdot C_N)=0$.
	 		
	 		Lastly, we deal with logarithmic averages. Let $c\in (0,1)$.
	 		Since the sequence $A(N)$ is slowly varying, it can be treated as a constant on intervals $[N^c,N]$ for $N$ large enough, and   using partial summation, we deduce
	 		$$
	 		\limsup_{N\to \infty}	\lE_{n\in [N^c,N]}|f(n)-e(A(N))|^2\leq  C \,	\sum_{p \in \P} \frac{\theta_p^2}{p}.
	 		$$	
	 		Letting $c\to 0^+$, we get the desired estimate.
	 	\end{proof}
	 	
	 \subsection{A decomposition}
	If a multiplicative function pretends to be a Dirichlet character, then  it is possible to decompose it into a product
	  of two terms, one is a multiplicative function for which Theorem~\ref{T:RAP} is satisfied and the other is a multiplicative function that   is approximately constant in density (Lemma~\ref{L:thetap} will allow us to conclude this).
	 \begin{lemma}\label{L:12e}
	 	Let $f\colon \N\to \U$ be a  multiplicative function such that $f\sim \chi$ for some primitive Dirichlet character $\chi$. Then  for every $\varepsilon>0$   there exists a slowly varying sequence $A_\varepsilon\colon \N\to \R_+$  such that the following holds: We can decompose $f$ as
	 	\begin{equation}\label{E:f12}
	 	f=f_{1,\varepsilon}\cdot f_{2,\varepsilon},
	 	\end{equation}
	 	where
	 	\begin{enumerate}
	 		\item\label{I:f1}
	 	 $f_{1,\varepsilon}\colon \N\to \U$ is a multiplicative function such that
	 	\begin{equation}\label{E:1-f}
	 	\sum_{p\in \P}\frac{1}{p}(1-f_{1,\varepsilon }(p)\cdot \overline{\chi(p)})\quad  \text{converges}  \, ;
	 	\end{equation}
\item\label{I:f2} $f_{2,\varepsilon}\colon \N\to \S^1$  is a multiplicative function 	such that
\begin{equation}\label{E:aNe}
	\limsup_{N\to\infty}\E_{n\in [N]}|f_{2,\varepsilon}(n)-e(A_\varepsilon(N))|^2\leq \varepsilon.
\end{equation}
	 	\end{enumerate}
  Furthermore, 	if  $f\sim 1$ and for some $p_0\in \P$ we have $f(p_0^s)=1$ for every $s\in \N$, then  we can also ensure that $f_{1,\varepsilon}$  in \eqref{I:f1} satisfies $f_{1,\varepsilon}(p_0^s)=1$ for every $s\in \N$.\footnote{This will be used  in the proof of the implication $\eqref{I:spectrum1}\implies \eqref{I:spectrum2}$ of Theorem~\ref{T:spectrum} in Section~\ref{SS:spectrumproof1}.}
 	
 Lastly,  in \eqref{I:f2} we can replace $\E_{n\in [N]}$ with $\lE_{n\in [N]}$ (without changing $A_\varepsilon(N)$) .
	 \end{lemma}
 \begin{remark}
 	If $f\sim 1$, then we can take  $A_\varepsilon(N):=\sum_{p\in \P\cap [P_\varepsilon,N]}\frac{\theta_p}{p}$,
 		where $P_\varepsilon$ satisfies $\sum_{p\in \P, p>P_\varepsilon}\frac{\theta_p^2}{p}\leq \frac{\varepsilon}{C}$ and $C$ is the universal constant defined in Lemma~\ref{L:thetap}.
 \end{remark}
 \begin{proof}
 	We give the argument for Ces\`aro averages, the proof is similar for logarithmic averages since we can use the variant of Lemma~\ref{L:thetap} that covers logarithmic averages.
 	
 	Suppose first that $f\sim 1$.
 	Let $\varepsilon>0$. We can write $f(p)=r_p \, e(\theta_p)$ for some $r_p\in [0,1]$ and $\theta_p\in [-1/2,1/2)$, $p\in \P$.
 	Since $f\sim 1$, we have by Lemma~\ref{L:pt1} that
 	\begin{equation}\label{E:rpp}
 	\sum_{p\in\P} \frac{1-r_p}{p}<+\infty \quad \text{and} \quad \sum_{p\in\P}\frac{\theta_p^2}{p}<+\infty.
 	\end{equation}
 	Let $P_\varepsilon\in \N$ be  such that
 	\begin{equation}\label{E:thetap2}
 	\sum_{p\in\P, p>P_\varepsilon}\frac{\theta_p^2}{p}<\frac{\varepsilon}{C},
 	\end{equation}
  where $C>0$ is the universal constant defined in Lemma~\ref{L:thetap}.
 	 We  define the multiplicative function $f_{1,\varepsilon}\colon \N\to \U$ by
 	\begin{equation}\label{E:f1e}
 	f_{1,\varepsilon}(p^s):=\begin{cases} f(p^s),\quad  &\text{for } \ p\leq P_\varepsilon,\,  s\in \N  \\
 		f(p^s)\, e(-\theta_p), \quad  &  \text{for } \   p>P_\varepsilon, \, s\in \N,
 	\end{cases}
	\end{equation}
 	and the multiplicative function $f_{2,\varepsilon}\colon \N\to \S^1$ by
 	 $$
 	f_{2,\varepsilon}(p^s):=\begin{cases} 1, \quad  &\text{for } \  p\leq P_\varepsilon, \, s\in \N  \\
 		 e(\theta_p), \quad  &  \text{for }   \  p>P_\varepsilon, \, s\in \N.
 	\end{cases}
 	$$
 Then \eqref{E:f12} is clearly satisfied and \eqref{E:rpp} implies that   \eqref{E:1-f} holds (with $\chi=1$).

  It remains to verify  \eqref{E:aNe}.
If we set $\theta_p:=0$ for all primes $p\leq P_\varepsilon$,  then we have $f_{2,\varepsilon}(p^s)=e(\theta_p)$ for all primes $p\in \P$ and $s\in \N$, and by  \eqref{E:thetap2},  we have
 \begin{equation}\label{E:thetap}
 	\sum_{p\in\P}\frac{\theta_p^2}{p} \leq \frac{\varepsilon}{C}.
 \end{equation}
   Combining  this with  Lemma~\ref{L:thetap},   we get   that there exists a slowly varying sequence $A_\varepsilon\colon \N\to [0,1)$
    such that   \eqref{E:aNe} holds.

 We consider now the general case, where $f\sim \chi$ for some primitive Dirichlet character $\chi$.
  Let $\tilde{\chi}\colon \N\to \U$ be the multiplicative function defined by
 $$
 \tilde{\chi}(p^s):=\begin{cases}
 	\chi(p^s), \quad &\text{if } \, \chi(p^s)\neq 0\\
 	1, \quad &\text{if } \, \chi(p^s)= 0.
 	 \end{cases}
 $$
 We introduce this variant of $\chi$ because it can be inverted.
 Since $\tilde{\chi}(p)=\chi(p)$ for all but finitely many primes $p$, we have
  $\tilde{f}:=f\cdot \overline{\tilde{\chi}}\sim 1$.  The previous case gives a decomposition
  $$
  \tilde{f}=\tilde{f}_{1,\varepsilon}\cdot \tilde{f}_{2,\varepsilon},
  $$
  where  $\tilde{f}_{1,\varepsilon}$ satisfies property~\eqref{I:f1} with $\chi=1$ and $\tilde{f}_{2,\varepsilon}$ satisfies property~\eqref{I:f2} for some slowly varying sequence $A_\varepsilon\colon \N\to \R_+$. It follows that
  $$
  f=f_{1,\varepsilon} \cdot f_{2,\varepsilon},
  $$
  where $f_{1,\varepsilon}:=\tilde{f}_{1,\varepsilon}\cdot \tilde{\chi}$ and $ f_{2,\varepsilon}:=\tilde{f}_{2,\varepsilon}$.  Then $f_{2,\varepsilon}$ clearly satisfies property~\eqref{I:f2}.   Also $f_{1,\varepsilon}(p)\cdot \overline{\chi(p)}=\tilde{f}_{1,\varepsilon}(p)$ for all but finitely many primes $p$, hence
  $f_{1,\varepsilon}$ satisfies property~\eqref{I:f1} for this $\chi$.

   Lastly, suppose that $f\sim 1$ and for  some $p_0\in \P$ we have $f(p_0^s)=1$ for every $s\in \N$. Then $e(\theta_{p_0})=1$ and by \eqref{E:f1e} we 
  have  $f_{1,\varepsilon}(p_0^s)=1$ for every $s\in \N$.
  This completes the proof.
 \end{proof}

	\subsection{Besicovitch rational almost periodicity along a subsequence} The goal of this subsection is to show that if a  multiplicative function   pretends to be a Dirichlet character, then it
	has strong rational almost periodicity properties, in the sense described in Propositions~\ref{P:RAPchi} and \ref{P:RAPchi'}.
	
For the purpose of studying Furstenberg systems of pretentious multiplicative functions,
we will only use Proposition~\ref{P:RAPchi}, which is better adapted to our needs.
 \begin{proposition}\label{P:RAPchi}
	Let $f\colon \N\to \U$ be a  multiplicative function such that $f\sim  \chi$ for some primitive Dirichlet character $\chi$.  Then there exist periodic sequences $f_m\colon \N\to \U$, $m\in\N$,  such that every  sequence  $N_k\to \infty$ has a subsequence $N_k'\to \infty$ for which the following holds:  For every $m\in\N$  there exists $\alpha_m\in[0,1)$ such that
	\begin{equation}\label{E:logCes}
		\limsup_{k\to\infty} \big(\E_{n\in [N_k']}|f(n)-e(\alpha_m)\cdot  f_m(n)|^2  +\lE_{n\in [N_k']}|f(n)-e(\alpha_m)\cdot  f_m(n)|^2\big)\leq 1/m.
	\end{equation}
In particular, $f$ is Besicovitch rationally almost periodic for Ces\`aro and logarithmic averages along $(N_k')$.
\end{proposition}
\begin{remark}
	We stress  that although the choices of $\alpha_m$ depend on the subsequence $(N_k')$, the choice of the periodic sequences $f_m$ depend only on $f$ and on $m$. Also, one could use in place of the periodic sequences  $f_m$ the multiplicative functions $f_{1,1/m}$ defined  in Lemma~\ref{L:12e} (for $\varepsilon:=1/m$). These functions  satisfy \eqref{E:1-f}, and moreover, 
	  	if for some $p_0\in \P$ we have $f(p_0^s)=1$ for every $s\in \N$, then  we can also ensure that $f_m$   satisfies $f_m(p_0^s)=1$ for every $s\in \N$.
\end{remark}
\begin{proof}
	  For $\varepsilon>0$, let $f=f_{1,\varepsilon}\cdot f_{2,\varepsilon}$ be the decomposition given by Lemma~\ref{L:12e}    for the slowly varying sequence $(A_\varepsilon(N))$.  Using a diagonal argument we can find a subsequence $(N_k')$ of $(N_k)$ such that the  limit
$\lim_{k\to\infty} e(A_{1/m}(N_k'))$ exists for every $m\in \N$. Let $\alpha_m\in [0,1)$ be such that
$$
e(\alpha_m)=\lim_{k\to\infty} e(A_{1/m}(N_k')), \quad m\in \N.
$$
Then  \eqref{E:aNe} implies that
$$
	\limsup_{k\to\infty}\big(\E_{n\in [N'_k]}|f_{2,1/m}(n)-e(\alpha_m)|^2+ \lE_{n\in [N'_k]}|f_{2,1/m}(n)-e(\alpha_m)|^2\big)
	\leq \frac{2}{m}
$$
for every $m\in\N$.
Since  	$f=f_{1,1/m}\cdot f_{2,1/m}$ and  $f_{1,1/m}$ is $1$-bounded,  we deduce that
$$
\limsup_{k\to\infty} \big(\E_{n\in [N_k']}|f(n)-e(\alpha_m)\cdot f_{1,1/m}(n)|^2+
 \lE_{n\in [N_k']}|f(n)-e(\alpha_m)\cdot f_{1,1/m}(n)|^2\big)\leq \frac{2}{m}
$$
for every $m\in \N$.
 The sequences $f_{1,1/m}$  satisfy property~\eqref{E:1-f}, hence they are Besicovitch rationally almost periodic for Ces\`aro averages by Theorem~\ref{T:RAP}.  It follows that  there exist periodic sequences $f_m\colon \N\to \U$,
such that
$$
\limsup_{N\to\infty}\big(\E_{n\in[N]} |f_m(n)-f_{1,1/m}(n)|^2+\lE_{n\in[N]} |f_m(n)-f_{1,1/m}(n)|^2\big)\leq \frac{1}{m}.
$$
Combining the above, we get the asserted estimate with a multiple of $1/m$ in place of $1/m$.
Lastly, note that the choices of $f_{1,1/m}$ and  $f_m$ do not depend on the subsequence $N_k'$ but only on $f$ and on $m$.
This completes the proof.
\end{proof}

Next, we give a  similar  result that  avoids passing to subsequences.   It is of
  independent interest and not needed for any other result in this article.
 \begin{proposition}\label{P:RAPchi'}
	Let $f\colon \N\to \U$ be a   multiplicative function such that $f\sim \chi$
	for some primitive Dirichlet character $\chi$  and $N_k\to \infty$ be a sequence of integers  such that the limit
	$$
	\lim_{k\to\infty}\E_{n\in [N_k]}\, f(n)\cdot \overline{\chi(n)}\quad \text{exists}.
	$$
	Suppose that $f(2^s)\cdot \overline{\chi(2^s)}\neq -1$ for some $s\in \N$.\footnote{The condition is always satisfied  if $f$ is completely multiplicative.}
	Then, for every $\varepsilon>0$, there exists $\alpha_\varepsilon \in [0,1)$ such that
	$$
	\limsup_{k\to\infty} \E_{n\in [N_k]}|f(n)-e(\alpha_\varepsilon)\cdot f_{\varepsilon}(n)|^2\leq \varepsilon
	$$
	for some periodic sequence $f_{\varepsilon}\colon \N\to \U$ that depends only on $f$ and $\varepsilon$ and not on the subsequence $N_k$.
 Furthermore, a similar statement holds if in the assumption and the conclusion we replace $\E_{n\in [N_k]}$ with  $\lE_{n\in [N_k]}$.
\end{proposition}
\begin{proof}
	We first give the argument for Ces\`aro averages.
	For $\varepsilon>0$, let $f=f_{1,\varepsilon}\cdot f_{2,\varepsilon}$ be the decomposition given by Lemma~\ref{L:12e} for the  slowly varying sequence $A_\varepsilon\colon \N\to \R_+$ that satisfies
	\begin{equation}\label{E:aNe'}
		\limsup_{N\to\infty}\E_{n\in [N]}|f_{2,\varepsilon}(n)-e(A_\varepsilon(N))|^2\leq \varepsilon.
	\end{equation}
	Using that 	$f_{1,\varepsilon}$ and $\chi$ are 1-bounded, we deduce
	$$
	\limsup_{k\to\infty} |\E_{n\in [N_k]}\, f(n)\cdot \overline{\chi(n)} -e(A_\varepsilon(N_k))\cdot \E_{n\in[N_k]}\, f_{1,\varepsilon}(n)\cdot\overline{\chi(n)}|^2\leq \varepsilon.
	$$
	Let $L:=\lim_{k\to\infty}\E_{n\in [N_k]}\, f(n)\cdot \overline{\chi(n)}$ and $L_\varepsilon:=\lim_{k\to\infty}\E_{n\in [N_k]}\, f_{1,\varepsilon}(n)\cdot \overline{\chi(n)}$; the first limit exists by our assumption, and the second because $f_{1,\varepsilon} $ satisfies \eqref{E:1-f}, hence Theorem~\ref{T:RAP} is applicable. Furthermore, since $f\cdot \overline{\chi}\sim 1$  and we do not have $f(2^s)\cdot \overline{\chi(2^s)}=-1$ for every $s\in \N$, Theorem~\ref{T:Halasz} gives that $L\neq 0$. It follows  that
	$$
	\limsup_{k\to\infty}|e(-A_\varepsilon(N_k))-L_\varepsilon/L|^2\leq \varepsilon/|L|^2.
	$$
	Hence (since $|e(-A_\varepsilon(N_k))|=1$), there exists $\alpha_\varepsilon\in [0,1)$ such that
	$$
	\limsup_{k\to\infty}|e(A_\varepsilon(N_k))-e(\alpha_\varepsilon)|^2\leq 4 \varepsilon/|L|^2.
	$$
	Using this, the identity  $f=f_{1,\varepsilon}\cdot f_{2,\varepsilon}$, the fact that $|L|\leq1$, and again the estimate \eqref{E:aNe'}, we get
	$$
	\limsup_{k\to\infty}\E_{n\in [N_k]}|f(n)-e(\alpha_\varepsilon) \cdot f_{1,\varepsilon}(n)|^2\leq 10 \varepsilon/|L|^2.
	$$
	 The sequences $f_{1,\varepsilon}$  satisfy property~\eqref{E:1-f}, hence they are Besicovitch rationally almost periodic for Ces\`aro averages by Theorem~\ref{T:RAP}.  The result follows easily from this.
	
	The proof is similar for logarithmic averages  since we can use the variant of Lemma~\ref{L:12e} that covers  logarithmic averages  and also  in the case of logarithmic averages we get that $L\neq 0$ by the variant of Theorem~\ref{T:Halasz} that covers logarithmic averages (see the second remark following the theorem).
\end{proof}

	\section{Structural results for  pretentious multiplicative functions}\label{S:StructurePretentious}
The goal of this section is to prove Theorem~\ref{T:StructurePretentiousRefined}, which is the main structural result for Furstenberg systems of pretentious multiplicative functions given in Section~\ref{SSS:StructurePretentious}.
In Section~\ref{SS:Ces=log}, we do some preparatory work that enables us for  multiplicative functions that pretend to be Dirichlet characters to work interchangeably with Furstenberg systems defined using   Ces\`aro or logarithmic averages. In Section~\ref{SS:isomorphic}, we show that certain correlation limits of  multiplicative functions that pretend to be Dirichlet characters exist, and conclude that all Furstenberg systems of such multiplicative functions for Ces\'aro or logarithmic averages  are isomorphic (a property that we will later extend   to all pretentious multiplicative functions). In Section~\ref{SS:ProofPretentious}, we use these preparatory results in conjunction with other ergodic considerations to conclude the proof of Theorem~\ref{T:StructurePretentiousRefined}.
	\subsection{Ces\`aro and logarithmic correlations agree when $f\sim \chi$}\label{SS:Ces=log}
The goal of this subsection is to show in Proposition~\ref{P:ceslog'} that  if a  multiplicative function  pretends to be a Dirichlet character, then we can use interchangeably Ces\`aro and logarithmic averages in the definition of its Furstenberg systems. This is a consequence of the following correlation identity.
\begin{proposition}\label{P:ceslog}
	Let $f\colon \N\to \U$ be a multiplicative function such that  $f\sim \chi$ for some primitive Dirichlet character $\chi$. Then
	\begin{equation}\label{E:corrF}
		\lim_{N\to\infty}\Big| \E_{n\in [N]}\, \prod_{j=1}^\ell f^{k_j}(n+n_j)-\lE_{n\in [N]}\, \prod_{j=1}^\ell f^{k_j}(n+n_j)\Big|=0
	\end{equation}
	for all $\ell\in \N$, $n_1,\ldots, n_\ell, k_1,\ldots, k_\ell\in \Z$, where we let  $f^k:=\overline{f}^{|k|}$ for $k<0$.
\end{proposition}
\begin{remark}
	The result fails if $f=n^{it}$, $t\neq 0$,  and $\ell=1$, $k_1=1$.  In this case, the logarithmic averages converge to $0$
	but the Ces\`aro averages behave like $N^{it}/(1+it)$.  It can actually be seen that it fails for all completely multiplicative functions that satisfy $f\sim n^{it}$ for some $t\neq 0$.
\end{remark}
\begin{proof}
	Arguing by contradiction, suppose that the conclusion fails. Then for some choice of
	$\ell\in \N$, $n_1,\ldots, n_\ell, k_1,\ldots, k_\ell\in \Z,$ there exists $N_k\to\infty$ along which the difference of the  averages in \eqref{E:corrF} is bounded away from zero. By  Proposition~\ref{P:RAPchi} there exists a subsequence $(N_k')$ of $(N_k)$,  such that for every $m\in\N$ we have
	\begin{equation}\label{E:approxf}
		\limsup_{k\to\infty} \big(\E_{n\in [N_k']}|f(n)- g_m(n)|^2  +\lE_{n\in [N_k']}|f(n)- g_m(n)|^2\big)\leq 1/m
	\end{equation}
	for some periodic sequence $g_m\colon \N\to \U$.  Since  \eqref{E:corrF} clearly holds
	when we replace $f$ with the periodic sequence $g_m$, using \eqref{E:approxf}, we deduce that \eqref{E:corrF}
	also holds for $f$ as long as we average over the sequence of intervals $[N_k']$. This is a contradiction, since we have assumed that along  the sequence $(N_k)$  the difference of the averages in \eqref{E:corrF}
	is bounded away from zero, and $(N_k')$ is a subsequence of $(N_k)$.
\end{proof}
We immediately deduce the following. 	
\begin{proposition}\label{P:ceslog'}
	Let $f\colon \N\to \U$ be a multiplicative function such that $f\sim\chi$ for some primitive Dirichlet character $\chi$. Let also $N_k\to\infty$. Then the Furstenberg system of $f$ for Ces\`aro averages along $(N_k)$ is well defined if and only if it is well defined for logarithmic averages along $(N_k)$, and the two Furstenberg systems are equal (meaning the corresponding $T$-invariant measures agree).
\end{proposition}
\begin{proof}
	Suppose that the Furstenberg system of $f$ for Ces\`aro averages is defined along the sequence  $(N_k)$ (the argument is similar if we assume it is defined for logarithmic averages).
	Then Proposition~\ref{P:ceslog} implies that the Furstenberg system for logarithmic averages is also well defined along $(N_k)$ and the  measures that define  the two Furstenberg systems coincide, since they agree on all cylinder functions defined in the remark after Definition~\ref{D:sequencespace}.
\end{proof}

\subsection{Correlation limits and isomorphism  when $f\sim \chi$} \label{SS:isomorphic}
The goal of this subsection is to show  in Proposition~\ref{P:isomorphic} that
 different Furstenberg systems of a fixed multiplicative function   that pretends to be a Dirichlet character  are isomorphic. This is a consequence of the correlation identities stated in the following lemma.
\begin{lemma}\label{L:ceslog}
	Let $f\colon \N\to \U$ be a multiplicative function such  that $f\sim  \chi$ for some primitive  Dirichlet character $\chi$.
	If $N_{k,1}\to \infty$ and $N_{k,2}\to\infty$ are sequences, then there exist subsequences
	$(N_{k,1}')$ and $(N_{k,2}')$ and $\alpha\in [0,1)$, such that the following holds:
		For all $\ell\in \N$ and $n_1,\ldots, n_\ell, k_1,\ldots, k_\ell\in \Z$, both   correlation limits below exist and we have
	\begin{equation}\label{E:alpha12}
	\lim_{k\to\infty}\E_{n\in[N_{k,1}']}\,  \prod_{j=1}^\ell(e(\alpha)\cdot f)^{k_j}(n+n_j)
	= \lim_{k\to\infty}\E_{n\in[N_{k,2}']} \,  \prod_{j=1}^\ell f^{k_j}(n+n_j),
	\end{equation}
 where we let  $f^k:=\overline{f}^{|k|}$ for $k<0$.
 A similar identity also holds for logarithmic averages.
\end{lemma}
\begin{remark}
More generally, if $f\sim n^{it}\cdot \chi$ for some $t\in \R$ and primitive Dirichlet character $\chi$, then we get a similar result with the identity \eqref{E:alpha12} replaced by the identity
	\begin{equation}\label{E:alpha12'}
	\lim_{k\to\infty} \E_{n\in[N_{k,1}']}\, n^{it'}\cdot \prod_{j=1}^\ell(e(\alpha)\cdot f)^{k_j}(n+n_j)
	= \lim_{k\to\infty} \E_{n\in[N_{k,2}']} \, n^{it'}\cdot \prod_{j=1}^\ell f^{k_j}(n+n_j),
\end{equation}
where $t':=-t\cdot \sum_{j=1}^\ell k_j$.
To see this,  use  \eqref{E:alpha12}  for the multiplicative function $f\cdot n^{-it} \sim \chi$ together with the fact that   $\lim_{n\to\infty}((n+h)^{it}-n^{it})=0$ for every $h\in \N$.
\end{remark}
\begin{proof}
	We give the argument for Ces\`aro averages, a similar argument applies to logarithmic averages.
	
 By Proposition~\ref{P:RAPchi} (and the remark following it), for every $m\in\N$  there exist
		periodic sequences  $f_m$,
		subsequences $(N'_{k,1})$ of $(N_{k,1})$ and $(N_{k,2}')$ of $(N_{k,2})$,
		and
		$\alpha_{1,m}, \alpha_{2,m}\in[0,1)$, such that
			\begin{equation}\label{E:ceslog3'}
			\limsup_{k\to\infty} \E_{n\in [N_{k,1}']}|f(n)-e(\alpha_{1,m})\cdot  f_m(n)|^2   \leq 1/m
		\end{equation}
		and
		\begin{equation}\label{E:ceslog4'}
			\limsup_{k\to\infty} \E_{n\in [N_{k,2}']}|f(n)-e(\alpha_{2,m})\cdot f_m(n)|^2   \leq 1/m.
		\end{equation}
	(Note that the function $f_m$ is the same in both cases.)
		By working with further subsequences of $(N'_{k,1})$ and $(N_{k,2}')$ along which
	 the sequences	$e(\alpha_{1,m})$ and $e(\alpha_{2,m})$ converge,
		 we can assume that \eqref{E:ceslog3'} and \eqref{E:ceslog4'}
 hold with  $e(\alpha_1)$ in place of $e(\alpha_{1,m})$ and $e(\alpha_2)$ in place of $e(\alpha_{2,m})$ for some $\alpha_1,\alpha_2 \in [0,1)$.   	We let
\begin{equation}\label{E:alpha}
 \alpha:=\alpha_2-\alpha_1.
 \end{equation}
		
		Using a diagonal argument, we can find further subsequences that we denote again by $(N_{k,1}')$ and $(N_{k,2}')$ such that all the correlation limits in \eqref{E:alpha12} exist.  Using \eqref{E:ceslog3'}
		and \eqref{E:ceslog4'}, we see that
it suffices to verify 		\eqref{E:alpha12} with  $f(n)$ replaced by $e(\alpha_{1})\cdot  f_m(n)$  on the first average and  by
$e(\alpha_{2})\cdot f_m(n)$ on the second average.  Taking into account \eqref{E:alpha}, it suffices to verify that
	 $$
	 \lim_{k\to\infty}\Big|\E_{n\in[N_{k,1}']}\,  \prod_{j=1}^\ell f_m^{k_j}(n+n_j)
	 -	\E_{n\in[N_{k,2}']}\,  \prod_{j=1}^\ell  f_m^{k_j}(n+n_j)\Big|=0
	 $$
for 	 all $\ell\in \N$ and $n_1,\ldots, n_\ell, k_1,\ldots, k_\ell\in \Z$.
	Since the sequence $f_m$ is periodic, both averages have limits and these limits coincide, therefore, the asserted identity follows.
\end{proof}

\begin{proposition}\label{P:isomorphic}
	Let $f\colon \N\to \U$ be a  multiplicative function such that $f\sim \chi$ for some primitive Dirichlet character $\chi$.
 Then any two Furstenberg systems of $f$ for Ces\`aro or logarithmic  averages are isomorphic.
\end{proposition}
\begin{remark}
With substantial additional effort, the result will be  extended  to arbitrary pretentious multiplicative functions in the next subsection, where we prove part~\eqref{I:StructurePretentious2.5} of Theorem~\ref{T:StructurePretentiousRefined}.
	\end{remark}
\begin{proof}
	By Proposition~\ref{P:ceslog'}, we have that Furstenberg systems for Ces\`aro and logarithmic averages coincide, hence it suffices to treat the case of Ces\`aro averages.
Let $(X,\mu,T)$ and $(X,\mu',T)$ be two Furstenberg systems for Ces\`aro averages taken along the sequences $N_{k,1}\to \infty$ and $N_{k,2}\to \infty$, respectively. 
Using  Lemma~\ref{L:ceslog} and  \eqref{E:correspondence} to translate \eqref{E:alpha12} into identities for the two  Furstenberg systems, we get that there exists $\alpha\in [0,1)$ such that
\begin{equation}\label{E:alphaF0}
\int \prod_{j=1}^\ell T^{n_j}(e(\alpha)\cdot F_0)^{k_j}\, d\mu=\int \prod_{j=1}^\ell T^{n_j} F_0^{k_j}\, d\mu'
\end{equation}
holds for all $\ell\in \N$ and $n_1,\ldots, n_\ell, k_1,\ldots, k_\ell\in \Z$, where as usual $F_0$ is the $0^{\text{th}}$-coordinate projection. This easily implies that the two systems are isomorphic with the isomorphism given by the map $M_{e(\alpha)}\colon X\to X$ defined by~\eqref{E:Mz}. Indeed,  $M_{e(\alpha)}$ is bijective, satisfies the commutation relation~\eqref{E:TMz}, and since  $F_0\circ M_{e(\alpha)}=e(\alpha)\cdot F_0$, equation \eqref{E:alphaF0} implies that $\mu'= (M_{e(\alpha)})_\ast\mu$, (the two measures agree on the linearly dense subset of $C(X)$ given by the cylinder functions, see remark after Definition~\ref{D:sequencespace}).
\end{proof}

 \subsection{Proof of Theorems~\ref{T:StructurePretentious}-\ref{T:StructurePretentiousRefined}}\label{SS:ProofPretentious}
Theorem~\ref{T:StructurePretentious} is an immediate consequence of
Theorem~\ref{T:StructurePretentiousRefined}, so we focus on the proof of the latter.

\subsubsection{Proof of part~\eqref{I:StructurePretentious1}}
Let $f\sim \chi$ for some primitive Dirichlet character $\chi$.
Suppose that the  Furstenberg system is
 defined using Ces\`aro averages along the sequence $(N_k)$.
By  Proposition~\ref{P:RAPchi}, there exists a subsequence $(N_k')$ of $(N_k)$ such that $f$ is Besicovitch rationally almost periodic for Ces\`aro averages along $(N_k')$. By Theorem~\ref{T:BRAP}, the system $(X,\mu,T)$ is an ergodic procyclic system.

 It remains to prove that if $f\neq 1$, then the system is non-trivial.
Suppose first that $f\sim 1$. Since $f\neq 1$, there exist $p\in \P$ and $s\in \N$ such that $f(p^s)\neq 1$. Then Theorem~\ref{T:spectrum} applies and gives that $1/p\in \spec(X,\mu,T)$. 
On the other hand, if $f\not\sim 1$, then the conductor $q$ of the primitive Dirichlet character $\chi$  satisfies $q>1$. Then for every $p\in \P$ with $p\mid q$ we have by  Theorem~\ref{T:spectrum} that $1/p\in \spec(X,\mu,T)$.
In either case we get that the  system is non-trivial.

A similar argument applies to  logarithmic averages since Proposition~\ref{P:RAPchi},
Theorem~\ref{T:BRAP}, and Theorem~\ref{T:spectrum},  also  apply to logarithmic averages. Alternatively, we can use Proposition~\ref{P:isomorphic}, which implies that in this case, Furstenberg systems taken with respect to Ces\`aro and logarithmic averages coincide.

\subsubsection{Proof of part~\eqref{I:StructurePretentious2}}
 We give all the arguments for Ces\`aro averages, the proof carries over to logarithmic averages, since we have proved variants for  all the results needed for logarithmic averages.

  Before embarking into the proof of the various claims, we do some preparatory work.  Let $f\sim  n^{it}\cdot \chi $ for some $t\neq 0$ and some primitive Dirichlet character $\chi$, and suppose that $(X,\mu,T)$ is a Furstenberg system of $f$ for Ces\`aro  averages,  defined along the sequence $(N_k)$.
 Then $f=\tilde{f}\cdot n^{it}$, where $\tilde{f}:=f\cdot n^{-it}\sim \chi$. Upon passing to a subsequence $(N_k')$ of $(N_k)$ we get that the Furstenberg systems of the sequences $\tilde{f}$ and
  $n^{it}$, taken along $(N_k')$, are both well defined. We denote the Furstenberg system of
   $\tilde{f}$ along $(N_k')$ by $(Y,\nu,S)$ ($Y:=\U^Z$ and $S$ is again the shift map) and recall that, by part~\eqref{I:StructurePretentious1},
   it is isomorphic to an ergodic procyclic system. Furthermore, by \cite[Corollary~5.5]{GLR21}, for $t\neq 0$, the Furstenberg system of $n^{it}$ along $(N_k')$  defines a continuous measure  on $X$ that is supported on the
   diagonal set  $\{(z)_{k\in\Z}\colon z\in \S^1\}$, on which $T$ acts as an identity. Hence, the Furstenberg system of $n^{it}$ can be identified with the system $(Z:=\S^1,\lambda,\id)$, where  $\lambda$ is a continuous measure on $\S^1$ (in fact, $\lambda$ is equivalent and not equal to  the Lebesgue measure, but we shall not use this).  Since $f=\tilde{f}\cdot n^{it}$, it follows that the system $(X,\mu,T)$ is a factor of a joining of the systems
   $(Y,\nu,S)$ and $(Z,\lambda, \id)$, with factor map
   \begin{equation}\label{E:pixz}
   	  \pi(y,z):= M_z(y), \quad y\in Y,\, z\in \S^1,
   \end{equation}
     where $M_z(y)(k):=z\cdot y(k),\ k\in \Z$. (We will  make use of the form of $\pi$ only  at the end of our argument to get the product structure of the measure $\mu$). Since the two systems are disjoint, we obtain that
      \begin{equation}\label{E:disjoint}
   (X,\mu,T) \text{ is a factor of the direct product of }(Y,\nu,S) \text{ and } (Z,\lambda, \id).
   \end{equation}
After this preparation, we are now ready to prove the claimed properties.

 \smallskip

 {\em Rational discrete spectrum.}
  Since both systems $(Y,\nu,S)$ and $(Z,\lambda, \id)$ have rational discrete spectrum, so does their direct product. By \eqref{E:disjoint} the same holds for  $(X,\mu,T)$.

 \smallskip

{\em Non-trivial rational spectrum.} We show that if $\tilde{f}\neq 1$, then  the spectrum of the Furstenberg system $(X,\mu,T)$ of
 $f$ is non-trivial (hence,  if the spectrum is trivial, then $f(n)=n^{it}$, $n\in\N$, for some $t\in \R$). Arguing by contradiction, suppose that it is trivial. Then since it has rational discrete spectrum, it is an identity system.
  Moreover,  as remarked above, the Furstenberg system    of $n^{-it}$ is an identity system.
 Since $\tilde{f}=f\cdot n^{-it},$ we  deduce that the
 Furstenberg system $(Y,\nu,S)$ of $\tilde{f}$  is a factor of a joining  of two identity systems,  thus, it  is an identity system. This contradicts  part~\eqref{I:StructurePretentious1}, which implies that since $\tilde{f}\sim \chi$ and $\tilde{f}\neq 1$, the system $(Y,\nu,S)$  has non-trivial spectrum (since it is a non-trivial ergodic procyclic system).


 \smallskip

 {\em Non-ergodicity.}\footnote{Although a different proof for this fact can be given using the ``product structure'' result proved later in this subsection,  we  choose to also give this shorter and more direct argument.} We show that the system $(X,\mu,T)$ is non-ergodic.
 We first reduce to the case, where $\chi=1$. Since the non-zero values of $\chi$ are roots of unity of fixed order, there exists $q\in \N$ such that  $\chi^q\sim 1$. Then the  Furstenberg system of $f^q$ along $(N_k')$ is also well defined and ergodic  if $(X,\mu,T)$ was ergodic (since it defines
 a factor of $(X,\mu,T)$, see the second remark after Definition~\ref{D:Furstenberg}). Moreover,  $f^q\sim n^{iqt}\cdot \chi^q\sim n^{iqt}$. So, it suffices to show that
 the Furstenberg system of $f^q$ along $(N_k')$ is non-ergodic. Furthermore, by considering $f^{2q}$ in place of $f^q$ if needed, we can assume that  $f(2)\neq -2^{it}$. Thus, we have reduced matters to the case $\chi=1$ and $f(2)\neq -2^{it}$.

  Recall that we have
    $f=\tilde{f}\cdot n^{it}$ where $\tilde{f}=f\cdot n^{-it}$, hence $\tilde{f}\sim 1$ and $\tilde{f}(2)\neq -1$.
    Moreover, \eqref{E:disjoint} holds.
 We want to show that the system $(X,\mu,T)$ is non-ergodic, so arguing by contradiction, suppose that it  is ergodic.
Then
   \begin{equation}\label{E:erg1}
\lim_{N\to\infty}   \E_{n\in[N]}\lim_{k\to\infty}\E_{m\in [N_k']} \, f(m+n)\cdot \overline{f(m)} =|\lim_{k\to\infty} \E_{m\in [N_k']}\, f(m)|^2.
   \end{equation}
   (By \eqref{E:correspondence},  this is equivalent to the identity $\lim_{N\to\infty}   \E_{n\in[N]}\int T^nF_0\cdot \overline{F_0}\, d\mu=|\int F_0\, d\mu|^2$, which holds by the von Neumann ergodic theorem.)
Using the disjointness of the Furstenberg system of $(n^{it})$ and $(\tilde{f}(n))$, we get that     the right hand side  equals
   $$
   |\lim_{k\to\infty} \E_{m\in [N_k']}\, m^{it}|^2\cdot |\lim_{k\to\infty} \E_{m\in [N_k']}\, \tilde{f}(m)|^2.
   $$
   On the other hand, for every $n\in\N$ we have $\lim_{m\to\infty} (m+n)^{it}\cdot m^{-it}=1$, hence the left hand side in \eqref{E:erg1} equals
 $$
  \lim_{N\to\infty}   \E_{n\in[N]}\lim_{k\to\infty}\E_{m\in [N_k']} \, \tilde{f}(m+n)\cdot \overline{\tilde{f}(m)}=|\lim_{k\to\infty} \E_{m\in [N_k']}\, \tilde{f}(m)|^2,
   $$
   where the last identity follows as before from the ergodicity of the Furstenberg system  of $\tilde{f}$ along $(N_k')$ that was proved in  part~\eqref{I:StructurePretentious1} of Theorem~\ref{T:StructurePretentiousRefined}. Combining the above identities, we deduce that
$$   |\lim_{k\to\infty} \E_{m\in [N_k']}\, m^{it}|^2\cdot |\lim_{k\to\infty} \E_{m\in [N_k']}\, \tilde{f}(m)|^2=
   |\lim_{k\to\infty} \E_{m\in [N_k']}\, \tilde{f}(m)|^2 .
   $$
   Since $\tilde{f}\sim 1$ and $\tilde{f}(2)\neq -1$, by Theorem~\ref{T:Halasz}, we have  $\lim_{k\to\infty} \E_{m\in [N_k']}\, \tilde{f}(m)\neq 0$, hence the previous identity implies that
   $$
    |\lim_{k\to\infty} \E_{m\in [N_k']}\, m^{it}|=1.
   $$
   But one easily verifies (see the remark after Proposition~\ref{P:ceslog}) that this can only happen if $t=0$: a contradiction.

 \smallskip
 {\em Product structure.}
Recall that  \eqref{E:disjoint} holds and our plan is to use it in order to apply Proposition~\ref{P:product_factor} in the appendix. To this end, we need  to verify the following.

   \smallskip

   {\em Claim: There exists at least one $z\in\S^1$ such that the measure $\nu$ is not invariant under $M_z$.}
  Arguing by contradiction, suppose that the claim fails. Then, for every $F\in C(Y)$, we have
   $$
   \int F(z\cdot y)\, d\nu=\int F(y)\, d\nu, \quad \text{for every } z\in \S^1.
   $$
Let    $F_0$ be the $0^{\text{th}}$-coordinate projection  on $Y=\U^\Z$.
   Applying the previous identity for $F:=F_0^m$,  $m\in \N$, and using that $F_0(zy)=z\, y(0)=zF_0(y)$,  we deduce
   $$
   (z^m-1)\int F_0^m\, d\nu=0, \quad \text{for every } z\in \S^1.
   $$
   Hence,    $\int F_0^m\, d\nu=0$ for every $m\in \N$. Recall that $(Y,\nu,S)$ is the Furstenberg system for Ces\`aro averages of $\tilde{f}$ along the sequence $(N_k')$.   We deduce using \eqref{E:correspondence} that
   \begin{equation}\label{E:tildef}
   \lim_{k\to\infty}\E_{n\in[N_k']} \, (\tilde{f}(n))^m=0, \quad \text{ for every } m\in \N.
   \end{equation}
Since $\tilde{f}\sim \chi$ for some Dirichlet character $\chi$, there  exists $q\in \N$   such that  $\tilde{f}^q\sim 1$. Furthermore, by considering $2q$ in place of $q$, if needed, we can assume that   $\tilde{f}(2)^q \neq  -1$. Keeping these two facts in mind, we get that \eqref{E:tildef}, for $m:=q$, contradicts  Theorem~\ref{T:Halasz}, which claims that we cannot have a vanishing subsequential limit in this case. This completes the proof of the claim.

 Hence, Proposition~\ref{P:product_factor} in the appendix gives that there exists $r\in \N$ such that the system $(X,\mu,T)$ is isomorphic to
 $(Y,\nu,S)\times(\S^1,\lambda_r,\id)$, where $\lambda_r$ is the pushforward of the continuous measure   $\lambda$ by $z\mapsto z^r$.\footnote{If $f(n)=n^{i}\cdot \chi$ where $\chi:={\bf 1}_{4\Z+1}-{\bf 1}_{4\Z+3}$, then $r=2$ and the group $G$ that appears in the proof of Proposition~\ref{P:product_factor} is equal to $\{-1, 1\}$.}  Lastly, note that  since   $\lambda_r$ is   a continuous measure, and any two identity transformations on Lebesgue spaces with continuous probability measures are isomorphic (see for example \cite[Theorem~17.41]{Ke12}),
 the system $(\S^1,\lambda_r,\id)$   is isomorphic to the system $(\T,m_\T,\id)$.
 This completes the proof.

\subsubsection{Proof of part~\eqref{I:StructurePretentious2.5}}
Suppose that  $f\sim n^{it}\cdot \chi$ for some non-zero $t\in \R$ and primitive  Dirichlet character $\chi$. Let  $(X,\mu_1,T)$ and $(X,\mu_2,T)$ be two Furstenberg systems of $f$ for Ces\`aro or logarithmic averages (not necessarily both taken using Ces\`aro or logarithmic averages).  Using  part~\eqref{I:StructurePretentious2},
we get that
for $i=1,2$ the system $(X,\mu_i,T)$ is isomorphic
to the direct product of the system $(\T,m_\T,\id)$
and the system $(X,\mu_i',T)$, which is
some Furstenberg system of $\tilde{f}:=f\cdot n^{-it}$ (with respect to Ces\`aro or logarithmic averages). Since  $\tilde{f}\sim \chi$, by Proposition~\ref{P:isomorphic},
 we have  that
the systems  $(X,\mu_1',T)$ and $(X,\mu_2',T)$ are isomorphic. Combining the above, we get that the systems $(X,\mu_1,T)$ and $(X,\mu_2,T)$ are isomorphic.
This completes the proof.

\subsubsection{Proof of part~\eqref{I:StructurePretentious2.5'}}  Follows immediately from   Proposition~\ref{P:ceslog'}.

\subsubsection{Proof of part~\eqref{I:StructurePretentious3}}  In case \eqref{I:StructurePretentious1},  the asserted property  follows from  Proposition~\ref{P:RAPchi}.
In case \eqref{I:StructurePretentious2}, suppose that there exists a sequence $N_k\to\infty$ along which $f$ is Besicovitch rationally almost periodic for Ces\`aro or logarithmic averages. Upon passing to a subsequence, we can assume that the Furstenberg system of $f$ for Ces\`aro or logarithmic averages along $(N_k)$ is well defined. Then by Theorem~\ref{T:BRAP} this system is ergodic. This contradicts the non-ergodicity  established in  part~\eqref{I:StructurePretentious2}.

\subsection{Proof of Corollary~\ref{C:chi}}\label{SS:chi}
To establish the equivalence of \eqref{I:chi1}, \eqref{I:chi2}, \eqref{I:chi3} we argue as follows.
The implication $\eqref{I:chi1}\implies  \eqref{I:chi3}$ follows from part~\eqref{I:StructurePretentious1} of Theorem~\ref{T:StructurePretentiousRefined}.
The implication $\eqref{I:chi3}\implies  \eqref{I:chi2}$ is obvious.
The implication $\eqref{I:chi2}\implies  \eqref{I:chi1}$ follows from part~\eqref{I:StructurePretentious2} of Theorem~\ref{T:StructurePretentiousRefined}.

To establish the equivalence of \eqref{I:chi1}, \eqref{I:chi4}, \eqref{I:chi5} we argue as follows.
The implication $\eqref{I:chi1}\implies  \eqref{I:chi5}$ follows from the first assertion in part~\eqref{I:StructurePretentious3} of Theorem~\ref{T:StructurePretentiousRefined}.
The implication $\eqref{I:chi5}\implies  \eqref{I:chi4}$ is obvious.
The implication $\eqref{I:chi4}\implies  \eqref{I:chi1}$ follows from the second assertion in part~\eqref{I:StructurePretentious3} of Theorem~\ref{T:StructurePretentiousRefined}.

 \subsection{Proof of Corollary~\ref{C:nit}}\label{SS:nit}
The implication $\eqref{I:nit1}\implies\eqref{I:nit5}$ is obvious.
The implication $\eqref{I:nit5}\implies\eqref{I:nit4}$ is simple and follows from \cite[Corollary~5.5]{GLR21} and  \cite[Section~1.3]{FH21}. The implications  $\eqref{I:nit4}\implies\eqref{I:nit3}\implies \eqref{I:nit2}$ are trivial.  It remains to establish the implication
 $\eqref{I:nit2}\implies\eqref{I:nit1}$ Namely, we want to show that  if some Furstenberg system of $f$ for Ces\`aro or logarithmic averages has trivial rational spectrum, then $f(n)=n^{it}$, $n\in \N$, for some $t\in \R$. This follows  from part~\eqref{I:StructurePretentious2} of Theorem~\ref{T:StructurePretentiousRefined}.

\section{Spectral results for pretentious multiplicative functions  and applications}
In this section we prove  Theorems~\ref{T:spectrum}-\ref{T:Chowla}.
\subsection{Necessary conditions for spectrum}
The next result gives  a necessary  condition that often helps us  decide when a certain number belongs to the spectrum of some Furstenberg system of a  bounded sequence.
Furthermore, it  gives an  identity that we will use to prove Theorem~\ref{T:ceslogzero} when we happen to know  that  this number is not in the spectrum.
\begin{proposition}\label{P:spectrum}
	Let $(X,\mu, T)$ be the Furstenberg system for Ces\`aro averages  of a sequence  $a\colon \N\to \U$ taken along a sequence
	$(N_k)$. If $\alpha\in (0,1)$ is such that $\alpha\not\in \spec(X,\mu, T)$, then
\begin{equation}\label{djoi}
	\lim_{k\to\infty} \E_{n\in[N_k]}\, e(-n\alpha) \prod_{j=1}^\ell a_j(n+n_j)=0
\end{equation}
	for all   $\ell\in \N$, $n_1,\ldots, n_\ell\in\N$, and  $a_1,\ldots, a_\ell\in \{a,\overline{a}\}$.
	A similar statement also holds when we consider logarithmic averages.
\end{proposition}
\begin{proof}
	Arguing by contradiction, suppose that the conclusion fails for some $\ell\in \N$, $n_1,\ldots, n_\ell\in\N$, and \ $a_1,\ldots, a_\ell\in \{a,\overline{a}\}$.  Using van der Corput's lemma (the variant needed follows from \cite[Lemma~3.1]{KN74}), we get
	$$
	\lim_{H\to\infty} \E_{h\in[H]} \, e(-h\alpha) \Big(\lim_{k\to\infty} \E_{n\in[N_k]}\,  \prod_{j=1}^\ell a_j(n+n_j)  \prod_{j=1}^\ell \overline{a_j}(n+h+n_j)\Big) \neq 0.
	$$
	(Note that all the limits as $k\to \infty$ exist because  $a$ is  assumed to admit correlations  along $(N_k)$.)
Using \eqref{E:correspondence} to translate this equation to the Furstenberg system,	we get that
	$$
	\lim_{H\to\infty} \E_{h\in[H]} \, e(-h\alpha) \int F \cdot T^h\overline{F}\, d\mu\neq 0,\footnote{The existence of the limit follows by applying the mean ergodic theorem to an appropriate product system.}
	$$
	where $F:=  \prod_{j=1}^\ell T^{n_j}F_j$ for some $F_j\in \{F_0,\overline{F_0}\}$ and $F_0$ is the $0^{\text{th}}$-coordinate projection.  This implies that $\alpha$ is on the spectrum of $(X,\mu,T)$, in fact,
	$$
	G:=\lim_{H\to\infty} \E_{h\in[H]} \, e(-h\alpha) \, T^h\overline{F},
	$$
	 where the limit is taken in $L^2(\mu)$, is a non-zero $e(\alpha)$-eigenfunction. This contradicts our assumption $\alpha\not\in \spec(X,\mu, T)$ and completes the proof.\footnote{Alternatively, one can argue that if the conclusion fails, then   we get a nontrivial joining of the system with the rotation by $\alpha$ on the circle.
Now, if the limit in \eqref{djoi} is non-zero, then the conditional expectation (with respect to the Furstenberg system in the joining) of the eigenfunction (on the circle) corresponding to $\alpha$ is non-zero. Hence,  $\alpha$ belongs to the spectrum of the system $(X,\mu,T)$.}
\end{proof}
The following  simpler condition  often suffices to deduce that a certain rational number belongs to the spectrum of a Furstenberg system of a bounded sequence (we caution the reader though, that the condition is far from necessary).
\begin{corollary}\label{C:spectrum}
	Let $(X,\mu, T)$ be the Furstenberg system for Ces\`aro averages of a sequence  $a\colon \N\to \U$ taken along
	$N_k\to\infty$. Suppose that for some $q\geq 2$ and $r\in \{0,\ldots, q-1\}$ we have
	$$
	\lim_{k\to\infty}\big( \E_{n\in [N_k/q]} \, a(qn+r)- \E_{n\in [N_k]} \, a(n)\big)\neq 0.
	$$
	Then $p/q$ belongs to the spectrum of $(X,\mu,T)$ for some $p\in \{1,\ldots, q-1\}$.
	A similar statement holds for logarithmic averages, and in this context we can replace
	$ \lE_{n\in [N_k/q]}$ with  $\lE_{n\in [N_k]}$.
\end{corollary}
\begin{proof}
 Our assumption is equivalent to
\begin{equation}\label{E:qNr}
\lim_{k\to\infty} \E_{n\in [N_k]} \, (1- q\, {\bf 1}_{q\N+r}(n))\, a(n)\neq 0.
\end{equation}
Upon substituting the identity
$$
q\, {\bf 1}_{q\N+r}(n)=\sum_{p=0}^{q-1}e\Big(p\cdot \frac{n-r}{q}\Big)
$$
in \eqref{E:qNr},
we deduce that
$$
\lim_{k\to\infty} \E_{n\in [N_k]} \, e\Big(-n\cdot \frac{p}{q}\Big) \, a(n)\neq 0
$$
for some  $p\in \{1,\ldots, q-1\}$. It follows from Proposition~\ref{P:spectrum} (we only need to appeal to  the $\ell=1$ case) that $p/q$ is in the spectrum of the system.
\end{proof}

\subsection{Preparation for the proof of  Theorem~\ref{T:spectrum}}
 If $\chi$ is a primitive Dirichlet character with conductor $q$, then it has a unique Furstenberg system, which  is periodic with minimal period $q$, and its spectrum is spanned by multiples of $1/q$. 
The next result shows that if $f\sim \chi$ for some primitive Dirichlet character $\chi$, then the spectrum of every Furstenberg system of $f$ is at least as large as the spectrum of the Furstenberg system of $\chi$.
\begin{proposition}\label{P:fsimchi}
Let $(X,\mu, T)$ be the Furstenberg system for Ces\`aro averages  of a  multiplicative function  $f\colon \N\to \U$ that satisfies $f\sim \chi$ for some primitive Dirichlet character $\chi$ with conductor $q\geq 2$. 
Then for every $p\in \P$ with $p\mid q$ we have $1/p\in \spec(X,\mu,T)$. 
\end{proposition}
\begin{proof}
We assume that the Furstenberg system is defined along a sequence $(N_k)$ and  by passing to a further subsequence we can assume that the Furstenberg system of all other multiplicative functions defined subsequently are also taken along 
$(N_k)$.

 Suppose first that $q=p^k$ for some $p\in \P$. Since $f\sim \chi$ we have $f\cdot \overline{\chi}\sim 1$. If  $f(2^s)\overline{\chi(2^s)}\neq -1$ for some $s\in \N$, then by Theorem~\ref{T:Halasz} we have  
$$
\limsup_{k\to\infty}|\E_{n\in [N_k]} \, f(n)\cdot \overline{\chi(n)} |>0.
$$
Since $\chi$ is a primitive Dirichlet character with conductor $q>1$, it is non-principal, and as a consequence has mean value zero. Since $\chi$ is periodic with period $q$,  using its Fourier expansion in $\Z_q$, we deduce that 
$$
\limsup_{k\to\infty}|\E_{n\in [N_k]} \, f(n)\cdot e(n\cdot r/q) |>0
$$
for some $r\in \{1,\ldots, q-1\}$. It follows from Proposition~\ref{P:spectrum} that $r/q\in \spec(X,\mu, T)$ and since $q=p^k$ and the spectrum is closed under integer multiplication, we deduce that $1/p\in \spec(X,\mu,T)$.
 Suppose now that $f(2^s)\cdot \overline{\chi(2^s)}=-1$ for every $s\in \N$. Then $\chi(2)\neq 0$, hence $p\neq 2$. We let $\tilde{f}:= {\bf 1}_{2\Z+1} \cdot f\sim \chi$, then $\tilde{f}\cdot \overline{\chi} \sim 1$ and 
$ (\tilde{f}\cdot \overline{\chi})(2)=0\neq -1$, and the  previous case gives that $1/p$ belongs to the spectrum of the Furstenberg system of $\tilde{f}$ along $(N_k)$. Since  $\tilde{f}:= {\bf 1}_{2\Z+1} \cdot f$, the Furstenberg system of $\tilde{f}$ along $(N_k)$ is a factor of a joining of the Furstenberg system of $f$ along $(N_k)$, which by Theorem~\ref{T:StructurePretentiousRefined} is an ergodic procyclic system, and  the Furstenberg system of ${\bf 1}_{2\Z+1}$, which is periodic with period two.  It follows that $1/p$ is an integer combination of 
$1/2$ and elements in $\spec(X,\mu,T)$. Since $p\neq 2$, we deduce that  $1/p\in \spec(X,\mu,T)$. 
  
  We consider now the general case where $q$ is not a prime power (our assumption that $\chi$ is primitive is crucially used here). If $q=p_1^{k_1}\cdots p_\ell^{k_\ell}$, $\ell\geq 2$,  is the prime factorisation of $q$, then it is a standard fact that $\chi$ can be decomposed as $\chi=\chi_1\cdots \chi_\ell$ where $\chi_1,\ldots, \chi_\ell$ are primitive Dirichlet characters with conductors $p_1^{k_1},\ldots, p_\ell^{k_\ell}$ respectively. Note that then $f\cdot \overline{\chi_1}\cdots \overline{\chi_{\ell-1}}\sim \chi_\ell$, and by the first case we get that $1/p_\ell$ belongs to the spectrum of the Furstenberg system of  $f\cdot \overline{\chi_1}\cdots \overline{\chi_{\ell-1}}$ along $(N_k)$, which is a factor of a  joining of the system $(X,\mu,T)$, which by Theorem~\ref{T:StructurePretentiousRefined} is an ergodic procyclic system, and the Furstenberg system of $\overline{\chi_1}\cdots \overline{\chi_{\ell-1}}$. We deduce that $1/p_\ell$ is an integer combination of elements of the form $1/p_i^{k_i}$, $i=1,\ldots, \ell-1$ (these elements span the spectrum of  $\overline{\chi_1}\cdots \overline{\chi_{\ell-1}}$), and elements in 
  $\spec(X,\mu,T)$. Since $p_j\neq p_\ell$ for $j=1,\ldots, \ell-1$, it follows that  $1/p_\ell\in \spec(X,\mu,T)$. Similarly we get that  $1/p_i\in \spec(X,\mu,T)$ for $i=1,\ldots,\ell-1$. This completes the proof. 
\end{proof} 
We will also need the following simple fact. 
\begin{lemma}\label{L:chimodified}
If  $\chi$ is a primitive Dirichlet character with conductor $q\geq 2$,
we define the completely multiplicative function $\tilde{\chi}$  on prime numbers $p$ by
$$
\tilde{\chi}(p):=\begin{cases}
	\chi(p), \quad & p\nmid q \\
	1, \quad & p\mid q.
\end{cases} 
$$
Then $\tilde{\chi}$ has a unique Furstenberg system  $(X,\mu,T)$ and  for $p_0\in \P$ we have $1/p_0 \in \spec(X,\mu,T)$   if and only if   $p_0\mid q$. 
\end{lemma}
\begin{proof}
The uniqueness of the Furstenberg system of $\tilde{\chi}$ will be established in the second part of the proof. For the moment we assume that $(X,\mu,T)$ is some Furstenberg system of $\tilde{\chi}$. 

Since $\tilde{\chi}\sim \chi$ and  $\chi$ is a primitive Dirichlet character with conductor $q$ and $p_0\mid q$,  we have by Proposition~\ref{P:fsimchi} that $1/p_0\in (X,\mu,T)$.

 Suppose that $p_0\in \P$ is such that  $p_0\nmid q$, we shall   show that $1/p_0\not\in \spec(X,\mu,T)$. We define the multiplicative function 
$$
\tilde{\chi}_m(n):=\begin{cases}
	\tilde{\chi}(n), \quad & \text{if }\,  q^m\nmid n \\
	0, \quad & \text{if } \, q^m\mid n.
\end{cases} 
$$
One easily verifies that $\tilde{\chi}_m$ is periodic with period $q^m$, hence it has a unique  Furstenberg system with spectrum   a subset of  the subgroup generated by $1/q^m$, and 
$$
\lim_{m\to\infty}\limsup_{N\to\infty}\E_{n\in[N]} |\tilde{\chi}_m(n)-\tilde{\chi}(n)|=0.
$$
By   \cite[Lemma~3.17]{BPLR19},  $\tilde{\chi}$ has a unique Furstenberg system, and it  is a factor of
a joining of the periodic systems generated by $\tilde{\chi}_m$, $m\in \N$.
Hence,
$\spec(X,\mu,T)$ is contained in the subgroup $\{j/q^m\colon j=0,1\ldots q^m-1,m\in \N\}$ of $\T$.  Since  $p_0\nmid q$, we deduce that $1/p_0\not\in \spec(X,\mu,T)$, completing the proof.
\end{proof}
Lastly, we will use the following elementary fact. 

\begin{lemma}\label{L:2rational} Let $G_1,G_2$ be subgroups of $\T$ consisting of rational numbers. If  $p\in\P$ satisfies
$1/p\not\in G_j$ for   $j=1,2$, then $1/p$ does not belong to the subgroup generated by  $G_1$ and $G_2$.
\end{lemma}
\begin{proof}
We argue by contradiction. Suppose that $1/p$ belongs to the subgroup generated by $G_1$ and $G_2$. Then, for $j=1,2$, there exist $a_j,b_j \in \N$ such that $a_j/b_j\in G_j$, $(a_j,b_j)=1$, and such that
	$1/p=a_1/b_1+a_2/b_2.$
	Since $1/p$ does not belong to the subgroup $G_1$ and we also have
	$a_1/b_1\in G_1$
	and $(a_1,b_1)=1$, we easily deduce  that $p$ does not divide $b_1$. Similarly, we get that $p$ does not divide $b_2$.   On the other hand, we have $b_1b_2=p(a_1b_2+a_2b_1)$, hence $p$ divides either $b_1$ or $b_2$,
	a contradiction. This completes the proof.
\end{proof}

\subsection{Proof of the implication $\eqref{I:spectrum1}\implies \eqref{I:spectrum2}$ of Theorem~\ref{T:spectrum}}\label{SS:spectrumproof1}
By Proposition~\ref{P:ceslog'}, Furstenberg systems for Ces\`aro and logarithmic averages coincide, so we only explain the argument for Ces\`aro averages. Suppose that the Furstenberg system $(X,\mu,T)$ is taken along the sequence $N_k\to\infty$.

Before proceeding further, let us make the following simple observation: Suppose  that $p_1,\ldots,p_k\in\mathbb{P}$, $s_1,\ldots,s_k\in\mathbb{N},$ and define the multiplicative function $f$ by setting $f(p^s_i):=\alpha_{i,s}\in\mathbb{U}$ for $i=1,\ldots,k$ and $s=1,\ldots,  s_i$,  and $f(p^s):=1$ for all other prime powers. Then $f(n)$ is completely determined from the knowledge of all remainders of $n$ modulo $p_i^{s_{i}+1}$, $i=1,\ldots,k$, hence
\begin{equation}\label{E:fper}
f\colon \mathbb{N}\to \mathbb{U}\ \text{ is periodic with period }\ \prod_{i=1}^kp_i^{s_i+1}.
\end{equation}
In our argument below we will approximate $f$ by periodic multiplicative functions of the previous form and deduce properties for $f$ from those of its periodic approximants. 
 
 Suppose first that $f\sim 1$  and  $p_0\in \P$ is such that $f(p_0^s)=1$ for every $s\in \N$. Our goal is to show
that  $1/p_0\not\in \spec(X,\mu,T)$.

By Proposition~\ref{P:RAPchi} (see the remark following the result), there exist a subsequence $(N_k')$ of $(N_k)$ and constants $\alpha_m\in [0,1)$, $m\in \N$, such that
\begin{equation}\label{E:fkn}
		\limsup_{k\to\infty}\E_{n\in[N_k']}|f(n)-e(\alpha_m)\cdot f_m(n)|^2\leq 1/m,
\end{equation}
where
$f_m\colon \N\to \U$, $m\in \N$, are  defined as
in Lemma~\ref{L:12e} (for $\varepsilon:=1/m$) and satisfy property \eqref{I:f1} of Lemma~\ref{L:12e} (with $\chi:=1$). Furthermore, since $f(p_0^s)=1$ for every $s\in \N$, we
can also assume that
\begin{equation}\label{E:fm1}
f_m(p_0^s)=1 \text{ for every } m,s\in \N.
\end{equation}
  Moreover, since for $m\in\N$ the sequence $f_m$ satisfies  \eqref{E:1-f} with $\chi=1$,
as shown in the proof of \cite[Lemma~4]{DD82} (we crucially use here that $f\sim 1$), we have
(in fact, we could replace $\limsup_{k\to\infty}\E_{n\in[N_k']}$ with
	$\lim_{N\to \infty} \E_{n\in [N]}$)
\begin{equation}\label{E:fkn'}
	\limsup_{k\to\infty}\E_{n\in[N_k']}|f_m(n)-f'_m(n)|^2\leq 1/m,
\end{equation}
where for every $m\in \N$ the multiplicative functions $f_m'\colon \N\to \U$ are  defined by
$$
f'_m(p^s):=\begin{cases}
	f_m(p^s), \quad &\text{for } p\leq r_m, s\in \N\\
	1, \quad &\text{for all other prime powers},
\end{cases}
$$
for $r_m\in \N$ sufficiently large that we can arrange to form an increasing sequence.
Note that  \eqref{E:fm1} implies
\begin{equation}\label{E:fm1'}
	f'_m(p_0^s)=1 \quad  \text{ for every } m,s\in \N.
\end{equation}
Let $f'_{m,l}\colon \N\to \U$ be multiplicative functions  defined by  
$$
f'_{m,l}(p^s):=\begin{cases}
	f'_m(p^s), \quad &\text{for } p\leq r_m, s\leq l\\
	1, \quad &\text{for all other prime powers}.
\end{cases}
$$
Then \eqref{E:fm1'} implies that
\begin{equation}\label{E:fml1'}
f'_{m,l}(p_0^s)=1 \, \text{ for every } m,l,s\in \N.
\end{equation}
Combining  \eqref{E:fper} with \eqref{E:fml1'}, we get that  for every $m,l\in \N$,  the sequence $n\mapsto f'_{m,l}(n)$ is  periodic   with period $P_{m,l}:=\prod_{p\leq r_m, p\neq p_0}p^{l+1}$.
Arguing as in the proof of   \cite[Lemma~4]{DD82}, we get
\begin{equation}\label{E:fkn''}
	\limsup_{k\to\infty}\E_{n\in[N'_k]}|f'_m(n)-f'_{m,l_m}(n)|^2\leq 1/m
\end{equation}
for  $l_m\in \N$ sufficiently large that we can arrange to form an increasing sequence.
Combining \eqref{E:fkn}, \eqref{E:fkn'}, \eqref{E:fkn''},   we get
\begin{equation}\label{E:fmlm}
\lim_{m\to\infty}	\limsup_{k\to\infty}\E_{n\in[N_k']}|f(n)-e(\alpha_m)\cdot f'_{m,l_m}(n)|^2=0.
\end{equation}
 Now, for every $m\in \N$, the sequence $(e(\alpha_m)\cdot f'_{m,l_m}(n))_{n\in\N}$
is periodic with period
$$
P_m:=\prod_{p\leq r_m, p\neq p_0}p^{l_m+1}.
$$
 Hence,    the spectrum of its Furstenberg system $(X,\mu_m,T)$ is contained in the subgroup $G_m$ of $\T$ generated by $1/P_m$. Since $p_0$ does not divide $P_m$, we have that   $1/p_0\not\in G_m$. In summary,
 \begin{equation}\label{E:Gm}
 \spec(X,\mu_m,T)\subset G_m \, \text{ and } \, 1/p_0\not\in  G_m\,  \text{ for every }\,  m\in \N.
 \end{equation}
Note also that because $(r_m)$ and $(l_m)$ are increasing sequences, we have that $P_m$ divides $P_{m+1}$, hence  $(G_m)$ is an increasing sequence of rational subgroups.

By \eqref{E:fmlm} and  \cite[Lemma~3.17]{BPLR19},  the system  $(X,\mu,T)$ is a factor of
a joining of the periodic systems $(X,\mu_m,T)$, $m\in \N$.
Hence,
$\spec(X,\mu,T)$ is contained in the subgroup generated by the  sequence of subgroups $\spec(X,\mu_m,T)$, $m\in \N$ (see  the remark~\eqref{I:ratdisc} at the end of Section~\ref{SS:mps}), which in turn, by~\eqref{E:Gm}, is contained in the subgroup generated by the increasing sequence of rational subgroups $G_m$, $m\in \N$. By \eqref{E:Gm} we have that  $1/p_0\not\in G_m$ for every $m\in\N$, hence $1/p_0\not\in \spec(X,\mu,T)$.

  Next, we deal with the  case where $f\sim \chi$ for some primitive Dirichlet character $\chi$ with conductor $q\geq 2$. 
 Let $p_0\in \P$ be such that  $p_0\nmid q$ and $f(p_0^s)=\chi(p_0^s)$ for every $s\in \N$. Our goal is to show
that  $1/p_0\not\in \spec(X,\mu,T)$. 
We take    $\tilde{\chi}$  as in Lemma~\ref{L:chimodified}  and let  $\tilde{f}:=f\cdot \overline{\tilde{\chi}}$. Then $\tilde{f}\sim 1$ and $\tilde{f}(p_0^s)=1$ for every $s\in \N$ (we used that $p_0\nmid q$ here). 
Hence, by the previous part we have that $1/p_0$ does not belong to the spectrum of the   Furstenberg system of $\tilde{f}$ taken along  a subsequence $(N_k')$ of $(N_k)$.   Furthermore, by  Lemma~\ref{L:chimodified}, it follows  that $1/p_0$ does not belong to the spectrum of the  (unique)  Furstenberg system of $\tilde{\chi}$.
Since $f=\tilde{f}\cdot \tilde{\chi}$ (we used that $|\tilde{\chi}|=1$ here), arguing as before,  and using Lemma~\ref{L:2rational}, we get that $1/p_0\not\in (X,\mu,T)$. This  completes the proof. 

Finally, to deal with the case $f\sim n^{it}\cdot \chi$ for some $t\in \R$ and primitive Dirichlet character $\chi$,  we can combine the previous part with  part~\eqref{I:StructurePretentious2} of Theorem~\ref{T:StructurePretentiousRefined}. Alternatively, given the previous part,  one can get a more direct proof by arguing as in Section~\ref{SSS:spectrumreverse4} below.

\subsection{Proof of the implication $\eqref{I:spectrum2}\implies \eqref{I:spectrum1}$ of Theorem~\ref{T:spectrum}}\label{SS:spectrumproof2}
   We assume that $f\sim \chi$ for some primitive Dirichlet character $\chi$ with conductor $q$  and we have either $p\nmid q$ or 
$f(p^s)\neq \chi(p^s)$ for some $p\in \P$ and $s\in\N$.  Our goal is to show that $1/p\in \spec(X,\mu,T)$.
We do this in three  steps. The bulk of the proof is contained on the first step, where we work under the additional hypothesis  $f\sim 1$ and $f(2^{s_0})\neq -1$ for some $s_0\in \N$.  On the second step we use the first one to  cover the complementary case $f\sim 1$ and $f(2^s)= -1$ for every $s\in \N$. Finally, on the third step we use the first two steps to cover the general case $f\sim n^{it} \cdot   \chi$ for some $t\in \R$ and primitive Dirichlet character $\chi$.

\subsubsection{The case  $f\sim 1$ and $f(2^{s_0})\neq -1$ for some $s_0\in \N$}\label{SSS:spectrumreverse1} We assume that
$f\sim 1$ and $f(2^{s_0})\neq -1$ for some $s_0\in \N$, this is in addition to our hypothesis that
$f(p^s)\neq 1$ for some $s\in\N$.
Note first that since $f\sim 1$,  Proposition~\ref{P:ceslog'} implies that the Furstenberg systems of $f$ for Ces\`aro and logarithmic averages  coincide. So, we only have to treat the case of logarithmic averages, and this turns out to offer  a  substantial advantage since
it enables us to replace $\lim_{k\to\infty}\lE_{n\in[N_k/p]}$ with
$\lim_{k\to\infty}\lE_{n\in[N_k]}$ throughout.

So, let  $(X,\mu,T)$ be a Furstenberg system of $f$  for logarithmic averages
taken along the sequence $N_k\to \infty$.
 Note that then the limit
$L:=\lim_{k\to\infty}  \lE_{n\in [N_k]} \, f(n)$ exists and since $f\sim 1$ and $f(2^{s_0})\neq -1$ for some $s_0\in \N$, we get by Theorem~\ref{T:Halasz} that  $L\neq 0$ (see the remarks following the  theorem).

Arguing by contradiction, suppose that $1/p\notin \spec(X,\mu,T)$.  By Corollary~\ref{C:spectrum}, this implies
\begin{equation}\label{E:pn}
L=\lim_{k\to\infty}\lE_{n\in [N_k]} \, f(pn+j) = \lim_{k\to\infty} \lE_{n\in [N_k]} \, f(n), \quad \text{for every }\,  j\in \{0,\ldots, p-1\}.
\end{equation}
Using these identities, our goal is to show that  $f(p^s)= 1$ for every $s\in\N$, which would contradict our hypothesis.

For convenience, we use the convention
$$
\lE_{n\in \bN}\,  a(n):=\lim_{k\to\infty}\lE_{n\in [N_k]} \, a(n)
$$
whenever the limit exists.
In what follows we will use  the following basic property of logarithmic averages:
If $a\colon \N\to \U$  is such that the limit  $\lim_{k\to\infty}\lE_{n\in [N_k]} \, a(n)$ exists,
then we also have  that the limit $\lim_{k\to\infty}\lE_{n\in [p N_k]} \, a(n)$  also exists and
the two limits are equal. Using this fact to justify the first identity below and
  \eqref{E:pn} to justify the second, we get
\begin{align*}
\lE_{n\in \bN}\, f(n)& = \frac{1}{p} \lE_{n\in \bN}\, f(pn)
+\frac{1}{p}\sum_{j=1}^{p-1}\lE_{n\in \bN}\, f(pn+j)\\
& =
\frac{1}{p} \lE_{n\in \bN}\, f(pn)+L  \cdot \frac{p-1}{p}.
\end{align*}

Repeating this process one more time  in order to evaluate
$\frac{1}{p} \lE_{n\in \bN}\, f(pn)$, we get
\begin{align*}
\frac{1}{p} \lE_{n\in \bN}\, f(pn)&=
\frac{1}{p^2} \lE_{n\in \bN}\, f(p^2n)+\frac{1}{p^2}\sum_{j=1}^{p-1}\lE_{n\in \bN}\, f(p(pn+j))\\
&= \frac{1}{p^2} \lE_{n\in \bN}\, f(p^2n)+ L \cdot\frac{p-1}{p}\cdot  \frac{f(p)}{p},
\end{align*}
where to get the second  identity we used that $f(p(pn+j))=f(p)\, f(pn+j)$ for $j=1,\ldots, p-1$ and \eqref{E:pn}.
Combining these two identities gives
$$
\lE_{n\in \bN}\, f(n) =\frac{1}{p^2} \lE_{n\in \bN}\, f(p^2n)+L  \cdot \frac{p-1}{p} \Big(1+  \frac{f(p)}{p}\Big).
$$
Repeating this process $M-2$ more times, we deduce
$$
\E_{n\in \bN}\, f(n) = \frac{1}{p^M} \lE_{n\in \bN}\, f(p^Mn)+ L\cdot \frac{p-1}{p}\, \sum_{s=0}^{M-1} \frac{f(p^{s})}{p^s}.
$$
Letting $M\to\infty$ gives
$$
L=\E_{n\in \bN}\, f(n) =  L\cdot \frac{p-1}{p}\cdot \sum_{s=0}^\infty \frac{f(p^{s})}{p^s}.
$$
Since $L\neq 0$, we deduce that
\begin{equation}\label{E:fp-1}
\sum_{s=1}^\infty \frac{f(p^{s})}{p^{s}}  =\frac{1}{p-1}.
\end{equation}
Since $\sum_{s=1}^\infty \frac{1}{p^s}=\frac{1}{p-1}$ and $\Re(f(p^s))\in [-1,1]$,  comparing real parts in \eqref{E:fp-1} we get  $\Re(f(p^s))=1$ for every $s\in \N$.
Since $f(p^s)\in \U$, we deduce that $f(p^s)=1$ for every $s\in \N$, completing the proof in this first case.

\subsubsection{The case  $f\sim 1$ and $f(2^s)=-1$ for every $s\in\N$.}\label{SSS:spectrumreverse2}
We assume that $f\sim 1$ and $f(2^s)=-1$ for every $s\in\N$. Let $p\in \P$ be such that
$f(p^{s_0})\neq 1$ for some $s_0\in \N$ and our goal is to show that $1/p\in \spec(X,\mu,T)$. Suppose that the Furstenberg system of $f$ for logarithmic averages is taken along $N_k\to\infty$.

We first show that $1/2\in \spec(X,\mu,T)$. Since $f(2^s)=-1$ for every $s\in\N$, by Theorem~\ref{T:Halasz}, we have  $\lim_{k\to\infty} \lE_{n\in [N_k]}\, f(n)=0$. On the other hand,
the multiplicative function
\begin{equation}\label{E:xodd}
\tilde{f}:=f\cdot {\bf 1}_{2\Z+1}
\end{equation}
 also satisfies $\tilde{f}\sim 1$, and in addition satisfies $\tilde{f}(2)\neq -1$. Hence, Theorem~\ref{T:Halasz} gives that for a subsequence $(N_k')$ of $(N_k)$ along which the next limit exists, we have
$$
0\neq \lim_{k\to\infty} \lE_{n\in [N_k']}\, \tilde{f}(n)=\lim_{k\to\infty} \lE_{n\in [N_k'/2]}\, f(2n+1)=\lim_{k\to\infty} \lE_{n\in [N_k']}\, f(2n+1),
$$
where the last identity holds because we use logarithmic averages.
Hence,
$$
\lim_{k\to\infty} \lE_{n\in [N_k']}\, f(2n+1)\neq \lim_{k\to\infty} \lE_{n\in [N_k']}\, f(n).
$$
 Since $(N_k')$ is a subsequence of $(N_k)$,  Corollary~\ref{C:spectrum} implies that $1/2\in  \spec(X,\mu,T)$.

Suppose now that  $p\neq 2$, we will show that $1/p\in \spec(X,\mu,T)$. We define again the multiplicative function
$\tilde{f}$ as in \eqref{E:xodd}. It satisfies $\tilde{f}\sim 1$ and  $\tilde{f}(p^{s_0})\neq 1$, together with $\tilde{f}(2)\neq -1$. So, by the  case treated in Section~\ref{SSS:spectrumreverse1}, we have that if $(X,\tilde{\mu},T)$ is the Furstenberg system of $\tilde{f}$ along a subsequence $(N_k')$ of $(N_k)$, then $1/p\in \spec(X,\tilde{\mu},T)$. Now, notice that
since $\tilde{f}=f\cdot {\bf 1}_{2\Z+1}$, we have that
$(X,\tilde{\mu},T)$ is a factor of a joining of the system $(X,\mu,T)$, which is an ergodic procyclic system by part~\eqref{I:StructurePretentious1} of Theorem~\ref{T:StructurePretentiousRefined}, and the ergodic rotation
on two elements. Hence, $1/p$ belongs to the subgroup spanned by $\spec(X,\mu,T)$ and $1/2$. Since $(X,\mu,T)$ is ergodic,
 $\spec(X,\mu,T)$ is a group, and as we showed just above, $1/2\in \spec(X,\mu,T)$. Combining these facts, we deduce  that $1/p\in \spec(X,\mu,T)$.

\subsubsection{The case  $f\sim \chi$.}\label{SSS:spectrumreverse3}
   We now treat the  case where   $f\sim \chi$ for some primitive Dirichlet character $\chi$ with conductor $q$.  If $p\in \P$ satisfies $p\mid q$, then by  Proposition~\ref{P:fsimchi} we have that $1/p\in \spec(X,
\mu,T)$. So it remains to show that if  $p\in \P$ is such that $p\nmid q$ and 
$f(p^{s_0})\neq \chi(p^{s_0})$ for some $s_0\in \N$, then  $1/p\in \spec(X,\mu,T)$. Suppose that the Furstenberg system  $(X,\mu,T)$  of $f$ for logarithmic averages is taken along $N_k\to\infty$.

We define the multiplicative function
 $$
\tilde{f}:=f\cdot \overline{\chi}.
$$
 Since  $f\sim \chi$, we have $\tilde{f}\sim 1$. Furthermore, since
$f(p^{s_0})\neq \chi(p^{s_0})$ and $|\chi(p^{s_0})|=1$,  we have $\tilde{f}(p^{s_0})\neq 1$.
Hence, if $(N_k')$ is a subsequence of $(N_k)$ along which the Furstenberg system $(X,\tilde{\mu},T)$ of $\tilde{f}$ for logarithmic averages is defined,  the case treated in Section~\ref{SSS:spectrumreverse2} implies that $1/p\in \spec(X,\tilde{\mu},T)$.

Since $\tilde{f}=f\cdot \overline{\chi}$,  the system
$(X,\tilde{\mu},T)$ is a factor of a joining of the system $(X,\mu,T)$ and  the Furstenberg system of  $\chi$, the first is an ergodic procyclic system by part~\eqref{I:StructurePretentious1} of Theorem~\ref{T:StructurePretentiousRefined} and the second is a periodic system with period $q$. Hence, $1/p$ belongs to the subgroup generated  by $G_1:=\spec(X,\mu,T)$ and the set $G_2:=\{j/q\colon j=0,\ldots, q-1\}$. Note that 
$G_1,G_2$ are rational subgroups of $\T$ and since $p\nmid q$ we have $1/p\not\in G_2$. 
 To conclude the proof, we note that if $1/p\not\in G_1$, then  since we also have $1/p\not\in G_2$,  we get by Lemma~\ref{L:2rational}
 that $1/p$ does   not belong to   the subgroup generated by $G_1$ and $G_2$,  a contradiction. We deduce that  $1/p\in G_1=\spec(X,\mu,T)$, completing the proof.

\subsubsection{The case  $f\sim n^{it}\cdot  \chi$.}\label{SSS:spectrumreverse4}
Assuming the previous case, the result follows easily by part~\eqref{I:StructurePretentious2} of Theorem~\ref{T:StructurePretentiousRefined}, but we give a more direct argument below.

  Suppose that $f\sim n^{it}\cdot  \chi$ for some $t\neq 0$ and primitive Dirichlet character $\chi$ with conductor $q$.  Let $p\in \P$ be such that either $p\mid q$ or 
$f(p^s)\neq p^{ist}\cdot \chi(p^s)$ for some $s\in \N$ and our goal is to show that $1/p\in \spec(X,\mu,T)$. Suppose that the Furstenberg system  $(X,\mu,T)$  of $f$ for Ces\`aro averages is taken along $N_k\to\infty$, the argument is similar for logarithmic averages.

Let $\tilde{f}:=f\cdot n^{-it}$.   Then our assumptions imply that $\tilde{f}\sim \chi$
and  either $p\mid q$ or  $\tilde{f}(p^s)\neq \chi(p^s)$ for some $s\in \N$.
Hence, if $(N_k')$ is a subsequence of $(N_k)$ along which the Furstenberg system $(X,\tilde{\mu},T)$ of $\tilde{f}$ for Ces\`aro averages is defined,  the case treated in Section~\ref{SSS:spectrumreverse3} implies that $1/p\in \spec(X,\tilde{\mu},T)$.

Since  $\tilde{f}:=f\cdot n^{-it}$, the system
$(X,\tilde{\mu},T)$ is a factor of a joining of the system $(X,\mu,T)$ and  the Furstenberg system of  $n^{-it}$. The first system is  an   ergodic procyclic system by part~\eqref{I:StructurePretentious1} of Theorem~\ref{T:StructurePretentiousRefined} and the second is isomorphic to an identity transformation in $\T$ by \cite[Corollary~5.5]{GLR21} . Hence, the two systems are disjoint, and since the second system has trivial spectrum,  it follows that  $\spec(X,\tilde{\mu},T)$ is contained in $\spec(X,\mu,T)$. Since $1/p\in
\spec(X,\tilde{\mu},T)$, we deduce that $1/p\in  \spec(X,\mu,T)$, completing the proof.

\subsection{Proof of Theorem~\ref{T:ExactSpectrum}}\label{SS:ExactSpectrum}
Since $f\sim \chi$
for some primitive Dirichlet character $\chi$ with conductor $q$,   Proposition~\ref{P:ceslog'} implies that  Furstenberg systems of $f$ for  Ces\`aro and logarithmic averages coincide. We
are going to work with logarithmic averages in order to have
 part~\eqref{I:SpectrumDivisible1} of Corollary~\ref{C:SpectrumDivisible}   available to us.

 The inclusion $\spec(X,\mu,T)\subset \Lambda$ follows from  the following stronger fact that applies to general multiplicative functions that pretend to be Dirichlet characters.

\smallskip

{\bf Claim.}  {\em  Let $f\colon \N\to \U$ be a multiplicative function such that $f\sim \chi$ for some primitive  Dirichlet character $\chi$ with conductor $q$. If
 $$
 A:=\{p\in \P\colon \text{either } p\mid q \text{ or }f(p^s)\neq \chi(p^s) \text{ for some } s\in \N\}
 $$
	and $\Lambda$ is the subgroup generated by $\{1/p^s\colon p\in A, s\in \N\}$,
	then $\spec(X,\mu,T)\subset \Lambda$.}

\smallskip

Let us see how we prove the claim. Since  by the first part
$(X,\mu,T)$ is an ergodic procyclic system,  $\spec(X,\mu,T)$ is a subgroup of the rationals in $\T$.
 Any such subgroup is generated by elements of the form $1/p^s$ for  $p\in \P$ and $s\in \N$ (see, for example, \cite{Sc65}). By Theorem~\ref{T:spectrum}, if $p\not \in A$, then $1/p\not\in\spec(X,\mu,T)$, hence $1/p^s\not\in \spec(X,\mu,T)$ for every $s\in\N$. Combining the above facts, we deduce that  $\spec(X,\mu,T)$ is contained in the subgroup generated by elements of the form $1/p^s$ for $p\in A$, $s\in \N$, that is, $\spec(X,\mu,T)\subset \Lambda$. This proves the claim.

So, it remains to establish the inclusion $\Lambda\subset \spec(X,\mu,T)$.
 Since $f\sim \chi$  and for $p\in A$ we  either have $p\mid q$ or $f(p)\neq \chi(p)$,  Theorem~\ref{T:spectrum}
implies that $1/p\in \spec(X,\mu,T)$ for every $p\in A$.  Using this,
our  assumption that  $f(p) \neq 0$, and the fact that $f$ is completely multiplicative, we deduce from part~\eqref{I:SpectrumDivisible1} of Corollary~\ref{C:SpectrumDivisible} 
  that $1/p^s\in \spec(X,\mu,T)$ for every $p\in A$ and $s\in\N$.  By
  part~\eqref{I:StructurePretentious1} of Theorem~\ref{T:StructurePretentiousRefined}, the system
  $(X,\mu,T)$ is ergodic, hence its spectrum is a subgroup of $\T$. Combining the above, we get
   $\Lambda\subset \spec(X,\mu,T)$, completing the proof.

 \subsection{Proof of Theorem~\ref{T:ceslogzero}}\label{SS:ceslogzero}
 By Proposition~\ref{P:spectrum}, it suffices to show that under the stated assumptions, we have  $\alpha \not\in \spec(X,\mu,T)$ for every Furstenberg system $(X,\mu,T)$ of $f$ for Ces\`aro averages.  In case~\eqref{I:alpha1}, this follows from the fact that $\alpha$ is irrational and  $\spec(X,\mu,T)$ contains only rational values, which is an immediate consequence of  Theorem~\ref{T:StructurePretentious}. In case~\eqref{I:alpha2}, this follows  from Theorem~\ref{T:spectrum}. 

\subsection{Proof of Theorem~\ref{T:ceslogconverge}}\label{SS:ceslogconverge}
We use Lemma~\ref{L:ceslog}, in its more general version mentioned on the remark
after the lemma, and  follow the notation there.
	Note that since $\sum_{j=1}^\ell k_j=0$,  in \eqref{E:alpha12'} we have $t'=0$ and  $\prod_{j=1}^\ell(e(\alpha))^{k_j}=1$. The result then follows from the fact that a sequence $a\colon \N\to \U$ is convergent if  any two subsequences of $a$ have further subsequences whose difference converges to $0$.
 \subsection{Proof of Theorem~\ref{T:Sarnak}}\label{SS:Sarnak}
Suppose that $a\colon \N\to \U$  satisfies the Sarnak conjecture for Ces\`aro averages along $(N_k)$,
and $b=a \cdot f$, where $f\colon \N\to \U$ is a  pretentious multiplicative function.
Let $w\colon \N \to \U$ be a completely deterministic  sequence along $(N_k)$.
By Theorem~\ref{T:StructurePretentious}, the sequence $f$ is completely deterministic. It follows from the remark
made after Definition~\ref{D:ergodicnotions}, that the sequence $f\cdot w$ is also completely deterministic along $(N_k)$. Hence, our assumption on $a$ gives that
$$
\lim_{k\to\infty}\E_{n\in[N_k]} \, a(n)\, f(n)\, w(n)=0.
$$
This shows that the sequence $a\cdot f$ satisfies the Sarnak conjecture for Ces\`aro averages along $(N_k)$
and completes the proof.
A similar argument works for logarithmic averages since Theorem~\ref{T:StructurePretentious}  also applies to this setting.

\subsection{Proof of Theorem~\ref{T:Chowla}}\label{SS:Chowla}
Suppose that $a\colon \N\to \{-1,1\}$  satisfies the Chowla-Elliott  conjecture for Ces\`aro averages along $(N_k)$,
and $b=a \cdot f$, where  $f\colon \N\to \U$ is a  pretentious multiplicative function.
It suffices to show that
$$
\lim_{k\to \infty} \E_{n\in[N_k]}\, (a\cdot f)^{\epsilon_1}(n+n_1)\cdots (a\cdot f)^{\epsilon_\ell}(n+n_\ell)=0
$$
for all $\ell\in \N$, distinct  $n_1,\ldots, n_\ell\in \N$, and $\epsilon_1,\ldots, \epsilon_\ell\in \{-1,1\}$ (recall our notation $f^{-1}:=\overline{f}$). If this fails, then for some  $\ell\in \N$, distinct $n_1,\ldots, n_\ell\in \Z_+$,
and $\epsilon_1,\ldots, \epsilon_\ell\in \{-1,1\}$, we have convergence to a nonzero constant, along a subsequence $(N_k')$ of $(N_k)$ on which the Furstenberg systems of $a$ and $f$ are well defined. Our  assumption implies that the Furstenberg system of $a$ along $(N_k')$ is Bernoulli\footnote{Crucially here we use that $a$ takes values in $\{-1,1\}$, if $a$ took values on the unit circle we would have to assume that \eqref{E:CE} holds for all  $\epsilon_1,\ldots, \epsilon_\ell\in \Z$ not all of them $0$.}   and by Theorem~\ref{T:StructurePretentious} the Furstenberg system of $f$  along $(N_k')$ has zero entropy. Hence, the two systems are disjoint. This implies that the correlation along $(N_k')$ of $a\cdot f$ that we assumed to be non-zero,   is equal to the product of the individual correlations of $a$ and $f$, hence it is  zero since $a$ satisfies the Chowla-Elliott conjecture along $(N_k')$. This is a contradiction. 	A similar arguments works for logarithmic averages since Theorem~\ref{T:StructurePretentious}  also applies to this setting.

 	 \section{Furstenberg systems of MRT functions - Ces\`aro averages}\label{S:MRTCesaro}
 In this section we will prove structural results for Furstenberg systems of MRT multiplicative functions (see Definition~\ref{D:MRT}) when these systems are defined using Ces\`aro averages. In particular we will prove Theorems~\ref{T:StructureMRTCesaro2} and \ref{T:StructureMRTCesaro1}.

\smallskip

{\em  Throughout this section and the next one we use the convention $e(t):=e^{it}$ for $t\in \R$.}

 	 \subsection{Correlations of MRT functions  for  Ces\`aro averages}\label{SS:Nmsm1}
 	 We will use the following result from
 	 \cite[Lemmas~3.2-3.3]{GLR21} to rewrite correlations of an MRT  multiplicative function in a convenient form.
 	 \begin{lemma}\label{L:MRTreduction}
 	 	Let $f\colon \N\to \S^1$ be an MRT multiplicative function. Following the terminology in Definition~\ref{D:MRT}, let
 	 	 $N_m\to\infty $  with $N_m\leq t_{m+1}$, $m\in\N$. Then
 	 	$$
 	 	\lim_{m\to\infty}   \frac{|n\in [N_m]\colon |f(n)-n^{is_{m+1}}|>t_m^{-1}|}{N_m}=0.
 	 	$$
 	 	We also have the same conclusion  when we consider the logarithmic density of the set $\{n\in [N_m]\colon |f(n)-n^{is_{m+1}}|>t_m^{-1}\}$.
 	 \end{lemma}
  \begin{remark}
  	The argument was given in  \cite{GLR21} for the standard density, but it also applies for the logarithmic density.
  \end{remark}
 	Let  $(N_m)$ be a strictly increasing sequence of integers with $t_m\leq N_m\leq t_{m+1}$.  In the  computation of  the correlations of the function $f$ on the interval $[N_m]$,  the previous lemma  allows us to replace $f(n)$ with $n^{i s_{m+1}}=e(s_{m+1}\log{n})$.
 	 We deduce  that
 	 \begin{equation}\label{E:corrfCes}
 	 	\lim_{m\to \infty}	\E_{n\in[N_m]}\prod_{j=1}^\ell f^{k_j}(n+n_j)=
 	 	\lim_{m\to \infty} 	\E_{n\in[N_m]}\, e\big(s_{m+1}\big(\sum_{j=1}^\ell k_j \log(n+n_j)\big)\big)
 	 \end{equation}
holds   for all $\ell\in\N$, $k_1,\ldots, k_\ell, n_1,\ldots, n_\ell \in \Z$.
 Depending on how we choose $N_m$ in relation with $s_{m+1}$ we get different Furstenberg systems. In the following subsections we will use this formula
   to compute the correlations of all MRT multiplicative functions for various choices of $N_m$ and determine the structure of their corresponding Furstenberg systems for Ces\`aro averages. This will lead to a proof of Theorems~\ref{T:StructureMRTCesaro2} and \ref{T:StructureMRTCesaro1}.
 	
 	\subsubsection{The case  $N_m:=\lfloor\alpha s_{m+1}^{1/d}\rfloor$}
 	The goal of this subsection is to compute the correlations of an MRT function when
 	we average over intervals     $[N_m]$ that satisfy   $N_m=\lfloor\alpha s_{m+1}^{1/d}\rfloor$  for some $d\in \Z_+$ and $\alpha>0$. This is the context of the next result.
 		\begin{proposition}\label{P:corrlog}
 		Let $\alpha>0$, $d\in \N$.
 		For every $\ell\in\N$, $k_1,\ldots, k_\ell, n_1,\ldots, n_\ell \in \Z$,
 		let  $i_0$ be the minimum $i\in \Z_+$
 		such that $\sum_{j=1}^\ell k_j n_j^i\neq  0$ and $i_0:=+\infty$ if no such   $i$ exists.
 		Then 	
 		$$
 			\lim_{m\to \infty}	\E_{n\in[\alpha s_{m+1}^{1/d}]}\prod_{j=1}^\ell f^{k_j}(n+n_j)=
 		\begin{cases}
 			0, & \quad \text{if }\,  0\leq i_0< d\\
 			\int_0^1 (G_{\alpha,d}(x))^{\sum_{j=1}^\ell k_jn_j^d} \, dx, & \quad \text{if } \, i_0\geq d
 		\end{cases}
 		$$
 		where
 		$$
 		G_{\alpha,d}(x):=e(1/(\alpha^d x^d)) \quad  \text{for } x\in (0,1).
 		$$
 	\end{proposition}
 	\begin{remark}
 	Note that the function defined by $x\mapsto e(1/(\alpha x^d))=e^{i /(\alpha^d x^d)}$ for $x\neq0$, and, say, $0\mapsto 0$, is Riemann integrable in $[0,1]$, since it is bounded and Riemann integrable on $[a,1]$ for every $a\in [0,1]$.
 \end{remark}
 	Proposition~\ref{P:corrlog} is an immediate consequence of \eqref{E:corrfCes} and the next result.
 		\begin{lemma}\label{L:corrlog}
   Let $\alpha>0$, $d\in \N$.
 For every $\ell\in\N$, $k_1,\ldots, k_\ell, n_1,\ldots, n_\ell \in \Z$,
 let  $i_0$ be the minimum $i\in \Z_+$
 such that $\sum_{j=1}^\ell k_j n_j^i\neq  0$ and $i_0:=+\infty$ if no such   $i$ exists.
 Then 	
 $$
 	\lim_{N\to \infty} 	\E_{n\in[\alpha N^{1/d}]}\, e\big(N\big(\sum_{j=1}^\ell k_j \log(n+n_j)\big)\big)=
 \begin{cases}
 	0, & \quad \text{if }\,  0\leq i_0< d\\
 	\int_0^1 (G_{\alpha,d}(x))^{\sum_{j=1}^\ell k_jn_j^d} \, dx, & \quad \text{if } \, i_0\geq d
 \end{cases}
 $$
 where
 $$
 G_{\alpha,d}(x):=e(1/(\alpha^d x^d)) \quad  \text{for } x\in (0,1).
 $$
 	\end{lemma}
 	\begin{proof}
 		Let
 		$$
 		A:=\lim_{N\to \infty} 	\E_{n\in[\alpha N^{1/d}]}\, e\big(N\big(\sum_{j=1}^\ell k_j \log(n+n_j)\big)\big).
 		$$
 		For convenience, we approximate $\log(n+a)$ with $\log n$ plus a polynomial.  Since
 		$$
 		\Big| \log(n+a)-\log n- \sum_{i=1}^d (-1)^{i-1} \frac{a^i} {n^i}\Big|\leq \frac{C_a}{M_N^{d+1}} \,  \text{ for all } n\geq  M_N,
 		$$
 		we have
 		$$
 		\lim_{N\to\infty} \max_{n\in[N^c,N^{1/d}]} N\Big| \log(n+a)-\log n- \sum_{i=1}^d (-1)^{i-1} \frac{a^i} {n^i}\Big|= 0
 		$$
 		for every $c\in \R_+$ such that $1/(d+1)<c<1/d$.
 		 		So in order to compute $A$ we can replace $\log(n+n_j)$ with
 		$\log n- \sum_{i=1}^d (-1)^{i-1} \frac{n_j^i} {n^i}$ throughout. Hence,
 		$$
 		A=\lim_{N\to \infty} 	\E_{n\in[\alpha N^{1/d}]}\, e\big(N\big(\sum_{j=1}^\ell k_j \log n+\sum_{i=1}^d (-1)^{i-1} \sum_{j=1}^\ell k_jn_j^i \, n^{-i} \big)\big).
 		$$
 		
 		Suppose first that $i_0<d$. Then
 		  $\sum_{j=1}^\ell k_jn_j^{i_0}\neq 0$  and  by combining   Corollary~\ref{C:basicestimate}  (it applies for this $i_0$ since $N^{1/d}\prec N^{1/i_0}$) and Lemma~6.3, we get that $A=0$.
 		
 		Suppose now that   $i_0\geq d$. Then  $\sum_{j=1}^\ell k_jn_j^i= 0$ for $i=0,\ldots, d-1$,  and
 		$$
 		A=\lim_{N\to \infty} 	\E_{n\in[\alpha N^{1/d}]}\, e\big(K N /n^d\big),
 		$$
 		where
 		\begin{equation}\label{E:K}
 		K:=\sum_{j=1}^\ell k_jn_j^d.
 		\end{equation}
 	Then
 		$$
 		A=\lim_{N\to \infty} 	\E_{n\in[\alpha N^{1/d}]}\, e\big(K \alpha^{-d} (\lfloor \alpha N^{1/d}\rfloor +\epsilon(N))^d /n^d \big),
 		$$
 		where $\epsilon(N)\in \{0,1\}$.  Note that  the error we make by replacing
 		$(\lfloor \alpha N^{1/d}\rfloor +\epsilon(N))^d$ with $(\lfloor \alpha N^{1/d}\rfloor)^d$
 		is bounded by  $C N^{(d-1)/d}$ for some $C>0$, and this is much smaller than  $n^d$
 		for $n\in [N^c,N^{1/d}]$ whenever  $1/d-1/d^2<c<1/d$. We  deduce  that $\epsilon(N)$ can be ignored in the computation of the above limit without affecting the value of $A$. Hence,
 			$$
 		A=\lim_{N\to \infty} 	\E_{n\in[\alpha N^{1/d}]}\, e\big(K \alpha^{-d} \lfloor \alpha N^{1/d}\rfloor^d /n^d \big),
 		$$
 		where $K$ is as in \eqref{E:K}. More conveniently,
 			$$
 		A=\lim_{N\to \infty} 	\E_{n\in [N]}\, e\big(K \alpha^{-d} N^d/n^d \big),
 		$$
 		assuming that the last limit exists, which is something we shall prove shortly.
 		The last limit  can be rewritten as  a limit of  Riemann sums
 		$$
 		\lim_{N\to \infty} 	\E_{n\in[N]}\, (G_{\alpha, d}(n/N))^K
 		$$
 		where $G_{\alpha,d}\colon (0,1]\to \S^1$ is the Riemann integrable function in the statement.  Hence,
 		$$
 		A=\int_0^1 (G_{\alpha,d}(x))^K \, dx.
 		$$			
This completes the proof.
 	\end{proof}

 		\subsubsection{The case  $s_{m+1}^{1/(d+1)}\prec N_m\prec s_{m+1}^{1/d}$.}
 		The goal of this subsection is to compute the correlations of an MRT function when
we average over intervals     $[N_m]$ that satisfy $s_{m+1}^{1/(d+1)}\prec N_m\prec s_{m+1}^{1/d}$ for some $d\in \Z_+$. This is the context of the next result.
	\begin{proposition}\label{P:corrlog2}
	Let  $d\in \Z_+$ be fixed. If $d=0$ suppose that
	$s_{m+1}\prec N_m\leq t_{m+1}$ ,  and  if $d\in \N$ suppose that $s_{m+1}^{1/(d+1)}\prec N_m\prec s_{m+1}^{1/d}$.
	For every $\ell\in\N$, $k_1,\ldots, k_\ell, n_1,\ldots, n_\ell \in \Z$,
	let  $i_0$ be the minimum $i\in \Z_+$
	such that $\sum_{j=1}^\ell k_j n_j^i\neq  0$ and $i_0:=+\infty$ if no such   $i$ exists.
	Then 	
	$$
		\lim_{m\to \infty}	\E_{n\in[N_m]}\prod_{j=1}^\ell f^{k_j}(n+n_j)=
	\begin{cases}
		0, & \quad \text{if }\,  0\leq i_0\leq d\\
		1, & \quad \text{if } \, i_0>d.
	\end{cases}
	$$
\end{proposition}
	Proposition~\ref{P:corrlog2} is an immediate consequence of \eqref{E:corrfCes} and the next result.
 	\begin{lemma}\label{L:corrlog2}
 		 Let  $d\in \Z_+$ be fixed. Suppose that the sequence $(L_N)$ of positive integers
 		 satisfies
 		 $N\prec L_N$ if $d=0$, and  $N^{1/(d+1)}\prec L_N\prec N^{1/d}$ if $d\in \N$.
 For every $\ell\in\N$, $k_1,\ldots, k_\ell, n_1,\ldots, n_\ell \in \Z$,
 let  $i_0$ be the minimum $i\in \Z_+$
 such that $\sum_{j=1}^\ell k_j n_j^i\neq  0$ and $i_0:=+\infty$ if no such   $i$ exists.
 Then 	
 $$
\lim_{N\to \infty} 	\E_{n\in[L_N]}\, e\big(N\big(\sum_{j=1}^\ell k_j \log(n+n_j)\big)\big)=
 \begin{cases}
 	0, & \quad \text{if }\,  0\leq i_0\leq d\\
 	1, & \quad \text{if } \, i_0> d.
 \end{cases}
 $$
 \end{lemma}
 \begin{proof}
 	Let $d\in \Z_+$ and $A$ be the limit we aim to compute. As in the proof of  	Lemma~\ref{L:corrlog}
 	 we start by approximating $\log(n+a)$ with $\log n$ plus a polynomial.  Since
 	$$
 	\Big| \log(n+a)-\log n- \sum_{i=1}^d (-1)^{i-1} \frac{a^i} {n^i}\Big|\leq \frac{C_a}{M_N^{d+1}} \, \text{ for all } n\geq  M_N,
 	$$
 	and $N^{1/(d+1)}\prec L_N$, we have
 	$$
 	\lim_{N\to\infty} \max_{n\in[cL_N,L_N]} N	\Big| \log(n+a)-\log n- \sum_{i=1}^d (-1)^{i-1} \frac{a^i} {n^i}\Big|= 0
 	$$
 	for every positive $c\in (0,1)$.
 	So using Lemma~\ref{L:cN}, we get that in order to compute $A$, we can replace $\log(n+n_j)$ with
 	$\log n- \sum_{i=1}^d (-1)^{i-1} \frac{n_j^i} {n^i}$ throughout.\footnote{This is a trivial but crucial reduction that does not work for logarithmic averages and leads to different correlations and Furstenberg systems.}  Hence,
 	\begin{equation}\label{E:ALN}
 	A=\lim_{N\to \infty} 	\E_{n\in[L_N]}\, e\big(N\big(\sum_{j=1}^\ell k_j \log n+\sum_{i=1}^d (-1)^{i-1} \sum_{j=1}^\ell k_jn_j^i \, n^{-i} \big)\big).
 	\end{equation}
 	We also get a similar identity with $d+1$ in place of $d$.
 	
 	Suppose first  that  $i_0\leq d$. We have $\sum_{j=1}^\ell k_jn_j^{i_0}\neq 0$  and we get by combining  identity \eqref{E:ALN} with  Corollary~\ref{C:basicestimate}  (it applies for this $i_0$ since $L_N\prec N^{1/d}$, hence  $L_N\prec N^{1/i_0}$) and Lemma~6.3, that $A=0$.
 	
 	Suppose now that  $i_0>d$. Then    $\sum_{j=1}^\ell k_jn_j^i= 0$ for $i=0,\ldots, d$. In this case, using identity \eqref{E:ALN} with $d+1$ in place of $d$  we have
 		$$
 	A=\lim_{N\to \infty} 	\E_{n\in[L_N]}\, e\big(  K \, N /n^{d+1}\big),
 	$$
 	where 
 	$K:=(-1)^d\, \sum_{j=1}^\ell k_jn_j^{d+1}$.
 	Our growth assumption $ N^{1/(d+1)}\prec L_N$ implies that  for all    $c\in (0,1)$ we have
 	$$
\lim_{N\to\infty} 	\max_{n\in [cL_N,L_N]} (N/n^{d+1})=0.
 	$$
 	 It follows that
 		$$
 	\lim_{N\to \infty} 	\E_{n\in[cL_N,L_N]}\, e\big( K \, N /n^{d+1}\big)=1
 	$$
 	for all  $c\in (0,1)$.  Hence, $A=1$ by Lemma~\ref{L:cN}.\footnote{This last conclusion fails badly for logarithmic averages, which is  the reason why similar computations for logarithmic averages lead to very different expressions.} This completes the proof.
 \end{proof}

 \subsection{Ergodic models of MRT functions for Ces\`aro correlations}\label{SS:ergmod1}
 Having computed the correlations of MRT functions for certain ranges of the parameter $N_m$ we would like to identify simple measure preserving systems and functions that reproduce these correlations. When we manage to do this, it will be an easy matter to show isomorphism between this system and the Furstenberg system of the MRT function.
 \subsubsection{The case  $N_m:=\lfloor\alpha s_{m+1}^{1/d}\rfloor$.}
 The next lemma identifies systems and functions that reproduce the correlations of MRT functions when $N_m:=\lfloor\alpha s_{m+1}^{1/d}\rfloor$.
 \begin{lemma}\label{L:corrF}
 		For  $\alpha>0$ and $d\in \Z_+$, let
 	$(\T^{d+1}, S_{\alpha,d},m_{\T^{d+1}})$ be the  system given in Definition~\ref{D:Sad}.
 	 	Let also $F\colon \T^{d+1}\to \S^1$ be defined by
 	$$
 	F(x_0,\ldots, x_d):=e(x_d).
 	$$
 For every $\ell\in\N$, $k_1,\ldots, k_\ell, n_1,\ldots, n_\ell \in \Z$,
 let  $i_0$ be the minimum $i\in \Z_+$
 such that $\sum_{j=1}^\ell k_j n_j^i\neq  0$ and $i_0:=+\infty$ if no such   $i$ exists.
 Then 	
 \begin{equation}\label{E:Gad}
 \int \prod_{j=1}^\ell S_{\alpha,d}^{n_j}F^{k_j}\, dm_{\T^{d+1}}=
 \begin{cases}
 	0, & \quad \text{if }\,  0\leq i_0<d\\
 	\int_0^1 (G_{\alpha,d}(x))^{\sum_{j=1}^\ell k_jn_j^d} \, dx, & \quad \text{if } \, i_0\geq d,
 \end{cases}
 \end{equation}
 where
 $$
 G_{\alpha,d}(x):=e(1/(\alpha^d x^d)) \quad  \text{for } x\in (0,1).
 $$
 \end{lemma}
\begin{proof}
For $x=(x_0,\ldots, x_d)\in \T^{d+1}$ direct computation gives that
$$
F(S_{\alpha,d}^nx)= e\Big(\sum_{i=0}^{d-1}\binom{n}{i} x_{d-i}+ \binom{n}{d} \, g_{\alpha,d}( x_{0})\Big), \qquad n\in\N.
$$
Hence, for $x\in \T^{d+1}$ we have
$$
\prod_{j=1}^\ell F^{k_j}(S_{\alpha,d}^{n_j}x)=e\Big(\sum_{i=0}^{d-1} c_i \cdot x_{d-i}\Big)\cdot
(G_{\alpha,d}(x_0))^{c_d},
$$
where
$$
c_i:=\sum_{j=1}^\ell k_j \binom{n_j}{i}, \quad i=0,\ldots, d.
$$

Note that $c_i=0$ for $i=0,\ldots, d-1$ if and only if
$\sum_{j=1}^\ell k_jn_j^i=0$ for $i=0,\ldots, d-1$,  equivalently when $i_0\geq d$, and in this case we have $c_d=\sum_{j=1}^\ell k_jn_j^d$.
It follows that  the integral in \eqref{E:Gad}  is
equal to  	$\int_0^1 (G_{\alpha,d}(x))^{\sum_{j=1}^\ell k_j\binom{n_j}{d}}\,dx$   if  $i_0\geq d$,
and is equal to  $0$ otherwise, that is, when $i_0<d$.
This completes the proof.
\end{proof}

 \subsubsection{The case  $s_{m+1}^{1/(d+1)}\prec N_m\prec s_{m+1}^{1/d}$.}
 Repeating the computation in the proof of Lemma~\ref{L:corrF} with the function $x_0$ in place  of $g_{\alpha,d}(x_0)$, gives the following result.
 \begin{lemma}\label{L:corrFint}
 	For $d\in \Z_+$ let
 $(\T^{d+1}, S_d,m_{\T^{d+1}})$ be the level $d$ unipotent system given in Definition~\ref{D:Sk}. Let also
 	 $F\colon \T^{d+1}\to \S^1$ be defined by
 	$$
 	F(x_0,x_1,\ldots, x_d):=e(x_d).
 	$$
  	For every $\ell\in\N$, $k_1,\ldots, k_\ell, n_1,\ldots, n_\ell \in \Z$,
 	let  $i_0$ be the minimum $i\in \Z_+$
 	such that $\sum_{j=1}^\ell k_j n_j^i\neq  0$ and $i_0:=+\infty$ if no such   $i$ exists.
Then 	
 	$$
 	\int \prod_{j=1}^\ell S_d^{n_j}F^{k_j}\, dm_{\T^{d+1}}=
 	\begin{cases}
 		0, & \quad \text{if }\,  0\leq i_0\leq d\\
 		1, & \quad \text{if } \, i_0>d.
 	\end{cases}
 	$$
 \end{lemma}

 \subsection{Furstenberg systems  of MRT functions for  Ces\`aro averages} \label{SS:MRTCesaroProofs}
 We are now ready to prove our main results regarding Furstenberg systems of MRT functions for Ces\`aro averages.
 \begin{proof}[Proof of Theorem~\ref{T:StructureMRTCesaro1}]
We first show that the system	$(X,\mu_{\alpha,d},T)$ is isomorphic to the system   $(\T^{d+1},m_{\T^{d+1}},S_{\alpha,d})$.
 	
Recall that $X=\U^\Z$. Let $F_0\colon X\to \S^1$ be the $0^{\text{th}}$-coordinate projection and $F\colon \T^{d+1}\to \S^1$ be as in Lemma~\ref{L:corrF}.
Combining   Proposition~\ref{P:corrlog} with  Lemma~\ref{L:corrF}, we deduce that
$$
\lim_{m\to \infty} 	\E_{n\in[\alpha s_{m+1}^{1/d}]}\prod_{j=1}^\ell f^{k_j}(n+n_j)=\int \prod_{j=1}^\ell S_{\alpha,d}^{n_j}F^{k_j}\, dm_{\T^{d+1}}.
$$
Hence,
\begin{equation}\label{E:corrF0}
	\int \prod_{j=1}^\ell T^{n_j}F_0^{k_j}\, d\mu_{\alpha,d}=\int \prod_{j=1}^\ell S_{\alpha,d}^{n_j}F^{k_j}\, dm_{\T^{d+1}}
\end{equation}
holds for  every $\ell\in\N$, $k_1,\ldots, k_\ell, n_1,\ldots, n_\ell \in \Z$.

Next, we define the map $\Phi\colon \T^{d+1}\to X$ by
$$
\Phi(y):=(F(S_{\alpha,d}^ny))_{n\in\Z}, \quad y\in \T^{d+1}.
$$
We clearly have $\Phi\circ S_{\alpha,d}=T\circ \Phi$. Moreover, a direct computation shows that the map $\Phi$ is injective. Lastly, since $F=F_0\circ \Phi$, identity \eqref{E:corrF0} implies that
$$
	\int \prod_{j=1}^\ell T^{n_j}F_0^{k_j}\, d\mu_{\alpha,d}=\int \prod_{j=1}^\ell T^{n_j}F_0^{k_j}\, d(m_{\T^{d+1}}\circ \Phi^{-1}).
$$
Hence,  $\mu_{\alpha,d}=m_{\T^{d+1}}\circ \Phi^{-1}$, since a  linearly dense subset of $C(X)$ has the same integral with respect to the two measures (see second remark after Definition~\ref{D:sequencespace}). This establishes the asserted isomorphism.

Since the system $(\T^{d+1},m_{\T^{d+1}},S_{\alpha,d})$  has trivial rational spectrum,
 so does  the system $(X,\mu_{\alpha,d},T)$.


To prove that the system $(X,\mu_{\alpha,d},T)$ is    not   strongly stationary,  it suffices to show that
$$
\int \overline{F_0} \cdot TF_0\, d\mu\neq 	\int \overline{F_0} \cdot T^rF_0\, d\mu
$$
for some $r\in \N$. Equivalently,    by applying  \eqref{E:corrF0} (for $\ell=2$, $n_1=0$, $n_2=r$, $k_1=-1$, $k_2=1$) and \eqref{E:Gad},  it suffices to show that for fixed $\alpha,d$ we have
$$
\int_0^1 e(1 / (\alpha^d x^d))\, dx\neq \int_0^1 e(r  /(\alpha^d x^d))\, dx
$$
for some non-zero $r\in \Z$. Direct computation shows that this is the case, for example when $\alpha=d=1$ and $r=2$.
\end{proof}

\begin{proof}[Proof of  Theorem~\ref{T:StructureMRTCesaro2}]
We combine   Proposition~\ref{P:corrlog2} with Lemma~\ref{L:corrFint},  and argue  as in the proof of Theorem~\ref{T:StructureMRTCesaro1}.
\end{proof}

 		 \section{Furstenberg systems of MRT functions - Logarithmic averages}\label{S:MRTlogarithmic}
 		 In this section we will prove structural results for Furstenberg systems of MRT multiplicative functions (see Definition~\ref{D:MRT}) when these systems are defined using logarithmic averages. In particular, we will prove Theorem~\ref{T:StructureMRTLogarithmic}. In the case of logarithmic averages
 	the correlations of  MRT functions turn out to be different than those for Ces\`aro averages, and this leads to substantially different Furstenberg systems.

 \subsection{Correlations of MRT functions  for logarithmic averages} \label{SS:corrlog}
 The main goal of this subsection is to give in Proposition~\ref{P:corrlogc} an explicit description of the correlations for logarithmic averages of MRT functions when we average over sequences of intervals that grow as fractional powers of $s_{m+1}$, a notion that we define next.
  Again, using Lemma~\ref{L:MRTreduction} our starting point is the following identity
 \begin{equation}\label{E:corrflog}
 	\lim_{m\to \infty} 	\lE_{n\in[N_m]}\prod_{j=1}^\ell f^{k_j}(n+n_j)=
 	\lim_{m\to \infty} 	\lE_{n\in[N_m]}\, e\big(s_{m+1}\big(\sum_{j=1}^\ell k_j \log(n+n_j)\big)\big),
 \end{equation}
 which holds for all $\ell\in\N$, $k_1,\ldots, k_\ell, n_1,\ldots, n_\ell \in \Z$.
 As in the case of Ces\`aro averages,  depending on how we choose $N_m$ in relation with $s_{m+1}$ we get different Furstenberg systems, and the next notion will help us identify which sequences $N_m$ give rise to which systems.

 \begin{definition}\label{D:fracdeg}
 	We define the {\em fractional degree}  of a  sequence of positive real numbers $L_N\to +\infty$ to be
 	$$
 	\lim_{N\to\infty} \frac{\log{L_N}}{\log{N}}
 	$$
 	if the limit exists (could be $+\infty$).
\end{definition}
\begin{remarks}
	$\bullet$ Note that
	$$
	\lim_{N\to\infty} \frac{\log{L_N}}{\log{N}}=	\begin{cases}
		0, \quad &  \text{if }\,    L_N\prec N^\varepsilon \text{ for every } \varepsilon>0\\
		c\in (0,+\infty),\quad &   \text{if }\,   N^{c-\varepsilon}\prec  L_N\prec N^{c+\varepsilon} \text{ for every } \varepsilon>0\\
		+\infty, \quad &\text{if }\,  N^{d}\prec L_N \text{ for every } d\in \N.
	\end{cases}
	$$
	
	$\bullet$ If $c\in (0,+\infty)$, then the sequence $(N^c)$ has fractional  degree $c$.
More generally, 	the same  holds for the sequence $(N^c(\log{N})^b+L'_N)$ for every $b\in\R$ and every $(L'_N)$ with  $L'_N\prec N^c$.
\end{remarks}

\begin{proposition}\label{P:corrlogc}
	Let $N_m:=s_{m+1}^{1/c}$   for some $c>0$.
	For every $\ell\in\N$, $k_1,\ldots, k_\ell$, $n_1,\ldots, n_\ell \in \Z$,
	let  $i_0$ be the minimum $i\in \Z_+$
	such that $\sum_{j=1}^\ell k_j n_j^i\neq  0$ and $i_0:=+\infty$ if no such   $i$ exists.
	Then
	$$
	\lim_{m\to \infty} 	\lE_{n\in[N_m]}\prod_{j=1}^\ell f^{k_j}(n+n_j)=
	\begin{cases}
		0, & \quad \text{if }\,  0\leq i_0\leq c\\
		1-c/i_0, & \quad   \text{if }\,  c<i_0<+\infty  \\
		1, & \quad \text{if } \, i_0=+\infty.
	\end{cases}
	$$
	Furthermore, the same identity  holds if $(L_N)$ is a sequence of positive real numbers with fractional degree $1/c$ and $N_m:=L_{s_{m+1}}$.
\end{proposition}
\begin{remark}
	By taking $c:=1/2$ (or any other $c<1$)  we get the following
	non-vanishing property of the $2$-point correlations of $f$
	$$
	\lim_{m\to \infty} 	\lE_{n\in[s_{m+1}^2]}\ \overline{f(n)}\cdot f(n+h)=\frac{1}{2}
	$$
	for every non-zero  $h\in \Z$.
\end{remark}
Proposition~\ref{P:corrlogc} is an immediate consequence of  identity \eqref{E:corrflog} and the next result.
  \begin{lemma}\label{L:corrlogc}
	Let $(L_N)$ be a sequence of positive real numbers with fractional  degree $1/c$ for some $c\in (0,+\infty)$.
		For every $\ell\in\N$, $k_1,\ldots, k_\ell, n_1,\ldots, n_\ell \in \Z$,
	let  $i_0$ be the minimum $i\in \Z_+$
	such that $\sum_{j=1}^\ell k_j n_j^i\neq  0$ and $i_0:=+\infty$ if no such   $i$ exists.
	Then
	$$
	\lim_{N\to \infty} 	\lE_{n\in[L_N]}\, e\big(N\big(\sum_{j=1}^\ell k_j \log(n+n_j)\big)\big)=
	\begin{cases}
		0, & \quad \text{if }\,  0\leq i_0\leq c\\
		1-c/i_0, & \quad   \text{if }\,  c<i_0<+\infty  \\
		1, & \quad \text{if } \, i_0=+\infty.
	\end{cases}
	$$
\end{lemma}
  \begin{remarks}
  	  	$\bullet$ 	
  For example, if  $A$ is the value of the limit on the left hand side, then for  $c\in [1,2)$  (in which case $L_N$ has fractional degree in $(1/2,1]$) we have  $A=0$ if   either $\sum_{j=1}^\ell k_j\neq 0$ or  $\sum_{j=1}^\ell k_jn_j\neq 0$ (since $i_0\leq 1\leq c$); we have   $A=1-c/2$ if
  	 $\sum_{j=1}^\ell k_j= \sum_{j=1}^\ell k_jn_j=0$ and $\sum_{j=1}^\ell k_j n_j^2\ \neq 0$  (since $c<i_0=2$), and so on.

  	$\bullet$ Note that the limit $A$ remains the same if we replace $n_1,\ldots, n_\ell$ with
  	$rn_1,\ldots, rn_\ell$,  for every $r\in\N$ and $\ell\in\N$, $n_1\ldots, n_\ell\in\Z$. Hence,  the Furstenberg system
  	of the multiplicative function that has such correlations is strongly stationary. This is unlike the case of Ces\`aro averages, where  for $L_N:=N^{1/k}$, $k\in\N$, we  got that the corresponding correlations given in Lemma~\ref{L:corrF} were not dilation invariant.
  \end{remarks}
 \begin{proof}
 	In order to have a specific example in mind, in the course of the proof the reader may find it convenient to assume that $L_N=N^{1/c}$ where $c>0$. We denote by $A$ the limit of the averages we want to compute.

	We shall use the following basic estimate that follows by applying the Taylor-Lagrange theorem for   the function
	$\log{x}$: For every $a\in \R_+$,  $d\in \N$, and $M_N >0$,  we have
	\begin{equation}\label{E:Taylor}
		\Big| \log(n+a)-\log n- \sum_{i=1}^d (-1)^{i-1} \frac{a^i} {n^i}\Big|\leq \frac{C_a}{M_N^{d+1}}\,  \text{ for all } n\geq  M_N.
	\end{equation}
	The utility of this approximation is that it  enables us to connect various exponential sums that appear below to those treated in Corollary~\ref{C:basicestimatelog}.
	\medskip
	
	{\bf Case 1 ($i_0\leq c$).}	If   $i_0\leq c$, we claim that   $A=0$.
	Let $\varepsilon>0$. We also take $\varepsilon<1$.
	We apply \eqref{E:Taylor}  for $d_\varepsilon:=[1/\varepsilon]$,   $a:=n_j$, $j=1,\ldots, \ell$,  and use that
	$\lim_{N\to\infty} (N / M_N^{d_\varepsilon+1})= 0$ for $M_N:=N^{\varepsilon}$. We get that
	$$
	\lim_{N\to \infty} \max_{n\in [N^{\varepsilon},L_N]}  N\Big| \sum_{j=1}^\ell k_j \log(n+n_j)- \sum_{j=1}^\ell k_j\, \log{n}+ \sum_{i=1}^{d_\varepsilon} (-1)^{i-1} \sum_{j=1}^\ell k_j n_j^i\,  n^{-i}\Big|=0.
	$$
	Using this and the estimate  $ \lE_{n\in[N]}\, {\bf 1}_{[1,N^\varepsilon]} (n)\leq \varepsilon$, we get that for
	$$
	A_{\varepsilon}:=	\lim_{N\to\infty} \lE_{n\in[L_N]}\,  e\big(N\big( \sum_{j=1}^\ell k_j\, \log{n}+ \sum_{i=1}^{d_\varepsilon} (-1)^{i-1} \sum_{j=1}^\ell k_j n_j^i\,  n^{-i}\big)\big),
	$$	
	we have
	$$
	|A-A_{\varepsilon}|\leq \varepsilon.
	$$
	So in order to show that $L=0$, it suffices to show that for every $\varepsilon>0$ we have $A_{\varepsilon}=0$.
	Equivalently, it suffices to show that (for every $\varepsilon>0$)
	$$
	\lim_{N\to\infty} \lE_{n\in[L_N]} \,  e\big(N\big( \sum_{j=1}^\ell k_j\, \log{n}+ \sum_{i=1}^{d_\varepsilon} (-1)^{i-1} \sum_{j=1}^\ell k_j n_j^i\,  n^{-i}\big)\big)=0.
	$$	
	Recall that the defining property of $i_0$ implies that
	$\sum_{j=1}^\ell k_j n_j^i=0$  for $i=0,\ldots,i_0-1$. Since $c\geq i_0$  and the fractional degree of $L_N$ is $1/c$ (so positive in particular), we have  $N^\gamma\prec L_N\prec N^{\frac{1}{i_0}+\varepsilon}$ for some $\gamma>0$ and every $\varepsilon>0$.  Hence, Corollary~\ref{C:basicestimatelog} applies and gives  $A=0$.
	
	\medskip

	{\bf Case 2 $(i_0>c)$.}  	  We claim that  $A=1-c/i_0$ if $i_0$ is finite
	and $L=1$ if $i_0=+\infty$.  Suppose first that $i_0$ is finite.
	
	
	We  decompose
	$$
	A=A_1+A_2,
	$$
	where
	$$
	A_1:=	\lim_{N\to\infty} \lE_{n\in[L_N]}\, {\bf 1}_{[1,N^{1/i_0}]} (n)\,  e\big(N\big(\sum_{j=1}^\ell k_j \log(n+n_j)\big)\big)
	$$	
	and
	$$
	A_2:=\lim_{N\to\infty}	\lE_{n\in[N^{1/c}]}\, {\bf 1}_{[N^{1/i_0},L_N]} (n)\,  e\big(N\big(\sum_{j=1}^\ell k_j \log(n+n_j)\big)\big).
	$$			
(Note that since $L_N$ has fractional degree $1/c>1/i_0$ we have $N^{1/i_0}\prec L_N$ and the logarithmic averages over $[L_N]$ and $[N^{1/c}]$ coincide.)
	
	\medskip
	
	{\bf Case 2a (Computation of $A_1$).}  	We claim that  $A_1=0$.  Indeed, we have
	$$
	A_1=c\cdot i_0^{-1}\cdot 	\lim_{N\to\infty} \lE_{n\in[N^{1/i_0}]}\,   e\big(N\big(\sum_{j=1}^\ell k_j \log(n+n_j)\big)\big)
	$$	
	and
	by Case 1 the last limit is $0$.
	
	\medskip
	
	{\bf Case 2b (Computation of $A_2$).}
	It remains to show that  $A_2=1-c/i_0$.
	For $\varepsilon>0$ such that $1/i_0+\varepsilon<1/c$, let
	\begin{equation}\label{E:A2e}
	A_{2,\varepsilon}:=\lim_{N\to\infty}	\lE_{n\in[N^{1/c}]}\, {\bf 1}_{[N^{1/i_0+\varepsilon},N^{1/c}]} (n)\,  e\big(N\big(\sum_{j=1}^\ell k_j \log(n+n_j)\big)\big).
	\end{equation}
	
	Since $A_2=\lim_{\varepsilon\to 0^+}A_{2,\varepsilon}$, it suffices to compute $A_{2,\varepsilon}$ for these  values of $\varepsilon>0$.

	We apply \eqref{E:Taylor}  for $d:=i_0-1$ (which is $\geq 0$ since $i_0\geq 1$),   $a=n_j$, $j=1,\ldots, \ell$, and $M_N:=N^{1/i_0+\varepsilon}$ for $\varepsilon$ small enough so that $M_N\prec L_N$. 
	Note that since
	$\lim_{N\to\infty} (N / M_N^{d+1})= 0$,   we have
	$$
	\lim_{N\to\infty} \max_{n \in [N^{1/i_0+\varepsilon},L_N]} 	N	\Big| \log(n+a)-\log n- \sum_{i=1}^{i_0-1} (-1)^{i-1} \frac{a^i} {n^i}\Big|=0.
	$$
	(If $i_0=1$ the sum over $i$ is empty.)
	For $j=1,\ldots, \ell$, we apply these identities for $a:=n_j$,  multiply  them by $k_j$, and add them up.  We deduce using the defining property of $i_0$
	($\sum_{j=1}^\ell k_j n_j^i=0$  for $i=0,\ldots,i_0-1$),    that
	$$
	\lim_{N\to\infty} \max_{n\in [N^{1/i_0+\varepsilon},L_N]} N	\Big| \sum_{j=1}^\ell k_j \log(n+n_j)\Big|= 0.
	$$
	Hence,
	$$
	\lim_{N\to\infty} \max_{n\in [N^{1/i_0+\varepsilon},L_N]} \Big|e\big(N \sum_{j=1}^\ell k_j \log(n+n_j)\big)-1\big|= 0.
	$$
If we combine this with \eqref{E:A2e} we get  (we use here that the fractional degree of $L_N$ is $1/c$ and the fact that we use logarithmic averages)
	$$
	A_{2,\varepsilon}=\lim_{N\to\infty}	\lE_{n\in[L_N]}\, {\bf 1}_{[N^{1/i_0+\varepsilon},L_N]} (n)=1-c/i_0-c \varepsilon.
	$$
	Hence,  $A_2=\lim_{\varepsilon\to 0^+} A_{2,\varepsilon}=1-c/i_0$. This completes the proof of the Case 2 when $i_0$ is finite.
	
	If $i_0=+\infty$ (in this case $A=A_2$), then the previous argument  gives for every $i_0\in \N$ the following identity
	$$
	\lim_{N\to\infty}	\lE_{n\in[L_N]}\, {\bf 1}_{[N^{1/j_0},L_N]} (n)\,  e\big(N\big(\sum_{j=1}^\ell k_j \log(n+n_j)\big)\big)=1-c/j_0.
	$$
 Letting $j_0\to +\infty$  gives that $A=A_2=1$, as required.
\end{proof}


	 		 \subsection{Ergodic models of  MRT functions for logarithmic correlations} 		
	 As was the case in Section~\ref{SS:ergmod1}, our goal is to
	 		  identify simple measure preserving systems and functions that reproduce the correlations
	 		  in the number theory setting obtained in Proposition~\ref{P:corrlogc}.  With a bit of guesswork we
get that the systems in 	 		  Definition~\ref{D:muc} help us do the job; in fact the next lemma motivated their  definition.
	 		 \begin{lemma}\label{L:corrF'}
	 		 	Let $c>0$. For $d\geq \lfloor c \rfloor$ let  $(Y_d,\nu_d,S_d)$ be the level $d$ unipotent system given in  Definition~\ref{D:Sk}
	 		 	 and
	 		 	$F_d\colon \T^{d+1}\to \S^1$ be defined by
	 		 	$$
	 		 	F_d(x_0,\ldots, x_d):=e(x_d).
	 		 	$$
	 		 Let also the  system
	 		 	$(Z_c,\nu'_c, R_c)$ be as in Definition~\ref{D:muc} and
	 		  $G_c\in L^\infty(\nu'_c)$ be  defined by
	 		 	$$
	 		 	G_c:=\sum_{d= \lfloor c\rfloor }^\infty   {\bf 1}_{Y_d}\cdot F_d.
	 		 	$$
	 		 	Lastly, for every $\ell\in\N$, $k_1,\ldots, k_\ell, n_1,\ldots, n_\ell \in \Z$,
	 let  $i_0$ be the minimum $i\in \Z_+$
	 such that $\sum_{j=1}^\ell k_j n_j^i\neq  0$ and $i_0:=+\infty$ if no such   $i$ exists.
	 Then
	 		 	$$
	 		 	\int \prod_{j=1}^\ell R_c^{n_j}G_c^{k_j}\, d\nu'_c=\begin{cases}
	 		 		0, & \quad \text{if }\,  0\leq i_0\leq c\\
	 		 		1-c/i_0, & \quad   \text{if }\,  c<i_0<+\infty  \\
	 		 		1, & \quad \text{if } \, i_0=+\infty.
	 		 	\end{cases}
	 		 	$$
	 		 \end{lemma}
 		 \begin{proof}
We have that
$$ 	
 		 \int \prod_{j=1}^\ell R_c^{n_j}G_c^{k_j}\, d\nu'_c
 		=
 		 \Big(1-\frac{c}{\lceil c\rceil}\Big) C_{\lfloor c\rfloor }  +
 		 c\sum_{d=\lceil c\rceil}^\infty  \Big(\frac{1}{d}-\frac{1}{d+1}\Big)\,  C_d,
 $$
 	where for $d\in\Z_+$ we let
 	 $$
 	 	  C_d:=\int \prod_{j=1}^\ell S_d^{n_j}F_d^{k_j}\, d\nu_d.
 	 	  $$
 By  Lemma~\ref{L:corrFint} we have
$$
C_d 	 	 = \begin{cases}
 	 	  	0, & \quad \text{if }\,  0\leq i_0\leq d\\
 	 	  	1, & \quad \text{if } \, i_0>d.
 	 	  \end{cases}
 	 $$		
The  asserted identity follows by combining the previous identities and simple direct computation.
  \end{proof}
	 			 \subsection{Furstenberg systems  of MRT functions for logarithmic  averages}
	 			 Combining the previous results it is now easy to prove our main result regarding the structure of Furstenberg systems of MRT functions for logarithmic averages.
	 \begin{proof}[Proof of Theorem~\ref{T:StructureMRTLogarithmic}]
	 		We combine    Proposition~\ref{P:corrlogc} with  Lemma~\ref{L:corrF'}, and argue
 as in the proof of Theorem~\ref{T:StructureMRTCesaro1}.
\end{proof}



 \appendix

 \section{Exponential sum estimates}\label{A:estimates}
In  this appendix we  gather some simple facts and exponential sum estimates used in the proofs of the results regarding the structure of Furstenberg systems of MRT multiplicative functions.
 	\subsection{Ces\`aro averages} The next result is an immediate consequence of  \cite[Theorem~2.9]{GK91}. 	
 	
 	 \begin{theorem}[Kuzmin-van der Corput]\label{T:VDC}
 		Let $q\geq 2$ be an integer and $c\in (0,1)$. Suppose that $C_1,C_2$ are constants  (depending  on $q$ and $c$ only)  and  $h\in C^{q}([1,+\infty))$ be a function, such that for some   $L,M>0$, we have
 		$$
 		C_1 \, M \, L^{-r}\leq \max_{x\in [cL,L]} |h^{(r)}(x)| \leq C_2 \, M \, L^{-r}
 		$$
 		for $r=1,\ldots, q$. Then there exists a positive constant $C_3$, depending only on $c$, $C_1$, $C_2$ (and not on $L,M$),
 		such that 
 		$$
 		\big|\E_{n\in  [cL,L]} \, e(h(n))\big|\leq C_3((M /L^{q})^{1/Q}+1/M),
 		$$
 		where $Q:=4\cdot 2^{q-2}-2$.	 	
 	\end{theorem}
 We are only going to use the following consequence that applies to a special class of functions $g$
  appearing in our arguments.
 	\begin{corollary}\label{C:VDC}
 	Let  $g(x):=c_0\, \log x +\sum_{i=1}^kc_i x^{-i}$ for $x\in \R_+$. Let also
 	$i_0$ be the minimum   $i\in \{0,\ldots, k\}$ such that $c_i\neq 0$.
 	Then for every $q\geq 2$  and $c\in (0,1)$, there exists a constant $C_{c,q,g}$ such that the following holds: 	
 	 If
 	  $L>0$  is  large enough (depending only on  $c,q,g$),   then
 	  for every $N>0$ we have
 	$$
 	\big|\E_{n\in [cL,L]} \, e(N\,  g(n))\big|\leq C_{c,q,g}\big((N /L^{q+i_0})^{1/Q}+L^{i_0}/N\big),
 	$$
 	where $Q:=4\cdot 2^{q-2}-2$.	 	
 \end{corollary}
 	\begin{proof}
 		Let $N>0$,   $c>0$, and  $h_N(x):=Ng(x)$.
 		We have
 		$$
 		h_N^{(r)}(x)=N \big( \sum_{i=0}^k\, c_i\, c_{r,i} \, x^{-i-r}\big)
 		$$ for every $r\in \N$ for some non-zero constants $c_{0,r},\ldots,c_{k,r}$.
 		Bounding the derivatives of the function inside the parenthesis (which does not depend on $N$), we get  that there exist $L_0\in \N$ and  $C_1, C_2>0$, depending only on $c,g,r$ but not on $N$,  such that for   all $L\geq L_0$ we have (we also use the defining property of $i_0$ here)
 		$$
 		C_1\, N L^{-i_0-r}	\leq \max_{x\in [cL,L]} |h_N^{(r)}(x)|\leq C_2\, N L^{-i_0-r}
 		$$
 		for $r=1,\ldots, q$.
 		Using Theorem~\ref{T:VDC} for  $h:=h_N$ and $M:=NL^{-i_0}$, we get that the asserted estimate holds for all $L\geq L_0$.
 	\end{proof}
 	
 	\begin{lemma}\label{L:cN}
 		Let $L_N\to \infty$ be a sequence of positive integers. For $N\in\N$ let  $a_N\colon [L_N]\to \U$ be  finite sequences such that  for every  small enough $c>0$ we have
 		$$
 		\lim_{N\to\infty}  \E_{n\in [cL_N,L_N]} \ a_N(n)=0.
 		$$
 		Then
 		$$
 		\lim_{N\to\infty}\E_{n\in[L_N]} \, a_N(n)=0.
 		$$
 	\end{lemma}
 	\begin{remark}
 		We crucially  use here that $\frac{1}{L_N}\sum_{n\in [1,cL_N]}|a_N(n)|\leq c$ and this converges to $0$ as $c\to 0^+$.  An analogous property fails for logarithmic averages, so this convergence criterion cannot be used  for logarithmic averages.
 	\end{remark}
 	Combining the above,  we get the following qualitative result that we will use in subsequent sections.
 	\begin{corollary}\label{C:basicestimate}	
 		Let  $g(x)=c_0\, \log x +\sum_{i=1}^kc_i x^{-i}$ be non-zero
 		and $
 		i_0$ be the minimum   $i\in \{0,\ldots, k\}$ such that $c_i\neq 0$.
 		Let also   $(L_N)$ be a sequence of positive real numbers that satisfies $N^{\gamma}\prec L_N\prec N^{1/i_0} $ for some $\gamma>0$.
 		Then 	
 		$$
 		\lim_{N\to\infty} \, 	\E_{n\in [L_N]}\,  e(N\, g(n))=0.
 		$$
 	\end{corollary}
 \begin{remark}
 	If $i_0=0$, then  our only growth assumption is $N^{\gamma}\prec L_N$ for some $\gamma>0$.
 \end{remark}
 	\begin{proof}
 		Let $c\in (0,1)$. Choose $q\geq 2$ such that $(q+i_0)\gamma>1$.  Using Corollary~\ref{C:VDC} and our growth assumption on $L_N$, we get that
 		$$
 		\lim_{N\to\infty} \, 	\E_{n\in [cL_N,L_N]}\,  e(N\, g(n))=0.
 		$$
 		We deduce from this using Lemma~\ref{L:cN} that the asserted convergence to $0$ holds.
 	\end{proof}
 	
 		\subsection{Logarithmic averages}

 The next lemma will enable us to deduce convergence results for logarithmic averages
 of variable sequences from  corresponding results  for Ces\`aro averages once we have  some additional uniformity.
 \begin{lemma}\label{L:ceslogapp}
For $N\in \N$,  let   $L_N$ be a sequence of positive integers that satisfies  $N^\gamma\prec L_N$ for some $\gamma>0$, and let $a_N\colon [L_N]\to \U$ be  finite sequences.
 	 Suppose that
 	 for all small enough $c,\delta>0$ we have
  	 \begin{equation}\label{E:cdelta}
  	\lim_{N\to\infty} \sup_{n\in [L_N^\delta, L_N^{1-\delta}]}\big| \E_{k\in[cn,n]}\, a_N(k)\big|=0.
  \end{equation}
 	 Then
 	 \begin{equation}\label{E:logaN}
 	 \lim_{N\to\infty} \lE_{n\in[L_N]}\, a_N(n)=0.
 	 \end{equation}
 \end{lemma}
\begin{proof}
	We first note that our assumption  \eqref{E:cdelta}  implies
 \begin{equation}\label{E:cdelta'}
		\lim_{N\to\infty} \sup_{n\in [L_N^\delta, L_N^{1-\delta}]}\big| \E_{k\in[n]}\, a_N(k)\big|=0.
\end{equation}
	Indeed, one immediately verifies that  the difference between the two expressions is
	bounded by  a function of $c$  that  converges to $0$ as $c\to 0^+$.
	
	Note also that in order to show \eqref{E:logaN}, it suffices to show that for all small enough $\delta>0$ we have
	 \begin{equation}\label{E:dd}
	\lim_{N\to\infty} \lE_{n\in[L_N^\delta,L_N^{1-\delta}]}\, a_N(n)=0.
	\end{equation}
	Indeed, because of the logarithmic averaging, the difference between this limit  and the one in \eqref{E:logaN} is bounded by   a function of $\delta$ that converges to $0$ as $\delta\to 0^+$.  
	
	Now let   $ \varepsilon>0$ and $\delta>0$ be small enough. Equation  \eqref{E:cdelta'} implies that there exists $N_0$ such that for $N>N_0$ we have
	 \begin{equation}\label{E:delta}
		\sup_{n\in [L_N^\delta, L_N^{1-\delta}]}\big| \E_{k\in [n]}\, a_N(k)\big|\leq \varepsilon.
	\end{equation}
Using partial summation, we get  for $N\geq N_0$ that
$$
\Big|\sum_{n\in [L_N^\delta,L_N^{1-\delta}]} \frac{a_N(n)}{n} \Big|\leq 2+ \sum_{n\in [L_N^\delta,L_N^{1-\delta}]}\Big| \frac{1}{n^2}\sum_{k=1}^n\, a_N(k)\Big|\leq2+  \sum_{n\in [L_N^\delta,L_N^{1-\delta}]}\frac{\varepsilon}{n},
$$
where the last estimate follows from \eqref{E:delta}.  Note also that since $N^\gamma\prec L_N$, we have $\lim_{N\to\infty}\sum_{n\in [L_N^\delta,L_N^{1-\delta}]}\, \frac{1}{n}=+\infty$. We deduce that
$$
\limsup_{N\to\infty}\big| \lE_{n\in[L_N^\delta,L_N^{1-\delta}]}\, a_N(n)\big|\leq \varepsilon.
$$
Since $\varepsilon$ is arbitrary,
this implies \eqref{E:dd} and completes the proof.
\end{proof}
 	\begin{corollary}\label{C:basicestimatelog}	
 	Let  $g(x):=c_0\, \log x +\sum_{i=1}^kc_i x^{-i}$ be non-zero
 	and $
 	i_0$ be the minimum   $i\in \{0,\ldots, k\}$ such that $c_i\neq 0$.
 	Let also   $(L_N)$ be a sequence of positive real numbers that    satisfies
 	$N^{\gamma}\prec L_N\prec  N^{\frac{1}{i_0}+\varepsilon} $
 	 for some $\gamma>0$ and for every $\varepsilon>0$.
 	Then 	
 	$$
 	\lim_{N\to\infty} \, 	\lE_{n\in [L_N]}\,  e(N\, g(n))=0.
 	$$
 \end{corollary}
\begin{remarks}
$\bullet$
	If $i_0=0$, then  our only growth assumption is $N^{\gamma}\prec L_N$ for some $\gamma>0$.

$\bullet$  If $c\in (0,1/i_0]$, then   the assumptions are satisfied when $L_N:=N^c$  or $L_N:=N^c\log N$. If $i_0=0$,  then the assumptions are satisfied when  $L_N=N^c$ for all $c\in (0,+\infty)$.
 Note that by Corollary~\ref{C:basicestimate},  if we use  Ces\`aro averages, then a similar result holds for  $c\in (0,1/i_0)$ but not for  $c=1/i_0$.
 For example, although  $\lim_{N\to\infty} \lE_{n\in[N]}\, e(N/n)=0$, we have $\lim_{N\to\infty} \E_{n\in[N]}\, e(N/n)\neq 0$.
\end{remarks}
 \begin{proof}
 	Let $c\in (0,1)$ and $\delta>0$ be arbitrary. Let $\gamma>0$ be such that  $N^{\gamma}\prec L_N$. Choose $q\geq 2$ such that $(q+i_0)\gamma \delta>1$ and $\varepsilon>0$ such that $(1+\varepsilon) (1-\delta)<1$.  Using Corollary~\ref{C:VDC} and our growth assumption on $L_N$, we get that
 	 there exists a constant $C_{c,q}$  such that for  all large enough $N$ we have
 	$$
 	\sup_{L_N^\delta\leq n\leq L_N^{1-\delta}}\big |\E_{k\in [cn,n]} \, e(Ng(n))\big| \leq
 	C_{c,q}\big(N /N^{ (q+i_0)\gamma\delta})^{1/Q}+N^{(1+\varepsilon)(1-\delta)}/N\big),
 	$$
 	where $Q:=4\cdot 2^{q-2}-2$ if $i_0\geq 1$, and a similar estimate with $1$
 	in place of $N^{(1+\varepsilon)(1-\delta)}$ if $i_0=0$.
 	Since $(q+i_0)\gamma \delta>1$ and $(1+\varepsilon) (1-\delta)<1$,
 	we get
 	$$
\lim_{N\to\infty} 	\sup_{L_N^\delta\leq n\leq L_N^{1-\delta}}\big |\E_{k\in [cn,n]} \, e(Ng(n))\big| =0.
 	$$
 	Applying Lemma~\ref{L:ceslogapp}, we deduce
 	$$
 	\lim_{N\to\infty} \, 	\lE_{n\in [L_N]}\,  e(N\, g(n))=0
 	$$
 	completing the proof.
 \end{proof}

\section{Factor of a product Erg$\times$Id}
In  this appendix we prove an ergodic result that was used in the proof of part~\eqref{I:StructurePretentious2} of Theorem~\ref{T:StructurePretentiousRefined}.

Our setting is the following:
We consider  two systems $(Y,\nu,S)$ and $(Z,\lambda,\id)$, where
\begin{itemize}
	\item $Y:=\U^\Z$, $S\colon Y\to Y$ is the shift map, and $\nu$ is an ergodic $S$-invariant probability measure.
	
	\smallskip
	
	\item $Z:=\S^1$, and $\lambda$ is an arbitrary Borel probability measure on $\S^1$.
\end{itemize}
For $z\in\S^1$, we recall that $M_z\colon Y\to Y$ denotes the coordinatewise multiplication by $z$
\begin{equation}\label{E:MzA}
	(M_zy)(k):=(z\cdot y(k)), \quad k\in \Z.
\end{equation}
We define the  factor $(X,\mu,T)$ of the direct product $(Y,\nu,S)\times(Z,\lambda,\id)$ as follows:
$$
X:=\U^\Z \, \text{ and } \,  T\colon X\to X \, \text{ is also the shift map,}
$$
and  the factor map
$\pi\colon Y\times \S^1\to X$ is given by
\begin{equation}\label{E:pi}
	\pi(y,z):=M_z(y), \quad y\in Y,\ z\in \S^1,
\end{equation}
where $M_z$ is as in \eqref{E:MzA}, and
\begin{equation}\label{E:mu}
	\mu:=\pi_*(\nu\otimes\lambda)
\end{equation}
is the pushforward of the product measure $\nu\otimes\lambda$ by the factor map $\pi$.
Using the previous notation and assumptions, we have the following result.
\begin{proposition}
	\label{P:product_factor}
	Suppose that there exists $z\in\S^1$ such that the measure  $\nu$ is not invariant by $M_z$. Then there exists $r\in \N$,  such that $(X,\mu,T)$ is isomorphic to the direct product $(Y,\nu,S)\times(Z,\lambda_r,\id)$, where $\lambda_r$ is the pushforward of $\lambda$ by $z\mapsto z^r$.
\end{proposition}
\begin{proof}	
	Our plan is as follows: We  first define the positive integer  $r\in \N$  that appears in the conclusion of our statement and  then to each $x\in X$ we associate unique elements
	$\tilde{y}\in Y$ and $\tilde{z}\in Z$ such that $M_{\tilde{z}}(\tilde{y})=x$. Once this is done,  we prove that the map
	$\Phi\colon X\to Y\times Z$, defined by $\Phi(x):=(\tilde{y},\tilde{z}^r)$, is the required isomorphism.

	\smallskip
	
	{\em Definition of $r$.}   Let
	\begin{equation}\label{E:G}
		G := \{z\in Z\colon \nu\text{ is invariant by }M_z\}.
	\end{equation}
	Since $G$ is a closed subgroup of $\S^1$ and by assumption $G\neq \S^1$,
	there exists $r\in \N$ such that
	$$
	G=\{\zeta_r^k,\, k=0,\ldots, r-1\}
	$$
	where $\zeta_r$ is an order $r$ root of unity.
	\smallskip

	{\em Definition of $\tilde{y}(x)$ and  $\tilde{z}(x)$.}	 Using the ergodicity of the measure $\nu$ and the fact that the map $M_z$ preserves  $C(Y)$, we  deduce that for  $\nu$-almost all $y\in Y$ and  for all $z\in Z$  we have
	$$
	\lim_{N\to\infty} \E_{n\in[N]} \, \delta_{M_z S^n y} = (M_z)_\ast \nu,
	$$
	where $(M_z)_\ast \nu$ denotes the pushforward of $\nu$ by $M_z$. It follows
	from this and \eqref{E:pi}, \eqref{E:mu},
	that for $\mu$ almost every $x\in X$ the following limit exists
	$$
	\mu_x:=  \lim_{N\to\infty} \E_{n\in[N]} \, \delta_{T^nx},
	$$
	and for $\nu\otimes\lambda$ almost every $(y,z)\in Y\times Z$ we have $\mu_{M_z(y)} = (M_z)_\ast \nu$.
	Next, for $\mu$ almost every $x\in X$, we let
	$$
	H_x := \{w\in\S^1\colon (M_{w^{-1}})_\ast \mu_x = \nu\}.
	$$
	Then  for $\nu\otimes\lambda$ almost every $(y,z)\in Y\times Z$, using \eqref{E:pi} and \eqref{E:mu}, we get
	$$
	H_{M_z(y)} = \{w\in\S^1\colon (M_{w^{-1}z})_\ast \nu = \nu\} = zG.
	$$
	This establishes that for $\nu\otimes\lambda$ almost every $(y,z)\in Y\times Z$
	the set
	$zG$ is uniquely determined by  the element $x:=M_z(y)\in X$.
	Hence,  for $\mu$ almost every $x\in X$  we can define
	$$
	\tilde{z}(x):=  \text{  the unique element of } \ zG \  \text{ that has argument in } \ [0,2\pi/r).
	$$
	Let us justify that the map $x\mapsto \tilde{z}(x)$ is measurable. There is a Borel subset $X_0\subset X$ of full $\mu$ measure on which $x\mapsto \mu_x$ is Borel measurable, and such that for all $x\in X_0$, $\mu_x$ is of the form $(M_z)_\ast \nu$ for some $z\in\S^1$. Set
	$$
	\S^1_r := \{e(\theta):0\le\theta<2\pi/r\}\subset\S^1.
	$$
	Then the graph of $x\mapsto \tilde z(x)$, $x\in X_0$,  is the preimage of $\{\nu\}$ by the Borel map $(x,z)\mapsto (M_{z^{-1}})_\ast\mu_x$, $(x,z)\in X_0\times \S^1_r$. So this graph is Borel-measurable in $X_0\times \S^1_r$, and by~\cite[Theorem~14.12]{Ke12}, the map $x\mapsto\tilde z(x)$ is Borel measurable.
	%
		We also set, for $x\in X_0$,
	$$
	\tilde y(x):= M_{\tilde z(x)^{-1}} (x) \in Y.
	$$

	\smallskip
	
	{\em Definition of  $\Phi$.}  We define the map $\Phi\colon X_0\to Y\times Z$ as follows
	\begin{equation}\label{E:Phix}
		\Phi(x):=\bigl(\tilde{y}(x),\tilde{z}(x)^r\bigr),
	\end{equation}
	where $r\in \N$, $\tilde{y}(x)\in Y$,  $\tilde{z}(x)\in Z$, are  defined as above. The map $\Phi$ is measurable and we are going to show next that it establishes an isomorphism between the systems $(X,\mu,T)$ and $(Y,\nu,S)\times(Z,\lambda_r,\id)$.
	
	\smallskip

	{\em $\Phi$ is an isomorphism.}  We have to establish the following three claims.
	
	\smallskip
	
	{\em Claim 1: $\Phi$ is injective.}
	Note that for $\mu$ almost every $x\in X$ we have
	\begin{equation}\label{E:x}
		x=M_{\tilde z(x)}\bigl(\tilde y(x)\bigr).
	\end{equation}
	Replacing if necessary $X_0$ by a smaller Borel set (but still of full measure), we can assume that the above identity and \eqref{E:Phix} hold on $X_0$. Suppose that $\Phi(x)=\Phi(x')$ for some $x,x'\in X_0$.
	Then $\tilde{y}(x)=\tilde{y}(x')$ and $\tilde{z}^r(x)=\tilde{z}^r(x')$. Since both
	$\tilde{z}(x)$, $\tilde{z}(x')$ have argument in $[0,2\pi/r)$, we deduce that $\tilde{z}(x)=\tilde{z}(x')$.
	Then   \eqref{E:x} gives that $x=x'$, which  establishes that the map $\Phi$ is injective on $X_0$.

	\smallskip
	
	{\em Claim 2: We have $\Phi\circ T=(S\times \id)\circ\Phi$.}
	Equivalently, using \eqref{E:Phix}, it suffices to verify that for $\mu$ almost every $x\in X$ we have
	\begin{equation}\label{E:factormap}
		(\tilde{y}(Tx),\tilde{z}^r(Tx))=
		(S\tilde{y}(x), \tilde{z}^r(x)).
	\end{equation}
	Since  $\mu_{Tx}=\mu_x$, we have
	$$
	\tilde z(Tx)=\tilde z(x)
	$$
	and
	$$
	\tilde y(Tx)=M_{\tilde z(x)^{-1}} (Tx)=SM_{\tilde z(x)^{-1}} (x)=S\tilde y(x).
	$$
	Combining these two identities we get \eqref{E:factormap}.
	
	\smallskip
	
	{\em Claim 3: We have $\Phi_*\mu=\nu\otimes \lambda_r$.}
	Note first that  for $\nu\otimes\lambda$ almost every $(y,z)\in Y\times Z$ we have $\tilde z\left( M_z(y) \right)^{-1}=z^{-1}w$ for  some  $w\in G$, a  $r$-root of unity depending only on $z$, and
	$$
	\tilde y\bigl( M_z(y) \bigr) = M_{\tilde z\left( M_z(y) \right)^{-1}z} (y) = M_w(y).
	$$
	Since $w\in G$, we get by \eqref{E:G},  that for a fixed $z\in Z$, the pushforward of $\nu$ by $y\mapsto\tilde y\bigl( M_z(y) \bigr)$ is still $\nu$.
	This means that, for all  $f\in C(Y)$, we have for all $z\in Z$
	$$ \int_Y f\bigl(\tilde y\bigl(M_z(y)\bigr)\bigr)\, d\nu(y) = \int_Y f(y)\,d\nu(y).
	$$
	But then, integrating the above identity with respect to $z$ we get
	$$
	\int_X f\bigl(\tilde y(x)\bigl)\,d\mu(x)
	= \int_Z \left( \int_Y f\bigl(\tilde y\bigl(M_z(y)\bigr)\bigr)\, d\nu(y) \right)\,d\lambda(z) = \int_Y f(y)\,d\nu(y),
	$$
	which shows that the pushforward of $\mu$ by $x\mapsto \tilde y(x)$ is also $\nu$.
	
	Let us consider now the pushforward of $\mu$ by $x\mapsto \tilde{z}(x)^r$.
	Observe that, for $\nu\otimes\lambda$ almost every $(y,z)\in Y\times Z$, we have $\tilde z\bigl(M_z(y)\bigr)G=zG$, hence
	$$ \tilde z\bigl(M_z(y)\bigr)^r = z^r. $$
	Therefore, for an arbitrary continuous  $g\in C(\S^1)$, we have
	$$
	\int_X g\big(\tilde z(x)^r\big)\, d\mu(x) = \int_{Y\times Z} g\
	\bigl(\tilde z\bigl(M_z(y)\bigr)^r\bigr)\,d(\nu\otimes\lambda) =  \int_Z g\bigl( z^r\bigr)\, d\lambda(z) = \int_{Z} g(z)\, d\lambda_r(z).
	$$
	We conclude that the pushforward of $\mu$ by the map  $x\mapsto \tilde{z}^r(x)$ is $\lambda_r$.

	
	Combining the above, we have  that 	
	the measure $\Phi_*\mu$  is a joining of $(Y,\nu,S)$ and $(Z,\lambda_r,\id)$.
	Since, by assumption,  the system $(Y,\nu,S)$ is ergodic and the system $(Z,\lambda_r,\id)$ is an identity system, the two systems are disjoint. Hence,
	$\Phi_*\mu=\nu\otimes \lambda_r$, as required.
\end{proof}
Our argument gives that  $r$ depends only on $\nu$
and the factor generated by the map $x\mapsto \tilde z(x)$ from $(X,\mu,T)$ to $(Z,\lambda_r, \id)$ is non-trivial (in which case $(X,\mu,T)$ is non-ergodic)
if and only if $\lambda$ is not concentrated on the group of $r$-roots of unity. This is always  the case if $\lambda$ is a continuous measure.

\section{Ergodic consequences of a result of Klurman}\label{A:Klurman}
In this last  appendix we justify some remarks made immediately after Theorem~\ref{T:StructurePretentious} and Conjecture~\ref{Con:3}.

 The next  result
follows immediately from an argument given by Klurman~ \cite{Kl17}
and  crucially uses results  about 2-point correlations of  Tao~\cite{Tao15}. We sketch the argument for completeness.
\begin{proposition}{\cite[Proof of Lemma~4.3]{Kl17}}\label{P:Klurman}
	Let $f\colon \N\to \U$ be an aperiodic multiplicative function. Then  there exist
$N_k\to\infty$ such that
\begin{equation}\label{E:Klurman}
	\lim_{k\to \infty} \lE_{n\in [N_{k}]} \, \overline{f(n)}\cdot f(n+j)=0 \quad \text{ for every } j\in \N.
\end{equation}
\end{proposition}
\begin{remarks}
		$\bullet$
	As shown in \cite[Theorem~B.1]{MRT15},  some aperiodic multiplicative functions may have non-vanishing $2$-point correlations (see also the remark following Proposition~\ref{P:corrlogc}).
	
		$\bullet$ More elaborate arguments of \cite{KMT23}  give stronger results, see for example  Theorem~1.2 therein.
\end{remarks}
\begin{proof}
	We first claim that  for every $m\in \N$ there exists
	a subsequence $(N_{m,k})_{k\in\N}$  of $(N_k)_{k\in\N}$ (which also depends on $f$) such that
	$$
	\lim_{k\to \infty} \lE_{n\in [N_{m,k}]} \, \overline{f(n)}\cdot f(n+j)=0 \quad \text{ for } j=1,\ldots, m.
	$$ 	
	Indeed, 		suppose that the conclusion fails. Then there exist $m\in \N$,  $\varepsilon>0$, and $N_0=N_0(m,\varepsilon)$,  such that for all  $N\geq N_0$  we have
	$$
	|\lE_{n\in [N]} \, \overline{f(n)}\cdot f(n+j_N)|\geq \varepsilon
	$$
	for some $j_N\in [m]$. This fact, combined with the argument used to prove the second part of \cite[Lemma~4.3]{Kl17}
	gives, without any change, that $f$ is pretentious, a contradiction.
	
	Using the previous claim, we get  that for every $k\in \N$ there exists $N_k\in \N$
	such that
	$$
		|\lE_{n\in [N_k]} \, \overline{f(n)}\cdot f(n+j)|\leq 1/k \quad \text{ for } j=1,\ldots, k.
$$
	Furthermore, we can assume that the sequence $(N_k)_{k\in \N}$ is strictly increasing.  The result follows.
\end{proof}
\begin{corollary}\label{C:LebesgueComponent}
	If  $f\colon \N\to \U$  is a non-trivial aperiodic multiplicative function, then at least one of its Furstenberg systems  for logarithmic averages has  a Lebesgue component. In fact,  if  $F_0$ is the $0^{\text{th}}$-coordinate projection, on some Furstenberg system  $(X,\mu,T)$ of $f$ we have
	\begin{equation}\label{E:2pointergodic'}
		\int \overline{F_0}\cdot T^jF_0\, d\mu=0 \quad \text{ for every }  j\in \N.
	\end{equation}
\end{corollary}
\begin{proof}
	Using \eqref{E:correspondence} we see that $F_0\neq 0$ ($\mu$-a.e.) is equivalent to $\E_{n\in \N} \, |f(n)|^2\neq 0$, which in turn is equivalent to the  non-triviality assumption for $f$.
	
	Let $(N_k)$ be the subsequence given by Proposition~\ref{P:Klurman}. Pick a subsequence $(N_k')$ along which  $f$ admits correlations for logarithmic averages and let $(X,\mu,T)$ be the corresponding Furstenberg system of $f$. Then \eqref{E:correspondence} and \eqref{E:Klurman}    imply that \eqref{E:2pointergodic'} holds.
\end{proof}
By combining Corollary~\ref{C:LebesgueComponent} with the main argument  in \cite{FH21}
we deduce the following
subsequential variant of the logarithmically averaged Sarnak conjecture for ergodic weights that applies to all aperiodic multiplicative functions (for multiplicative functions  that satisfy  stronger aperiodicity properties \cite[Theorem~1.1]{FH21} covers a more general result).
\begin{theorem}\label{T:SarnakErgodic}
	Let   $f\colon \N\to \U$ be an aperiodic multiplicative function. Then there exists a subsequence $N_k\to\infty$ such that    for every deterministic, ergodic sequence $w\colon \N\to \U$, we have
	$$
	\lim_{k\to\infty} \lE_{n\in [N_k]} \, f(n)\, w(n)=0.
	$$
\end{theorem}
\begin{remark}
	As remarked in \cite[Section~5.2]{GLR21},
	there exists an aperiodic multiplicative function $f\colon \N\to \S^1$ and a non-ergodic deterministic sequence $w\colon \N\to \U$ such that  the averages
	$ \lE_{n\in [N]} \, f(n)\, w(n)$ do not converge to $0$ as $N\to\infty$.
	So passing to a subsequence is needed if one wants to allow $w$ to be  an arbitrary deterministic weight sequence (not necessarily ergodic). On the other hand, see Conjectures~\ref{Con:2} and \ref{Con:3} for variants that use relaxed assumptions.
\end{remark}
To prove Theorem~\ref{T:SarnakErgodic}, one works with the subsequence $N_k\to\infty$ provided by Corollary~\ref{C:LebesgueComponent}  and repeats the proof of \cite[Theorem~1.1]{FH21} given in Section~5.1 in \cite{FH21}.  The point is that
Corollary~\ref{C:LebesgueComponent} serves as a substitute for the 2-point correlation result of Tao~\cite{Tao15} (the latter is not applicable to all  aperiodic
multiplicative functions).

\end{document}